\documentclass[reqno,11pt]{amsart}

\usepackage{amsmath, amssymb, amsthm, graphicx}

\usepackage{color}

\newtheorem{lemma}{Lemma}[section]
\newtheorem{theorem}{Theorem}[section]

\theoremstyle{definition}
\newtheorem{definition}{Definition}[section]
\theoremstyle{remark}
\newtheorem{remark}{Remark}[section]

\numberwithin{equation}{section}

\newcommand{\p}{\partial}
\newcommand{\norm}[1]{\left\Vert#1\right\Vert}
\newcommand{\dd}{\mathrm{d}}

\newcommand{\di}{\mathrm{div}}
\newcommand{\grad}{\mathrm{grad}}

\newcommand{\vol}{\mathrm{vol}}
\makeatletter
\newcommand{\rmnum}[1]{\mathrm{\romannumeral #1}}
\newcommand{\Rmnum}[1]{\mathrm{\expandafter\@slowromancap\romannumeral#1@}}
\makeatother
\begin{document}
\title[Three-Dimensional Steady Compressible Euler System]{Decomposition of Three-Dimensional Steady Non-isentropic Compressible Euler System and Stability of Spherically Symmetric Subsonic Flows and Transonic Shocks under Multidimensional Perturbations}

\author{Li  Liu}
\address{L. Liu, Shanghai University of International Business and Economics, Shanghai 201620, China} \email{llbaihe@gmail.com}

\author{Gang Xu}
\address{G. Xu, Faculty of Science, Jiangsu University, Zhenjiang, Jiangsu 212013, China} \email{gxu@ujs.edu.cn}

\author{Hairong Yuan}
\address{H. Yuan (corresponding author), Department of Mathematics, %and Shanghai Key Laboratory of Pure Mathematics and Mathematical Practice,
East China Normal University, Shanghai
200241, China} \email{hryuan@math.ecnu.edu.cn}


\date{\today}

\begin{abstract}
We develop a method that works in general product Riemannian manifold to decompose the three-dimensional steady full compressible Euler system, which is of elliptic-hyperbolic composite-mixed type for subsonic flows. The method is applied to show stability of spherically symmetric subsonic flows and  transonic shocks in space $\mathbb{R}^3$ under multidimensional perturbations of  boundary conditions.
\end{abstract}
\maketitle

%%%----------------------------------------------------------------------------
\tableofcontents
\section{Introduction}\label{sec1}

We are interested in the three-dimensional steady non-isentropic compressible Euler system that governing the stationary motion of subsonic flows of perfect gases.  Let $p, \rho$, and $s$ represent the pressure, density of mass, and entropy of the flow respectively. To be specific, we consider polytropic gases, namely the constitutive relation is given by $p=A(s)\rho^\gamma$, with $\gamma>1$ being the adiabatic exponent, and $A(s)=k_0\exp(s/c_\nu)$, for two positive constants $k_0$ and $c_\nu$. The sonic speed is $c=\sqrt{\gamma p/\rho}.$ Let $u$ be the velocity of fluid flow, which is a vector field whose integral curves are called fluid trajectories. The flow is called {\it subsonic} if $|u|<c$, and {\it supersonic} if $|u|>c$. Now, for `$\di$' and `$\grad$' being respectively the divergence and gradient operator, set
\begin{eqnarray}
\varphi&\triangleq&\di (\rho u\otimes u)+\grad\, p-\rho b,\label{eq101}\\
\varphi_1&\triangleq&\di (\rho u),\label{eq102}\\
\varphi_2&\triangleq&\di (\rho E u)-\rho b\cdot u-\rho r.\label{eq103}
\end{eqnarray}
Here  $E\triangleq\frac{1}{2}|u|^2+\frac{c^2}{\gamma-1}$  is the so called  {\it Bernoulli constant} (\cite[p.22]{CF}),
which is the total energy of the fluid per unit mass, and $b$ is the exterior body force vector acting on the fluid (per unit mass), and $r$ is the heat supply (per unit mass and unit time). %Note that $b\cdot u$ means the inner product of $b$ and $u$.
Then the steady, three-dimensional full Euler system for compressible fluids with exterior  force $b$ and heat supply $r$ reads ({\it cf.} %\cite[(7.08.02), (7.09.2), (8.02.2) and (9.06) in pp.14-17]{CF} or
\cite[(3.3.29) in p.62]{Da}):
\begin{eqnarray}
\varphi=0, \quad \varphi_1=0, \quad \varphi_2=0.\label{eq104}
\end{eqnarray}
They represent respectively the conservation law of  linear momentum, mass and energy.

The Euler system \eqref{eq104} fits into the general form of multidimensional balance laws. For supersonic flows, it could be written as a symmetric hyperbolic system \cite[Section 2.1]{wangyuan2014}, for which a local well-posedness theory of (piecewise) classical solutions is now available  (see \cite{BS2007, Da}). For subsonic flows without stagnation points ({\it i.e.} $0<|u|<c$), system \eqref{eq104} is of elliptic-hyperbolic composite-mixed type \cite[p.538]{CY}. In fact, using the standard Descartes coordinates $(z^1, z^2, z^3)$ of Euclidean space $\mathbb{R}^3$, one could  write \eqref{eq104} (with $b=0$ and $r=0$) in the matrix form
\begin{eqnarray*}\label{eqsys}
A_1(U)\p_{z^1}U+A_2(U)\p_{z^2}U+A_3(U)\p_{z^3}U=0.
\end{eqnarray*}
Here $U=(u, p, s)^\top\in\mathbb{R}^5$, and $u=(u_1, u_2, u_3)^\top\in\mathbb{R}^3$; for $n=(n_1, n_2, n_3)^\top\in\mathbb{R}^3$,  the $5\times 5$ real symmetric matrices  $A_1, A_2$ and  $A_3$ are  given by
$$\sum_{j=1}^3A_j(U)n_j=\left(
                                        \begin{array}{ccc}
                                          \frac{u\cdot n}{\rho c^2} & n^\top& 0\\
                                          n& \rho(u\cdot n)I_3& 0\\
                                          0& 0& u\cdot n\\
                                        \end{array}
                                      \right),$$
where $I_3$  is the $3\times 3$ identity matrix. If $u_1\in (0, c)$ and $|u|<c$, then
for any $(\xi,\eta)\in\mathbb{R}^2\setminus\{(0,0)\}$, $A_1(U)$ is nonsingular and  $A_1(U)^{-1}(A_2(U)\xi+A_3(U)\eta)$ is diagonalizable over $\mathbb{C}$, with a pair of conjugate imaginary eigenvalues $\lambda_\pm=\lambda_R\pm\sqrt{-1}\lambda_I$ of multiplicity one, and a real eigenvalue $\lambda_0=\frac{u_2\xi+u_3\eta}{u_1}$ of multiplicity three. Here
$$\lambda_R=\frac{u_1}{(u_1)^2-c^2}(u_2\xi+u_3\eta)\quad\text{and}\quad\lambda_I=\frac{c\sqrt{(c^2-(u_1)^2)(\xi^2+\eta^2)-(u_2\xi+u_3\eta)^2}}
{(u_1)^2-c^2}$$
are real. A  system of first-order partial differential equations  with such properties is called of {\it elliptic-hyperbolic composite-mixed type}.
To our best knowledge,  up to now there is no any theory for multi-dimensional systems of such nonclassical type. (See \cite{dz} for some results about the two-independent-variable case, which is based on the theory of generalized analytic functions.)

The basic idea to treat such elliptic-hyperbolic composite-mixed type system is to decompose it to an elliptic-hyperbolic coupled system,  which consists of several hyperbolic equations coming from the real eigenvalues, and some elliptic equations originating from the imaginary eigenvalues. This was proposed by S. Chen in \cite{Chen2005, Chen2006} in studying a transonic shock problem for the two-dimensional steady compressible Euler system. He achieved this by introducing Lagrangian coordinates via conservation of mass and using characteristic decomposition method. This was then generalized by other authors, and paved the way for a quite well developed theory of steady subsonic flows and transonic shocks in two dimensional case (see, for example,  \cite{LY, Chen-Chen-Feldman2007, Yu1, fangliuyuan, lixinyin,yuan2006} and references therein).

An effective decomposition for the three-dimensional Euler system \eqref{eq104} is apparently more difficult.  The third author proposed a method in \cite{Yuan2007} %which is different from the one in \cite{Chen2005, Chen2006}
to decompose the  Euler system with cylindrical symmetry.  This was further developed  in \cite{CY}, where the authors used characteristic decomposition and pseudo-differential calculus to decompose \eqref{eq104} to a second order elliptic equation of pressure coupled with four transport equations of velocity and entropy, and solved a problem on the instability of a class of transonic shocks in three dimensional straight duct under perturbations of the back pressure ({\it i.e.}, the pressure given on the exit of the duct). (The case of three dimensional straight duct with general smooth cross section was considered by S. Chen in \cite{Chentams}.) The merit of this problem is that the background transonic shock solution is piecewise constant, which simplifies the decomposition since many terms can be considered directly as higher order terms in linearization. However, if we wish to deal with non-constant background solutions, such as spherically symmetric subsonic flows or  transonic shocks established in \cite{Yu3}, which turns out to be very important in understanding transonic shock phenomena in divergent nozzles \cite{LY, lixinyin}, we need to refine the decomposition to handle these extra terms. In \cite{CY2013}, the authors proposed another method to decompose the Euler system \eqref{eq104}. However, unlike \cite{CY}, it is quite difficult to show that the decomposed problems are equivalent to the Euler system, since there involve too much differentiations of the Euler system.

In this work, we combine the ideas presented in \cite{CY,CY2013} to decompose the Euler system \eqref{eq104} to a second order equation of pressure plus four transport equations of entropy, Bernoulli constant, and tangential velocity components respectively, and show that under appropriate boundary conditions, for quite regular solutions,  the decomposed system is equivalent to the Euler system \eqref{eq104} (see Theorem \ref{thm21}). Then we apply this method to study two typical problems in mathematical gas dynamics, namely the stability and uniqueness of spherically symmetric subsonic flows and transonic shocks under multidimensional perturbations of suitably given boundary conditions. The main results are Theorem \ref{thm41} and Theorem \ref{thm501}. It is remarkable that as in \cite{CY2013}, we utilized some tools from differential geometry to avoid the topological singularity of the spherical coordinates of $\mathbb{R}^3$, and our decomposition works in general product Riemannian manifold.  In the studies of transonic shocks, we discovered many beautiful intrinsic structures of the Rankine-Hugoniot conditions and the Euler system, and there arise problems like solving div-curl system on sphere ({\it cf.} \eqref{eq5100}), and a nonclassical nonlocal elliptic equation with  Venttsel boundary condition ({\it cf.} \eqref{eq564}). It should be noted that
the methods and ideas presented in this work can be used to treat some other typical  gas dynamics problems involving subsonic flows. %such as subsonic--supersonic transonic flows, if Euler system \eqref{eq104} are the governing partial differential equations.
We also remark that due to the carefully designed decompositions, we avoid loss of derivatives and the nonlinear (free boundary) problems could be solved just by using a standard generalized Banach fixed point theorem.

We note that for subsonic flows in a straight duct with rectangular cross section, Chen and Xie \cite{xiechen} found an approach different from ours to deal with isentropic Euler system under certain boundary conditions on the entry and exit. %Due to the special structure of the transport equation of vorticity for the isentropic case, the pressure and velocity they solved have the same regularity. (By our method, the pressure solved is one order smoother than the density and velocity, due to the stronger hyperbolic effect.) It is not clear if their method could be adapted to work for the non-isentropic case.
We also recommend the paper \cite{weng} of Weng for many interesting observations on the equations \eqref{eq104}.  The papers \cite{CY, Chentams, CY2013, xiechen,weng} are the only works we know investigating the genuinely three dimensional steady subsonic Euler system.

%It is worth mentioning that the variance of entropy might play a crucial role for the stability of transonic shocks in divergent nozzles under general perturbations of the back pressure. Chen and Feldman had already discovered  that the positions of transonic shocks in irrotational isentropic flows can not be uniquely determined (hence not possible to be stable), even if the flow is spherically symmetric ({\it cf.} \cite{Chen-Feldman-2003}). Later many authors proved that the spherical transonic shocks were stable in two dimensional non-isentropic Euler flows (see, for example,  \cite{LY, lixinyin}).  For three dimensional case, Bae and Feldman proposed a ``non-isentropic potential flow model'' in \cite{BaeFeldman} and proved stability of spherical transonic shock solutions of this model. In this work we also show that almost all spherical transonic shocks are stable in non-isentropic Euler flows. This further demonstrates the significant role of entropy on stability of transonic shocks in nozzles. (However, see also \cite{xuyin2010} for an interesting results on instability of transonic shocks attached to a slim cone for compressible Euler system, by utilizing variance of entropy along streamlines, and the special geometry of the conical domain.)

We remark in passing that, as a byproduct, we also proved in Theorem \ref{thm31} the global uniqueness of spherically symmetric subsonic flows in the class of $C^2$ functions; namely, any $C^2$ solution to the Euler system \eqref{eq104} with suitable spherically symmetric boundary data must be itself spherically symmetric. It is reduced to the case of irrotational flow by studying the transport equations of vorticity.  However, similar global uniqueness result for the transonic shocks is unknown, even for the simpler piecewise constant  case studied in \cite{CY}, or the two-dimensional cylindrical symmetric case (the case studied in \cite{LY}). (Note that the global uniqueness of several piecewise constant transonic shocks had been proved for two-dimensional steady compressible Euler system in \cite{fangliuyuan}.)

The rest of the paper is organized as follows. In Section \ref{sec2} we prove a general theorem (Theorem \ref{thm21}) on the decomposition of Euler system \eqref{eq104}. In Section \ref{sec3} we formulate the problem of stability and uniqueness of spherical subsonic flows, and  prove Theorem \ref{thm31} on the global uniqueness. In Section \ref{sec4}, we apply Theorem \ref{thm21} to prove Theorem \ref{thm41}. Section \ref{sec5} is devoted to formulating the transonic shock problems and proving Theorem \ref{thm501} on the stability of certain spherically symmetric transonic shock solutions under general perturbations of upcoming supersonic flows and back pressure. Finally, Section \ref{sec6} is an appendix, which contains solvability and estimates of  transport equations, and a div-curl system on sphere. We also present some details of estimates of  higher order terms in suitable function spaces. These higher order terms appear during linearization, and their estimates turn out to be  quite straightforward.

\section{A decomposition of steady Euler system}\label{sec2}

We develop here a strategy to decompose \eqref{eq104}, the three-dimensional steady full compressible Euler system, into a second order equation for pressure  and four transport equations. Then we show that any quite regular solution to the decomposed system is also a solution to the  Euler system.

For convenience of applications in mind, we would like to write the
decomposition in a quite general way by using some terminologies
from differential geometry. For a given Riemannian manifold
$\mathcal{M}$ with metric tensor $G$, we use $\dd$ to denote the
exterior differential operator and $D_u\omega$  the covariant
derivative of  a tensor field $\omega$ with respect to a vector
field $u$ in $\mathcal{M}$. $\mathcal{L}_u \omega$ is the Lie
derivative of $\omega$ with respect to $u$. %The
%symbol $\Delta'$ represents the standard Laplacian of functions on
%$\mathcal{M}$.
As a convention, repeated Roman indices will be summed up
for $0,1,2$, while repeated Greek indices are to be summed over for
$1,2$. We also use standard notations such as $C^k$, $C^{k,\alpha}(\Omega)$ and $H^s(\Omega)$ to denote respectively the class of $k$-times ($k\in\mathbb{N}\cup\{0\}$) continuously differentiable functions, and the H\"{o}lder space of $C^k$ functions on an open (or closed) set $\Omega$ whose all $k$-th  order derivatives are H\"{o}lder continuous in $\Omega$ with exponent $\alpha\in(0,1]$, and the Sobolev space $W^{s,2}(\Omega)$ with $s\ge0$.

\subsection{Some well-known reductions}\label{sec201}

We firstly derive some well-known equations from the Euler
system \eqref{eq104} which are valid only for $C^1$ flows without vacuum ($\rho>0$).

The conservation of mass $\varphi_1=0$ can be written (equivalently for $C^1$ flows) as
\begin{eqnarray*}
\varphi_1=D_u \rho+\rho\, \di\, u=0.
\end{eqnarray*}
Then the conservation of energy $\varphi_2=0$ becomes the Bernoulli law
\begin{eqnarray}\label{eq21}
\varphi_3\triangleq\frac{1}{\rho}({\varphi_2-E\varphi_1})=D_u E-G(b,u)-r=0.
\end{eqnarray}
By the identity $\di (u\otimes v)=(\di\, v)u+ D_v u$
for two vector fields $u$ and $v$, the conservation of momentum
$\varphi=0$ turns out to be
\begin{eqnarray}\label{eq22}
\varphi_0\triangleq\frac{1}{\rho}({\varphi-\varphi_1 u})= D_u
u+\frac{1}{\rho}\grad\, p-b=0.
\end{eqnarray}

Since $|u|^2=G(u,u)$, and $\frac 12 D_u\big(|u|^2\big)=G(D_u u, u)=G(\varphi_0, u)-\frac 1\rho D_up+G(u,b),$
we have
\begin{eqnarray}\label{eq23}
\varphi_4\triangleq\frac{\gamma-1}{\rho^{\gamma-1}}\big(\varphi_3-G(\varphi_0,
u)\big)=D_uA(s)-\frac{\gamma-1}{\rho^{\gamma-1}}r=0;
\end{eqnarray}
namely, for $r=0$, the well-known fact that the entropy is invariant along flow trajectories for $C^1$ flows.
From this, the conservation of mass may be written as
\begin{eqnarray}\label{eq27}
\bar{\varphi}_1\triangleq\frac{\varphi_1}{\rho}+\frac{\varphi_4}{\gamma A(s)}=\di\, u+\frac{D_up}{\gamma p}-\frac{\gamma-1}{c^2}r=0.
\end{eqnarray}

We now derive a transport equation for vorticity.
For a vector field $u=u^i\p_i$ in $\mathcal{M}$, we  use $\bar{u}=u^jG_{ij}\dd x^i$ to denote its corresponding $1$-form. Then by the formula $ \mathcal{L}_u \bar{u}=D_u\bar{u}+\dd(\frac{|u|^2}{2})$, \eqref{eq22} is equivalent to
\begin{eqnarray}
\bar{\varphi}_0=D_u \bar{u}+\frac{\dd p}{\rho}-\bar{b}=\mathcal{L}_u\bar{u}-\dd\left(\frac{|u|^2}{2}\right)+\frac{\dd p}{\rho}-\bar{b}=0.
\label{eq24}
\end{eqnarray}
Noting that $\dd$ commutes with Lie derivative and $\dd^2=0$, this implies
\begin{equation}\label{eq25}
\mathcal{L}_u(\dd \bar{u})=-\dd\left(\frac{1}{\rho}\right)\wedge
\dd p+\dd\bar{b}.
\end{equation}
%which is a system of transport equations of the vorticity of the flow field. For isentropic flow ({\it i.e.} $s$ is constant) and $\bar{b}=0$, the right-hand side  would be simply zero.

\subsection{Second-order equation for pressure}\label{sec202}

Straightforward computation yields the following tensor identity:
\begin{eqnarray}\label{eqtensor}
\di (D_uu)-D_u(\di u)=C_1^1C^1_2(Du\otimes Du)+\mathrm{Ric}(u,u),
\end{eqnarray}
where $Du$ is the covariant differential of the vector field $u$, $\mathrm{Ric}(\cdot, \cdot)$ is the Ricci curvature tensor, and $C^i_j(T)$ is the contraction on the upper $i$
and lower $j$ indices of a tensor $T$.  From  \eqref{eq22} and \eqref{eq27}, we get
\begin{eqnarray*}
\di (D_uu)-D_u(\di\, u)=\di\,
\varphi_0-D_u\bar{\varphi}_1+D_u\big(\frac{D_u p}{\gamma p}-\frac{\gamma-1}{c^2}r\big)-\di
\big(\frac{\grad\, p}{\rho}-b\big).
\end{eqnarray*}
It follows the identity:
\begin{eqnarray}\label{eq29}
D_u\bar{\varphi}_1-\di \varphi_0
&=&D_u\left(\frac{D_u p}{\gamma p}\right)-\di \left(\frac{\grad\,
p}{\rho}\right)\nonumber\\
&&- C^1_1C^1_2(D u\otimes D
u)-\mathrm{Ric}(u,u)-(\gamma-1)D_u\left(\frac{r}{c^2}\right)+\di\,b.
\end{eqnarray}
%We could check later that the second order operator for $p$ on the right-hand side is elliptic for subsonic flows.
We remark that in many applications, $\mathcal{M}$ is a connected open subset of $\mathbb{R}^3$ with metric $G$ induced from the standard metric of $\mathbb{R}^3$, so $\mathrm{Ric}(u,u)\equiv0.$

\subsection{The decomposed system and its equivalence to  Euler system} \label{sec203}

We need some global information of $\mathcal{M}$ to show that any regular solution of suitably combined reduced equations is also a solution of the Euler system \eqref{eq104}. To be specific, we assume that $\mathcal{M}$ is a product Riemannian manifold given by
$\mathcal{M}=(a,b)\times M,$
where $M$ is a closed surface, and $(a,b)\subset\mathbb{R}$ is an open interval. Then $\p\mathcal{M}$,  the boundary of $\mathcal{M}$, is given by $M_0\cup M_1$, with
$M_0=\{a\}\times M$ and $M_1=\{b\}\times M$. We also suppose that $(x^0, x^1,x^2)$ with $x^0\in(a, b)$, $(x^1,x^2)\in M$ is a local coordinates system on $\mathcal{M}$. The following is the first main theorem of this paper.

\begin{theorem}\label{thm21}
In addition to the above assumptions on $\mathcal{M}$, for polytropic gases, suppose also that $p \in C^2(\mathcal{M})\cap C^1(\overline{\mathcal{M}})$, $\rho, u\in C^1(\overline{\mathcal{M}})$, and $\rho\ne0, u^0=G(u,\p_0)\ne0$ in $\overline{\mathcal{M}}$. Then  $p, \rho, u$ solve the Euler system \eqref{eq104} in $\mathcal{M}$ if and only if they satisfy the following equations in $\mathcal{M}$:
\begin{eqnarray}
&&D_uA(s)-\frac{\gamma-1}{\rho^{\gamma-1}}r=0,\label{eq212}\\
&&D_uE-G(b,u)-r=0,\label{eq212ber}\\
&&D_u\left(\frac{D_u p}{\gamma p}\right)-\di \left(\frac{\grad\,
p}{\rho}\right)- C^1_1C^1_2(D u\otimes D u)-\mathrm{Ric}(u,u)-(\gamma-1)D_u\left(\frac{r}{c^2}\right)+\di\,b\nonumber \\ &&\quad\qquad+L^0(\frac{D_up}{\gamma p}+\di\,u-\frac{\gamma-1}{c^2}r)+L^1(D_uE-G(b,u)-r)\nonumber\\
&&\qquad\qquad\qquad\quad+L^2(D_uA(s)-\frac{\gamma-1}{\rho^{\gamma-1}}r)
+L^3(D_uu+\frac{\grad\,p}{\rho}-b)=0,\label{eq210}\\
&&G(D_u u+\frac{\grad\, p}{\rho}-b,\p_\alpha)=0,\quad\alpha=1,2\label{eq213}
\end{eqnarray}
and the boundary condition
\begin{multline}
\frac{D_u p}{\gamma p}+\di\,u-\frac{\gamma-1}{c^2}r+L_1(D_uE-G(b,u)-r)+L_2(D_uA(s)-\frac{\gamma-1}{\rho^{\gamma-1}}r)\\
+L_3(D_uu+\frac{\grad\,p}{\rho}-b)=0 \qquad \text{on}\quad M_1.\label{eq215}
\end{multline}
Here $L^0(\cdot)$ is a linear function, and  $L^k(\cdot)$, $L_k(\cdot)$ are smooth functions so that $L^k(0)=0, L_k(0)=0$ for $k=1,2,3$.
\end{theorem}

\begin{proof} The necessity is obvious from the above deductions. For sufficiency,
by comparing  \eqref{eq212} with \eqref{eq23}, \eqref{eq212ber} with \eqref{eq21}, \eqref{eq210} with \eqref{eq29}, \eqref{eq213} with \eqref{eq22}, and \eqref{eq215} with \eqref{eq27}, we conclude that the above equations \eqref{eq212}--\eqref{eq215} are equivalent to the following expressions:
\begin{eqnarray}
&\varphi_4=0=\varphi_3-G(\varphi_0, u)&\text{in}\quad\mathcal{M},\label{eqadd101}\\
&\varphi_3=0=\varphi_2-E\varphi_1&\text{in}\quad\mathcal{M},\label{eqadd101ber}\\
&D_u\bar{\varphi}_1-\di \varphi_0+L^0(\bar{\varphi}_1)+L^1(\varphi_3)+L^2(\varphi_4)+L^3(\varphi_0)=0 &\text{in}\quad\mathcal{M},\label{eq216}\\
&G({\varphi}_0, \p_\alpha)=0, \quad\alpha=1,2 &\text{in}\quad\mathcal{M},\label{eq217}\\
&\bar{\varphi}_1+L_1(\varphi_3)+L_2(\varphi_4)+L_3(\varphi_0)=0 &\text{on}\quad {M_1}. \label{eq221}
\end{eqnarray}
From \eqref{eqadd101} and \eqref{eqadd101ber}  we see $\varphi_4=0, \varphi_3=0$, hence $G(\varphi_0, u)=0$. Set  $u=u^k\p_k$. Then \eqref{eq217} yields $G(\varphi_0, u^0\p_0)=0$. Note that $u^0$ never vanishes in $\mathcal{M}$ by our assumption, one concludes that $G(\varphi_0, \p_0)=0$, hence $\varphi_0=0$. Therefore \eqref{eq216} is reduced to $D_u\bar{\varphi}_1+L^0(\bar{\varphi}_1)=0,$ and \eqref{eq221} becomes to be $\bar{\varphi}_1=0$ on $M_1$. Recalling that $L^0$ is a linear function, these
imply that $\bar{\varphi}_1=0$ in $\mathcal{M}$. So by $\varphi_4=0$, we have $\varphi_1=0$, and hence from $\varphi_3=0$ to get $\varphi_2=0$. Finally, by \eqref{eq22}, it follows that $\varphi=0$. So equations \eqref{eq104} hold as desired.
\end{proof}

\begin{remark}
It is clear that we could propose the condition \eqref{eq215} on $M_0$
and the same conclusion holds. Later in the studies of transonic shocks we will actually restrict \eqref{eq215} on a shock-front. All the functions $L^k, L_k$ in the theorem are to be specified in the applications.
\end{remark}

\begin{remark}
%Comparing to the decomposition given in \cite{CY}, the differences are that we derive the equation of pressure \eqref{eq210} not from pseudo-differential calculus, but by using the tensor identity \eqref{eqtensor} (the same as in \cite{CY2013}), and we replace one of the transport equation of velocity by the Bernoulli law \eqref{eq212ber}. The refinement to the method in \cite{CY2013} is that we no longer try to find an elliptic equation of velocity (which generally involves too much differentiations and hence seems to be impossible to recover the Euler system), hence we just solve the four transport equations,  \eqref{eq212}, \eqref{eq212ber} and \eqref{eq213}, to obtain entropy and velocity. So
We note that unlike \cite{CY2013}, the regularity of pressure here is one order smoother than other unknowns. An iteration scheme works to solve the nonlinear problem,  due to the special structure of the equation \eqref{eq210}, where second order derivatives of pressure occur, while for other unknowns (such as $\rho, u$), only first order derivatives appear.
\end{remark}

%\begin{remark}
%If $\p M$ is non-empty, then $\mathcal{M}$ has edge singularities at $\{a\}\times\p M$ and $\{b\}\times\p M$.  We know that  solutions to elliptic equations are generally not $C^2$ in such domains. However, our primary goal in this paper is to deal with the major difficulty caused by the composition of elliptic effect and hyperbolic effect in the steady subsonic Euler system. So we leave the issue of existence of lower regular solutions to  future works.
%\end{remark}

\begin{remark}
We may consider the case that $M$ is a closed manifold, such as $\mathbb{T}^2$,  the flat $2$-torus given by $\mathbb{R}^2/\mathbb{Z}^2$ ($\mathbb{Z}$ is the set of integers), or $S^2$, the unit 2-sphere in $\mathbb{R}^3$. For the former case, the existence and  stability of constant subsonic flow had been studied in \cite{weng} and \cite{xiechen} (for isentropic gas), and the transonic shock problem was solved in \cite{CY}. These could be handled in a similar but simpler way than the latter case, since the special solution is (piecewise) constant. So in this paper we focus on the case that $M=S^2$, for which the background solution is variable, and many interesting new phenomena, like nonlocal elliptic problems, might occur (see Section \ref{sec5}).
\end{remark}

\begin{remark}
We consider in this section decomposition of  \eqref{eq104} with exterior force and heat supply for convenience to applications to other problems later. In the rest of this paper, like all the references mentioned above for subsonic flows and transonic shocks, we only consider the Euler system \eqref{eq104} with exterior force $b=0$ and heat supply $r=0$. %In the following we may use $b$ or $r$ to denote some other objects and they should not be confused.
\end{remark}

\section{Existence and uniqueness of spherically symmetric subsonic flows}\label{sec3}

In the following two sections, we apply results in Section \ref{sec201} and Theorem \ref{thm21} to a physically interesting problem, namely, the uniqueness and stability of spherically symmetric subsonic flows in $\mathbb{R}^3$. We firstly formulate rigorously the boundary value problem to be concerned,  then show the existence of spherically symmetric subsonic solutions under suitable boundary conditions, and prove that such solutions are unique in the class of $C^2$ flows. Then in Section \ref{sec4}, we prove that some of these spherically symmetric subsonic solutions are stable under general multidimensional perturbations of boundary conditions.

\subsection{Formulation of subsonic flow problem for steady compressible Euler system}

We now specify $\mathcal{M}=(r_0, r_1)\times S^2$ for two constants $0<r_0<r_1<\infty.$ So $\mathcal{M}$ is a spherical shell in $\mathbb{R}^3$.
Let $r=x^0\in (r_0,r_1)$, and $x'=(x^1, x^2)$ be a  (local) spherical coordinates on $S^2$. The Euclidean metric of $\mathbb{R}^3$ in the coordinates $(x^0, x')$ takes the form
%\begin{eqnarray*}
$G=G_{ij}\dd x^i\otimes \dd x^j=\dd x^0\otimes \dd x^0+(x^0)^2g,$
%\end{eqnarray*}
where $g=\dd x^1\otimes \dd x^1+(\sin x^1)^2\dd x^2\otimes \dd x^2$
is the standard metric of $S^2$.
For later reference, we list below all the nonzero Christoffel symbols associated with $G$ (note that $\Gamma_{jk}^i=\Gamma^i_{kj}$):
\begin{gather*}
\Gamma_{11}^0=-x^0, \quad \Gamma^0_{22}=-x^0(\sin x^1)^2, \quad
\Gamma^1_{01}=\Gamma^1_{10}=\frac{1}{x^0},\\
\Gamma^1_{22}=-\sin x^1\cos x^1, \quad
\Gamma^2_{02}=\Gamma^2_{20}=\frac{1}{x^0}, \quad
\Gamma^2_{12}=\Gamma^2_{21}=\cot x^1.
\end{gather*}
We also set $\sqrt{G}=(x^0)^2\sin x^1$, and $(G^{ij})=(G_{ij})^{-1}$. %to be the inverse of the matrix $(G_{ij})$.

Now write the velocity as $u=u^0\p_0+u',$ with $u'$ a $x^0$-dependent vector field on $S^2$, which is $u'=u^1\p_1+u^2\p_2$ in a local coordinates system.
We call $u^0$ the {\it normal velocity}  and $u'$ the {\it tangential velocity}.  We prescribe the pressure on  $M_0$:
\begin{eqnarray}
p=p_0(x'), \label{eq31}
\end{eqnarray}
and the following boundary conditions on $M_1$:
\begin{eqnarray}\label{eq32}
E=E_1(x'),\quad s=s_1(x'), \quad u'=u'_1(x').
\end{eqnarray}
Here $u'_1(x')$ is a given vector field on $S^2$, while $E_1, s_1$ and $p_0$ are given functions on $S^2$.

\medskip
\fbox{
\parbox{0.90\textwidth}{
Problem (S): Solve the Euler system \eqref{eq104} in $\mathcal{M}$, subjected to the boundary conditions \eqref{eq31}--\eqref{eq32}.}}
\medskip

\begin{remark}
Choosing suitable boundary conditions for the steady subsonic compressible Euler system to formulate a well-posed problem is a delicate issue both for theoretical studies and numerical computations, due to its nature as a nonlinear elliptic-hyperbolic composite-mixed system. It had been shown that some choices of boundary conditions, such as given pressure both on $M_0$ and $M_1$, will lead to ill-posed problems, even in the two-space-dimension case ({\it cf.} \cite{Yu1}). The conditions  \eqref{eq31}--\eqref{eq32} are compatible with the decomposition stated in Theorem \ref{thm21}. They are also motivated by  previous studies of transonic shocks and subsonic flows ({\it cf.} \cite{Yu1, LY, Liu2010} etc.). For example, giving a boundary condition of pressure is a physically interesting case in studying nozzle flows \cite[p.373, p.385]{CF}; while the Bernoulli constant $E$ and the entropy $s$ are usually determined by the upstream or downstream far-field flows.
\end{remark}

\subsection{Existence of spherically symmetric subsonic flows}
We now consider the following special case of Problem (S).

\medskip
\fbox{
\parbox{0.90\textwidth}{
Problem (S1): { Solve the Euler system \eqref{eq104} in $\mathcal{M}$,  subjected to the boundary conditions \eqref{eq31}--\eqref{eq32},
where $p_0, E_1, s_1$ are all constants, and  $u_1'\equiv0$.}}}
\medskip

We can construct a special symmetric solution to this problem.
Suppose that the solution depends only on $r$,  and the tangential velocity $u'$ is identically zero in $\mathcal{M}$. Then the Euler system \eqref{eq104} is reduced to the following ordinary differential equations \cite{Yu3}:
\begin{eqnarray}
&&\frac{\dd u}{\dd r}=\frac{2c^2u}{r(u^2-c^2)},\label{eq34}\\
&&\frac{\dd \rho}{\dd r}=-\frac{2\rho u^2}{r(u^2-c^2)},\label{eq35}\\
&&\frac{\dd p}{\dd r}=-\frac{2\rho
c^2u^2}{r(u^2-c^2)}.\label{eq36}
\end{eqnarray}
Here, for simplicity, instead of $u^0(r)$, we have written  $u=u(r)$ to be the normal velocity.
Let $M=u/c$ be the Mach number. Then it solves
\begin{eqnarray}\label{eqmachode}
\frac{\dd M}{\dd r}=\frac{1}{r}\frac{M[2+(\gamma-1)M^2]}{M^2-1}.
\end{eqnarray}
So the flow is always subsonic if it is subsonic at the entry $r=r_0$ and $u$ is positive in $\mathcal{M}$. Integrating \eqref{eqmachode} yields that
\begin{eqnarray}\label{eq37add2}
r=c_1\frac{[2+(\gamma-1)M^2]^{\frac14+\frac{1}{2(\gamma-1)}}}{\sqrt{M}}.
\end{eqnarray}
Also note that
$$\frac{\dd u}{\dd M}=\frac{2u}{2M+(\gamma-1)M^3},$$
hence
$$u=c_2\frac{M}{\sqrt{2+(\gamma-1)M^2}}.$$
Both $c_1, c_2$ are constants to be determined by boundary conditions.

\begin{lemma}\label{lem31}
Suppose that
\begin{eqnarray}\label{eq37add}
p_0>\left(\frac{2}{\gamma}\cdot\frac{\gamma-1}{\gamma+1}\right)^{\frac{\gamma}{\gamma-1}}
E_1^\frac{\gamma}{\gamma-1}A(s_1)^{\frac{-1}{\gamma-1}},\quad u^0(r_0)>0.
\end{eqnarray}
Then there is one and only one symmetric subsonic solution $\underline{U}$ to Problem (S1). Furthermore, the solution is real analytic.
\end{lemma}

\begin{proof}
1. From \eqref{eq212} and \eqref{eq212ber}, we see $E\equiv E_1$ and $s\equiv s_1$ in $\mathcal{M}$. So the assumption \eqref{eq37add} means that the flow is subsonic at the entry $r=r_0$. By $p_0$ and $s_1$ we could solve $\rho=\rho_0>0$ at $r=r_0$ from $p=A(s)\rho^\gamma$, and then $u=u_0=u^0(r_0)>0, M=M_0\in(0,1)$ at $r=r_0$ by the fact that $E_1=\frac12 u^2+\frac{\gamma p}{(\gamma-1)\rho}$.

2. Now from \eqref{eq37add2}, we could determine $c_1$ so that
$$r=r^0\frac{\sqrt{M_0}}{[2+(\gamma-1)M_0^2]^{\frac14+\frac{1}{2(\gamma-1)}}}
\frac{[2+(\gamma-1)M^2]^{\frac14+\frac{1}{2(\gamma-1)}}}{\sqrt{M}}.$$
For $r>r_0$, we have $0<M<M_0$.

3. Then, we solve \eqref{eq34}--\eqref{eq36} for  $r\in [r_0,r_1]$ with initial data $\rho=\rho_0, p=p_0, u=u_0$ at $r=r_0$,  to get the unique subsonic solution $\underline{U}(r)$ by the theory of ordinary differential equations. Obviously the solution is smooth (actually real analytic), and $u>0$ in $\mathcal{M}$.
\end{proof}

\subsection{Global uniqueness of spherically symmetric subsonic flows}

We now prove that the spherically symmetric subsonic solution to Problem (S1) constructed above is unique in a larger class of $C^2$ functions.
To this end, we need the following lemma on vorticity.
\begin{lemma} \label{lem32}
Suppose that $p, \rho$ and $u=u^i\p_i$ are $C^1$ and solve the Euler system \eqref{eq104}, with $E$ and $s$ being constants in $\mathcal{M}$. The vorticity of the flow is given by \begin{eqnarray*}
\dd \bar{u}=\frac12\omega_{km}\dd x^k\wedge \dd x^m, \quad \omega_{km}\triangleq\p_k(G_{im}u^i)-\p_m(G_{ik}u^i),
\end{eqnarray*}
where $\bar{u}=u^iG_{ij}\dd x^j$ is the 1-form associated with the vector field $u$.
Then $\omega_{km}=-\omega_{mk}$, and for each $k$,  there holds in $\mathcal{M}$ the algebraic identities:
\begin{eqnarray}\label{eq37}
u^m\omega_{km}=0.
\end{eqnarray}
\end{lemma}

\begin{proof}
1. For $s$ a constant in $\mathcal{M}$,   the constitutive relation may be written as $p=A\rho^\gamma$, with $A=k_0\exp(s/c_\nu)$ being a constant.

2. From $E=\frac12|u|^2+\frac{c^2}{\gamma-1}$ and using $c^2=\gamma p/\rho$, we have $\frac12\dd(|u|^2)+A\gamma\rho^{\gamma-2}\dd\rho=0.$ Equation  \eqref{eq24} implies that $\mathcal{L}_u\bar{u}-\frac12\dd(|u|^2)+A\gamma\rho^{\gamma-2}\dd\rho=0$. It follows
\begin{eqnarray}\label{eq38}
\mathcal{L}_u\bar{u}-\dd(|u|^2)=0.
\end{eqnarray}
Direct computation in a local coordinates shows that
\begin{eqnarray*}
\dd\bar{u}=\frac12\omega_{km}\dd x^k\wedge \dd x^m,
\end{eqnarray*}
which is by definition the vorticity (2-form) of the flow. However,
\begin{eqnarray*}
\mathcal{L}_u\bar{u}-\dd(|u|^2)&=&\Big(u^m\p_m(u^iG_{ik})+u^iG_{ij}\p_ku^j\Big)\dd x^k-\Big(u^iG_{ij}\p_ku^j+u^j\p_k(G_{ij}u^i)\Big)\dd x^k\\
&=&u^m\left[\p_m(G_{ik}u^i)-\p_k(G_{im}u^i)\right]\dd x^k\\
&=&-u^m\omega_{km}\dd x^k,
\end{eqnarray*}
so by \eqref{eq38} one infers \eqref{eq37} as desired.
\end{proof}

The following is the second main result of this paper. %namely, the global uniqueness of spherically symmetric subsonic flows.
\begin{theorem}\label{thm31}
The only  $C^2(\mathcal{M})\cap C^1(\overline{\mathcal{M}})$ solution to Problem (S1) which satisfies $\rho u^0>0$ in $\overline{\mathcal{M}}$ is  the special subsonic solution $\underline{U}$.
\end{theorem}

\begin{proof}
1. Since $E$ and $s$ are constants in $\mathcal{M}$, we could use Lemma \ref{lem32}. For our special case that $\mathcal{M}=(r_0, r_1)\times S^2$, we have $k=0,1,2$. We also assumed that $u^0>0$. So on $M_1$, from the boundary condition $u'=0,$  ({\it i.e.} $u^1=0, u^2=0$,)  there holds
\begin{eqnarray*}
\omega_{12}=\sum_{i=0}^2[\p_1(G_{i2}u^i)-\p_2(G_{i1}u^i)]=\p_1(G_{02}u^0)-\p_2(G_{01}u^0)=0.
\end{eqnarray*}
The last equality follows from the fact that  $G_{ij}=0$ for $i\ne j$.
Hence we conclude that for all $m,k$,  $\omega_{mk}=0$ on $M_1$; namely,  $\dd\bar{u}=0$ on $M_1$.

2. Now acting $\dd$ on equation \eqref{eq38}, we have $\mathcal{L}_u \dd\bar{u}=0$ in $\mathcal{M}$. By Lemma A.5 in \cite[p.2539]{CY2013}, one concludes that $\dd\bar{u}=0$ in $\mathcal{M}$. Since the first Betti number of $\mathcal{M}$ is zero (that is, $\mathcal{M}$ is simply connected), by Poincar\'{e} lemma, $\bar{u}$ is exact: there is a function $\varphi\in C^3(\mathcal{M})$ so that $\dd\varphi=\bar{u}$, or $u=\grad \varphi$. So conservation of mass becomes
\begin{eqnarray}\label{eq39}
\di (\rho\cdot \grad \varphi)=0,
\end{eqnarray}
with $\rho$ a function of $|\grad \varphi|^2$ determined by the Bernoulli law
$\frac12 |\grad \varphi|^2+\frac{\gamma A\rho^{\gamma-1}}{\gamma-1}=E.$
Equation \eqref{eq39} is the so-called potential flow equation. It is a second order elliptic equation for subsonic flows.

3. The boundary conditions for the potential function $\varphi$ are
\begin{eqnarray}
&\varphi=0 &\text{on}\quad M_1,\label{eq310}\\
&|\grad \varphi|^2=K_0\triangleq2\left(E-\frac{\gamma A}{\gamma-1}\left(\frac{p_0}{A}\right)^{\frac{\gamma-1}{\gamma}}\right) &\text{on}\quad M_0.\label{eq311}
\end{eqnarray}
The theorem is proved by applying the following Lemma \ref{lem33}.
\end{proof}

%\subsection{Uniqueness of spherically symmetric subsonic solution for potential flow equation}

%Our goal is to use the classical maximum principles of second order elliptic equation to prove the following lemma.

\begin{lemma}\label{lem33}
Let $\varphi\in C^2(\overline{\mathcal{M}})$ be a solution to problem \eqref{eq39}--\eqref{eq311} so that $\p_r \varphi\ge0$ on $M_0$. Then it must be a spherically symmetric subsonic flow.
\end{lemma}

\begin{proof}
1. We use the standard global Descartes coordinates of $\mathcal{M}$  given by $(z^1, z^2, z^3)$ so that $r=x^0=\sqrt{(z^1)^2+(z^2)^2+(z^3)^2}.$ Then one checks that equation \eqref{eq39} could be written in the non-divergence form as
\begin{eqnarray*}
\sum_{i,j=1}^3a^{ij}(D\varphi)\p_{ij}\varphi \triangleq c(|D\varphi|^2)^2\Delta\varphi-\sum_{i,j=1}^3\p_i\varphi\p_j\varphi\p_{ij}\varphi=0.
\end{eqnarray*}
Here and in the rest of the proof,  $\p_i=\frac{\p}{\p z_i},\
\p_{ij}=\frac{\p^2}{\p z_i \p z_j}$, and $c$ is the sonic speed; $D\varphi$ and   $\Delta\varphi$  are respectively the standard gradient and Laplace operator in $\mathbb{R}^3$ acting on $\varphi$.

2. Let $\varphi_s=\varphi_s(r)$ be a special symmetric subsonic solution to problem \eqref{eq39}--\eqref{eq311}. The existence of $\varphi_s$ is guaranteed by Lemma \ref{lem31}.  We see that $w=\varphi_s-\varphi$ solves the following equation:
\begin{eqnarray*}\label{weq}
&&\sum_{i,j=1}^3\tilde{a}^{ij}(x)\p_{ij}w
+\sum_{i=1}^3\tilde{b}^i(x)\p_iw\nonumber\\
&\triangleq&\sum_{i,j=1}^3a^{ij}(D\varphi_s)\p_{ij}w
+\sum_{i,j=1}^3\p_{ij}\varphi\cdot(a^{ij}(D\varphi_s)-a^{ij}(D\varphi))
=0.
\end{eqnarray*}
Since $\varphi_s$ is a uniformly subsonic flow, this is a linear uniformly elliptic equation of $w$ in $\mathcal{M}$.

The boundary conditions of $w$ are
\begin{eqnarray*}
w=0&\text{on}\ \ M_1,\quad
D w\cdot l=0 &\text{on}\ \ M_0,
\end{eqnarray*}
with $l=D \varphi_s+D \varphi.$  By the assumption that $\p_r\varphi\ge0$ and the fact that
$\p_r\varphi_s>0$, we see $l\cdot (r^0, 0,0)/|r^0|=\p_r\varphi_s+\p_r\varphi>0$, hence on $M_0$ we have an oblique derivative condition of $w$. By maximum principles \cite[Theorem 3.5 in p.35]{GT} and Hopf boundary point lemma \cite[Lemma 3.4 in p.34]{GT}, we conclude that $w\equiv0$ as desired.
\end{proof}

\begin{remark}
The uniqueness of symmetric subsonic solutions to potential flows in infinite conical nozzles had been proved in \cite{LY2011} by applying Harnack inequalities, see also \cite{Liu2010} for the uniqueness result of %subsonic potential flows in infinite cylinders, {\it i.e.},
the case that $M=\mathbb{T}^2$. %We include the simple proof here just for completeness.
The existence of isentropic irrotational subsonic flows in general three-dimensional largely-open nozzles was proved in \cite{liuyuan2014}, while the same existence problem for the three-dimensional Euler system still remains open. See \cite{cdx} and references therein for some results on subsonic flows in strip-like domains for the two-dimensional compressible Euler system.
\end{remark}

\section{Stability of spherically symmetric subsonic flows}\label{sec4}

In this section we continue to investigate Problem (S).  We are wondering if a spherically symmetric subsonic flow $U_b$ (called as a {\it background solution} in the sequel) constructed by Lemma \ref{lem31} is stable under multidimensonal perturbations of boundary conditions.
The following is the third main theorem we will prove in this paper.
\begin{theorem} \label{thm41}
Suppose that $U_b$ is a background solution so that for $x^0\in (r_0,r_1)$ and $t=M_b^2(x^0)$, there holds
\begin{eqnarray}\label{eq41}
\gamma(1+2\gamma)t^4+
(-4\gamma^2+2\gamma-3)t^3+(14-7\gamma)t^2-19t+6
<0;
\end{eqnarray}
here $M_b=u_b/c_b$ is the Mach number of the background solution.
Then for $\alpha\in (0,1)$, there exist positive constants $\varepsilon_0$ and $C$ depending only on the background solution $U_b$ and $r_0, r_1, \alpha, \gamma$ so that if
\begin{multline}\label{eq42add}
\norm{u_1'(x')}_{C^{3,\alpha}(M_1)}+\norm{E_1(x')-E_b}_{C^{3,\alpha}(M_1)}+
\norm{s_1(x')-s_b}_{C^{3,\alpha}(M_1)}\\+
\norm{p_0(x')-p_b}_{C^{3,\alpha}(M_0)}
\le \varepsilon\le \varepsilon_0,
\end{multline}
there is uniquely one solution $U$ to Problem (S),  with $p\in C^{3,\alpha}(\overline{\mathcal{M}})$, $s, E, u\in C^{2,\alpha}(\overline{\mathcal{M}})$, and
\begin{eqnarray}\label{eq43}
\norm{p-p_b}_{C^{3,\alpha}(\overline{\mathcal{M}})}+
\norm{s-s_b}_{C^{2,\alpha}(\overline{\mathcal{M}})}+
\norm{E-E_b}_{C^{2,\alpha}(\overline{\mathcal{M}})}+
\norm{u'}_{C^{2,\alpha}(\overline{\mathcal{M}})}
\le C\varepsilon.
\end{eqnarray}
\end{theorem}

\begin{remark}\label{rm41}
The requirement \eqref{eq41} is  a quite rough condition for the stability of $U_b$ derived from the decomposition stated in Theorem \ref{thm21} ({\it cf.} Remark \ref{rm42}). Note that $f(1)=-2(\gamma+1)^2<0$,  so a sufficient condition for \eqref{eq41} is that the Mach number at the entry $r_0$ is quite close to $1$ and $r_1$ is close to $r_0$.
\end{remark}

To prove Theorem \ref{thm41},
firstly we formulate a nonlinear problem (S2) by specifying the functions $L^k, L_k$ appeared in Theorem \ref{thm21}, and by the same theorem,  we infer that Problem (S2) is equivalent to Problem (S). Then we construct a nonlinear iteration mapping and solve Problem (S2) by using a  Banach fixed point theorem. Lots of  expressions derived here will also be used in Section \ref{sec5} for the studies of transonic shocks.

\subsection{Problem (S2)}
To formulate Problem (S2), we  need to compute the exact expressions
of \eqref{eq215} and \eqref{eq210} and then specify the auxiliary
functions $L^k, L_k$. %to cancel certain terms coupled with the other
%equations \eqref{eq212}\eqref{eq212ber}\eqref{eq213}. %Recall that we always assume that repeated Greek
%indices are summed over $1,2$, and repeated Roman indices are summed
%up for $0,1,2$, unless otherwise stated.

\subsubsection{Specification of boundary conditions}

%We now specify the boundary condition \eqref{eq215}.
By the definitions of divergence operator and covariant derivative in a local spherical coordinates system of $\mathcal{M}$, we have
\begin{eqnarray}\label{eq42}
\di\,u&=&%\frac{1}{\sqrt{G}}\p_i\left(\sqrt{G}u^i\right)=
\p_0u^0+\frac{2u^0}{x^0}+\frac{1}{\sqrt{G}}\p_\alpha\left(\sqrt{G}u^\alpha\right)%\nonumber\\
=\p_0u^0+\frac{2u^0}{x^0}+\p_\alpha u^\alpha+u^1\Gamma^2_{21},\\
D_uu+\frac{1}{\rho}\grad\,p&=&(u^j\p_ju^k+u^ju^m\Gamma_{jm}^k+\frac{1}{\rho}\p_ipG^{ik})\p_k.\label{eq45add33}
\end{eqnarray}
It follows,
%\begin{eqnarray}\label{eq44}
%G(D_uu+\frac{1}{\rho}\grad\,p,
%\p_0)=u^0\p_0u^0+u^\alpha\p_\alpha u^0+u^\alpha
%u^\beta\Gamma^0_{\alpha\beta}+\frac{1}{\rho}\p_0p,
%\end{eqnarray}
%and
%\begin{eqnarray}\label{eq45add}
%u^0\left(\frac{D_up}{\gamma p}+\di\,u\right)&=&\left(\frac{u^0}{\gamma p}D_up-\frac{\p_0p}{\rho}\right)+\frac{2(u^0)^2}{x^0}-\frac{u^\alpha}{2u^0}\p_\alpha(u^0)^2\nonumber\\
%&&+\frac{u^0}{\sqrt{G}}\p_\alpha(\sqrt{G}u^\alpha)-u^\alpha u^\beta\Gamma_{\alpha\beta}^0+\frac{1}{u^0}G(D_uu+\frac{\grad\,p}{\rho}, u^0\p_0)\nonumber\\
%&=&\frac{1}{\rho}\left[\left(\frac{(u^0)^2}{c^2}-1\right)\p_0p+{u^0u^\alpha}\left(\frac{1}{c^2}+\frac{1}{(u^0)^2}
%\right)\p_\alpha p\right]-\frac{4\gamma}{\gamma-1}\frac{1}{x^0}A(s_0)^{\frac{1}{\gamma}}p^{1-\frac{1}{\gamma}}\nonumber\\
%&&+\frac{4E}{x^0}-\frac{2}{x^0}G_{\alpha\beta}u^\alpha u^\beta+\frac{1}{\gamma-1}\frac{u^\alpha}{u^0}\p_\alpha A(s)\rho^{\gamma-1}
%+\frac{u^\alpha}{2u^0}\p_\alpha(G_{\alpha\beta}u^\alpha u^\beta)-\frac{u^\alpha}{u^0}\p_\alpha E\nonumber\\
%&&+\frac{u^0}{\sqrt{G}}\p_\alpha(\sqrt{G}u^\alpha)-u^\alpha u^\beta\Gamma_{\alpha\beta}^0+\frac{1}{u^0}G(D_uu+\frac{\grad\,p}{\rho}, u^0\p_0).
%\end{eqnarray}
by setting  $L_3(\cdot)=-\frac{1}{(u^0)^2}G(\cdot,
u^0\p_0)$, %By multiplying $\rho$ to both sides of \eqref{eq45add},we get
the identity
\begin{multline}\label{eq48add}
\rho u^0(\frac{D_up}{\gamma p}+\di\, u+L_3(D_uu+\frac{\grad\,p}{\rho}))%\nonumber\\
=\left(\frac{(u^0)^2}{c^2}-1\right)\p_0p+{u^0u^\alpha}\left(\frac{1}{c^2}+\frac{1}{(u^0)^2}
\right)\p_\alpha p-\frac{4\gamma}{\gamma-1}\frac{1}{x^0}p\\
+\frac{4\rho E}{x^0}-\frac{2\rho}{x^0}G_{\alpha\beta}u^\alpha u^\beta+\frac{1}{\gamma-1}\frac{u^\alpha}{u^0}\p_\alpha A(s)\rho^{\gamma}
+\frac{\rho u^\delta}{2u^0}\p_\delta(G_{\alpha\beta}u^\alpha u^\beta)-\frac{\rho u^\alpha}{u^0}\p_\alpha E\\
+\frac{\rho u^0}{\sqrt{G}}\p_\alpha(\sqrt{G}u^\alpha)-\rho u^\alpha u^\beta\Gamma_{\alpha\beta}^0.
\end{multline}
Comparing this to \eqref{eq215}, we see it is actually a nonlinear
Robin condition for $p$ on $M_1$:
%\begin{eqnarray}\label{eq47add}
%&&\p_0p-\frac{4\gamma}{\gamma-1}\frac{1}{\frac{(u^0)^2}{c^2}-1}\frac{1}{x^0}p
%+\frac{1}{\frac{(u^0)^2}{c^2}-1}\frac{4\rho E}{x^0}\nonumber\\
%&=&\frac{1}{\frac{(u^0)^2}{c^2}-1}\left[\frac{2\rho}{x^0}G_{\alpha\beta}u^\alpha
%u^\beta-\frac{1}{\gamma-1}\frac{u^\alpha}{u^0}\p_\alpha
%A(s)\rho^{\gamma}
%-\frac{\rho u^\delta}{2u^0}\p_\delta(G_{\alpha\beta}u^\alpha u^\beta)+\frac{\rho u^\alpha}{u^0}\p_\alpha E\right.\nonumber\\
%&&\left.-\frac{\rho u^0}{\sqrt{G}}\p_\alpha(\sqrt{G}u^\alpha)+\rho u^\alpha u^\beta\Gamma_{\alpha\beta}^0-{u^0u^\alpha}\left(\frac{1}{c^2}+\frac{1}{(u^0)^2}
%\right)\p_\alpha p\right].
%\end{eqnarray}
%Using the definition of $E$, namely
%\begin{eqnarray}\label{eq47add2}
%E=\frac12G_{ij}u^iu^j+\frac{\gamma}{\gamma-1}A(s)^{\frac{1}{\gamma}}p^{1-\frac{1}{\gamma}},
%\end{eqnarray}
%\eqref{eq47add} could be simplified as
\begin{eqnarray}\label{eq48}
\p_0p +\frac{1}{x^0}\frac{2\gamma p
(u^0)^2}{(u^0)^2-c^2}=G_1(U)+G_2(U),
\end{eqnarray}
with
\begin{eqnarray}\label{eq49}
G_1&\triangleq&\frac{1}{\frac{(u^0)^2}{c^2}-1}\left[-\frac{\rho u^0}{\sqrt{G}}\p_\alpha(\sqrt{G}u^\alpha)\right],\\
G_2&\triangleq&\frac{1}{\frac{(u^0)^2}{c^2}-1}\left[-\frac{1}{\gamma-1}\frac{u^\alpha}{u^0}\p_\alpha A(s)\rho^{\gamma}
-\frac{\rho u^\delta}{2u^0}\p_\delta(G_{\alpha\beta}u^\alpha u^\beta)+\frac{\rho u^\alpha}{u^0}\p_\alpha E\right.\nonumber\\
&&\left.+\rho u^\alpha u^\beta\Gamma_{\alpha\beta}^0-{u^0u^\alpha}\left(\frac{1}{c^2}+\frac{1}{(u^0)^2}
\right)\p_\alpha p\right].\label{eq410}
\end{eqnarray}
%By restricting this on $\{x^0=r_1\}$, we get a boundary condition on $M_1$.

\subsubsection{Specify equation of pressure}
We now compute the explicit expression of the equation
\eqref{eq210}. It is straightforward to check that
\begin{multline}\label{eq415}
D_u\left(\frac{D_up}{\gamma p}\right)-\di\,\left(\frac{\grad\, p}{\rho}\right)-C^1_1C^1_2(D u\otimes D u)
=\frac{1}{\gamma p}\left[\Big((u^0)^2-c^2\Big)\p_0^2p-\frac{c^2}{(x^0)^2}\Delta' p\right]\\
+\frac{1}{\gamma p}\left[\frac{2}{x^0}\left((u^0)^2-c^2\right)\p_0p+u^0\p_0u^0\p_0p-\frac{(u^0)^2}{p}(\p_0p)^2+\frac{\gamma p}{\rho^2} \p_0\rho\p_0p\right]\\
-(\p_0u^0)^2+2\left(\frac{u^0}{x^0}\right)^2+2\frac{u^0}{x^0}\p_0u^0+H_1(U),
\end{multline}
where $\Delta'$ is the Laplacian on $S^2$, and
\begin{eqnarray}\label{eq425}
H_1(U)&=&\frac{2u^0}{x^0}\frac{u^\alpha\p_\alpha p}{\gamma p}+\frac{1}{\rho^2}G^{\alpha \beta}\p_\alpha p\p_\beta\rho+\sum_{(k,j)\ne(0,0)}\left(\frac{u^ku^j}{\gamma p}\p_{jk}p
+\frac{u^k}{\gamma p}\p_ku^j\p_j p-\frac{1}{\gamma p^2}u^ku^j\p_kp\p_jp\right)\nonumber\\
&&-\sum_{(l,j)\ne(0,0)}\left(\p_ju^l\p_lu^j+2
\Gamma^l_{j\beta}u^\beta\p_lu^j+\Gamma_{j\alpha}^lu^\alpha\Gamma_{l\beta}^ju^\beta\right)
-2\frac{u^0}{x^0}\left(\frac{D_up}{\gamma p}+\di\,u\right).
\end{eqnarray}
Replacing terms like $\p_0u^0, (\p_0u^0)^2, (u^0)^2$
in \eqref{eq415} by using suitable Euler equations, %hence obtain an equation of pressure.
after some straightforward computations, we get the identity
\begin{eqnarray}\label{eq432}
&&D_u\left(\frac{D_up}{\gamma p}\right)-\di\,\left(\frac{\grad\, p}{\rho}\right)-C^1_1C^1_2(D u\otimes D u)\nonumber\\
&=&\frac{1}{\gamma
p}\left[\Big(2E-\frac{\gamma+1}{\gamma-1}c^2\Big)\p_0^2p-\frac{c^2}{(x^0)^2}\Delta'p\right]
+\frac{4}{\gamma p x^0}\left(E-\frac{\gamma c^2}{\gamma-1}\right)\p_0p\nonumber\\
&&\qquad-\frac{2}{\gamma p^2}\left(E-\frac{c^2}{\gamma-1}+\frac{c^4}{4\gamma}\frac{1}{E-\frac{c^2}{\gamma-1}}\right)(\p_0p)^2
+\frac{4}{(x^0)^2}\left(E-\frac{c^2}{\gamma-1}\right)\nonumber\\
&&\qquad\qquad\qquad+H_1+H_2+H_3,
\end{eqnarray}
with
\begin{eqnarray}\label{eq431}
H_2&=&\frac{1}{\gamma p}\left[\p_0pG(D_uu+(\grad\,p)/\rho, \p_0)-\rho^{\gamma-1}\p_0p\frac{1}{u^0}D_uA(s)\right.\nonumber\\
&&\left.-\p_0p(u^\alpha\p_\alpha u^0+u^\alpha u^\beta\Gamma^0_{\alpha\beta})+\rho^{\gamma-1}\p_0p\frac{u^\alpha}{u^0}\p_\alpha A(s)\right]\nonumber%\\
\end{eqnarray}
\begin{eqnarray}
&&-\left(\frac{1}{u^0}\right)^2G(D_uu+(\grad\,p)/\rho,\p_0)^2+\frac{2\p_0p}{\rho(u^0)^2}G(D_uu+(\grad\,p)/\rho,\p_0)\nonumber\\
%\end{eqnarray}
%\begin{eqnarray}
&&+\frac{2}{(u^0)^2}(u^\alpha\p_\alpha u^0+u^\alpha u^\beta\Gamma^{0}_{\alpha\beta})G(D_uu+(\grad\,p)/\rho,\p_0)\nonumber\\
&&-\left(\frac{1}{u^0}\right)^2(u^\alpha\p_\alpha u^0+u^\alpha u^\beta\Gamma^{0}_{\alpha\beta})^2-\frac{2\p_0p}{\rho (u^0)^2}(u^\alpha\p_\alpha u^0+u^\alpha u^\beta\Gamma^{0}_{\alpha\beta})\nonumber\\
&&+\frac{2}{x^0}G(D_uu+(\grad\,p)/\rho,\p_0)-\frac{2}{x_0}(u^\alpha\p_\alpha
u^0+u^\alpha u^\beta\Gamma^{0}_{\alpha\beta}),\\
%\end{eqnarray}
%\begin{eqnarray}
\label{eq433}
H_3&=&-G_{\alpha\beta}u^\alpha u^\beta\left[\frac{1}{\gamma p}\p_0^2p+\frac{2}{\gamma px^0}\p_0p+\frac{2}{(x^0)^2}\right.\nonumber\\
&&\left.+\frac{(\p_0p)^2}{\gamma p^2}\left(-1+\frac{c^4}{\gamma}
\frac{1}{2E-\frac{2c^2}{\gamma-1}}\frac{1}{2E-G_{\alpha\beta}u^\alpha
u^\beta-\frac{2c^2}{\gamma-1}}\right)\right].
\end{eqnarray}
Now set
\begin{eqnarray*}
L^0(\bar{\varphi}_1)&=&\frac{2u^0}{x^0}\bar{\varphi}_1,\quad
L^1(\varphi_3)=0,\quad
L^2(\varphi_4)=\left(\frac{\rho^{\gamma-1}}{\gamma p u^0}\p_0p\right)\varphi_4,\\
L^3(\varphi_0)&=&-\left(\frac{\p_0p}{\gamma p}+\frac{2\p_0p}{\rho
(u^0)^2} +\frac{2}{(u^0)^2}(u^\alpha \p_\alpha u^0+u^\alpha u^\beta
\Gamma^0_{\alpha\beta})
+\frac{2}{x^0}\right)G(\varphi_0,\p_0)\nonumber\\&&+\frac{1}{(u^0)^2}G(\varphi_0,\p_0)^2.
\end{eqnarray*}
By multiplying $\gamma p$ to both sides of \eqref{eq432}, and comparing it to \eqref{eq210}, we see  \eqref{eq210} is equivalent to the following second order equation of pressure
\begin{eqnarray}\label{eq438}
N(U)&\triangleq&\left[\Big(2E-\frac{\gamma+1}{\gamma-1}c^2\Big)\p_0^2p-\frac{c^2}{(x^0)^2}\Delta'p\right]
+\frac{4}{x^0}\left(E-\frac{\gamma c^2}{\gamma-1}\right)\p_0p\nonumber\\
&&-\frac{2}{p}\left(E-\frac{c^2}{\gamma-1}+
\frac{c^4}{4\gamma}\frac{1}{E-\frac{c^2}{\gamma-1}}\right)(\p_0p)^2
+\frac{4\gamma p}{(x^0)^2}\left(E-\frac{c^2}{\gamma-1}\right)=F_1(U),
\end{eqnarray}
where
\begin{eqnarray}\label{eq434}
-F_1(U)&=&\gamma p H_3+\left[-\p_0p(u^\alpha\p_\alpha u^0+u^\alpha
u^\beta\Gamma^0_{\alpha\beta})+\rho^{\gamma-1}\p_0p\frac{u^\alpha}{u^0}
\p_\alpha A(s)\right]\nonumber\\
&&-\gamma p\left[\left(\frac{1}{u^0}\right)^2(u^\alpha\p_\alpha u^0+u^\alpha u^\beta\Gamma^{0}_{\alpha\beta})+\frac{2\p_0p}{\rho (u^0)^2}+\frac{2}{x_0}\right](u^\alpha\p_\alpha u^0+u^\alpha u^\beta\Gamma^{0}_{\alpha\beta})\nonumber\\
&&+{u^\alpha\p_\alpha p}\frac{2u^0}{x^0}+\frac{\gamma p}{\rho^2}G^{\alpha \beta}\p_\alpha p\p_\beta\rho+\sum_{(k,j)\ne(0,0)}\left({u^ku^j}\p_{jk}p
+{u^k}\p_ku^j\p_j p-\frac{1}{p}u^ku^j\p_kp\p_jp\right)\nonumber\\
&&-\gamma p\sum_{(l,j)\ne(0,0)}\left(\p_ju^l\p_lu^j+2
\Gamma^l_{j\beta}u^\beta\p_lu^j+\Gamma_{j\alpha}^lu^\alpha\Gamma_{l\beta}^ju^\beta\right).
\end{eqnarray}

\subsubsection{Problem (S2)} From the above computations, by Theorem \ref{thm21}, we see that Problem (S) could be written equivalently as the following Problem (S2), for those unknowns $p,\rho, u$ with regularity as assumed in Theorem \ref{thm21}.

\medskip
\fbox{
\parbox{0.90\textwidth}{
Problem (S2): {Solve \eqref{eq212}\eqref{eq212ber}\eqref{eq213} and \eqref{eq438} in $\mathcal{M}$, subjected to the boundary conditions \eqref{eq31}\eqref{eq32} and \eqref{eq48}.}}}
\medskip

%\begin{remark}
%At first glance Problem (S2) is over-determined, since there is an extra boundary condition \eqref{eq48}. To clarify this, we note that \eqref{eq48} is not independent of the Euler system and other boundary conditions from the above computations.
%\end{remark}

\subsection{Problem (S3)}
%Note that Problem (S2) is a nonlinear boundary value problem.
Since we are dealing with a small
perturbation problem, in this subsection we separate the linear main
terms from the nonlinear equations \eqref{eq48} and \eqref{eq438},
and write them in the form
$\mathcal{L}(U-U_b)=\mathcal{N}(U-U_b)$, where $\mathcal{L}$ is a
linear operator, and $\mathcal{N}(U)$ are certain higher-order terms
defined below. By this way we formulate Problem (S3), which is
equivalent to Problem (S2).  %The idea of this subsection is
%trivial, while the results are crucial for later analysis. So we
%present details of the computations for later reference.

\begin{definition}\label{def401}
For $U=p, E, \rho, s, u$ etc., set $\hat{U}=U-U_b.$ A higher-order
term is an expression containing either

(i) $\hat{U}|_{\p\mathcal{M}}$ and/or its first-order tangential derivatives;

\noindent
or

(ii) product of $\hat{U}$, and/or their
derivatives $D\hat{U}, D^2\hat{U}$, where $D^ku$ is a $k^{th}$-order  derivative of $U$.
\end{definition}

\subsubsection{Linearization of boundary conditions}
Recall that for the background solution, ${p_b}$ solves \eqref{eq36}, so
\eqref{eq48} is equivalent to
\begin{eqnarray}\label{eq436}
&&\p_0(p-p_b) +\frac{2\gamma}{x^0}\left(\frac{ p
(u^0)^2}{(u^0)^2-c^2}-\frac{p_b
(u_b^0)^2}{(u_b^0)^2-c_b^2}\right)=G_1+G_2.
\end{eqnarray}
Using expressions like
%For simplicity, we write in this paragraph that $u_b^0=u_b$.
%There hold the following identities:
\begin{eqnarray}
%\frac{p(u^0)^2}{(u^0)^2-c^2}-\frac{p_bu_b^2}{u_b^2-c_b^2}&=&\frac{u_b^2}{u_b^2-c_b^2}(p-p_b)
%+p_b\left(\frac{(u^0)^2}{(u^0)^2-c^2}-\frac{u_b^2}{u_b^2-c_b^2}\right)\nonumber\\
%&&+\left(\frac{(u^0)^2}{(u^0)^2-c^2}-\frac{u_b^2}{u_b^2-c_b^2}\right)(p-p_b),\label{eq411}\\
%\frac{(u^0)^2}{(u^0)^2-c^2}-\frac{u_b^2}{u_b^2-c_b^2}&=&\frac{u_b^2}
%{(u_b^2-c_b^2)^2}(c^2-c_b^2)
%-\frac{c_b^2}{(u_b^2-c_b^2)^2}((u^0)^2-u_b^2)\nonumber\\
%&&+\frac{u_b^2(c^2-c_b^2)-c_b^2((u^0)^2-u_b^2)}{u_b^2-c_b^2}\left[\frac{1}{(u^0)^2-c^2}-\frac{1}{u_b^2-c_b^2}\right],\label{eq412add2}\\
%(u^0)^2-u_b^2&=&2(E-E_b)-G_{\alpha\beta}u^\alpha u^\beta-\frac{2}{\gamma-1}(c^2-c_b^2),\label{eq413}\\
c^2-c_b^2
=\frac{\gamma-1}{\rho_b}(p-p_b)+\rho_b^{\gamma-1}(A(s)-A(s_b))+O(1)(p-p_b)^2+O(1)(A(s)-A(s_b))^2,
\label{eq414add2}
\end{eqnarray}
where $O(1)$ represents a bounded quantity with the bound depending only on the background solution ${U_b}$ and $|U-{U_b}|$, after some straightforward computations,  \eqref{eq436} could be further written as
\begin{eqnarray}\label{eq416}
\p_0\hat{p}+\gamma_1\hat{p}=G\triangleq G_1+G_2+G_3,
\end{eqnarray}
with $\gamma_1$ a  constant determined by the background solution at
$x^0=r_1$:
\begin{eqnarray}\label{eq417}
\gamma_1%&=&\frac{2\gamma}{r^1}\left\{\frac{u_b^2}{u_b^2-c_b^2}+\frac{2p_bE_b}{(u_b^2-c_b^2)^2}\frac{\gamma-1}{\rho_b}\right\}\nonumber\\
\triangleq\frac{2}{x^0}\frac{\gamma M_b^4-M_b^2+2}{(M_b^2-1)^2}>0,
\end{eqnarray}
%which is positive since $\gamma>1$,
and
\begin{eqnarray}\label{eq418}
G_3&=&-\frac{2\gamma}{x^0}\left\{\left(\frac{(u^0)^2}{(u^0)^2-c^2}-\frac{u_b^2}{u_b^2-c_b^2}\right)\hat{p}
+p_b\frac{u_b^2(c^2-c_b^2)-c_b^2((u^0)^2-u_b^2)}{u_b^2-c_b^2}
\left[\frac{1}{(u^0)^2-c^2}-\frac{1}{u_b^2-c_b^2}\right]\right.\nonumber\\
&&\left.-\frac{p_bc_b^2}{(u_b^2-c_b^2)^2}[\underline{2(E-E_b)}-G_{\alpha\beta}u^\alpha
u^\beta]+\frac{2p_bE_b}{(u_b^2-c_b^2)^2}\left[O(1)(|\hat{p}|^2+|\widehat{A(s)}|^2)
\right.\right.\nonumber\\
&&\left.\left.+\underline{\rho_b^{\gamma-1}({A(s)-A(s_b)})}%+\gamma
%[p^{1-\frac{1}{\gamma}}-p_b^{1-\frac{1}{\gamma}}][A(s)^{\frac{1}{\gamma}}-A(s_b)^{\frac{1}{\gamma}}]
\right]\right\}.
\end{eqnarray}
We note that $G_2, G_3$ are higher-order terms (the terms with
underlines are given by boundary data, so fulfill the item $(\rmnum{1})$ in
Definition \ref{def401}), while $G_1$ depends on the boundary values
of $u^1, u^2$ on $M_1$. Boundary condition \eqref{eq416} is an
equivalent form of \eqref{eq215}.

\subsubsection{Linearization of pressure's equation}
For \eqref{eq438}, note that $N(U_b)=0$,  we may get a linear operator $L$ and write $L(\hat{U})-F_2(U)=N(U)-N(U_b)$, with $F_2(U)$ a higher-order term. Then  \eqref{eq438} becomes $L(\hat{U})=F_1(U)+F_2(U)$. %To specify $L$, we compute term by term (recall that $|u|^2=G(u,u)$ and $u_b$ is $u_b^0$):
In fact, using \eqref{eq414add2}, by setting
$t={u_b^2}/{c_b^2}=M_b^2\in (0,1),$
direct computation yields that \eqref{eq438} can be written as
\begin{eqnarray}\label{eq441}
L(\hat{p})&\triangleq&-\frac{1}{(x^0)^2}\Delta'\hat{p}+(t-1)\p_0^2\hat{p}+\frac{4}{x^0}
b(t)\p_0\hat{p}+\frac{1}{(x^0)^2}e(t)\hat{p}+\frac{\rho_b}{
(x^0)^2}d_1(t)\hat{E}
+\frac{\rho_b^\gamma }{(x^0)^2}d_2(t)\widehat{A(s)}\nonumber\\
&=&F\triangleq \frac{1}{c_b^2}(F_1+F_2),
\end{eqnarray}
with
\begin{eqnarray}
b(t)&\triangleq&%\frac{1+\gamma t^2}{t-1}+\frac{t}{2}-1=
\frac{1}{2(t-1)}[(1+2\gamma)t^2-3t+4],\label{eq442}\\
e(t)&\triangleq&\frac{2}{(t-1)^3}\left[\gamma(1+2\gamma)t^4+
(-4\gamma^2+2\gamma-3)t^3+(14-7\gamma)t^2-19t+6\right], \label{eq443}\\
d_1(t)&\triangleq&\frac{4}{(t-1)^3}\Big((2\gamma-3)t^2+8t-3\Big),\\
d_2(t)&\triangleq&\frac{-2}{\gamma-1}\frac{1}{(t-1)^3}[2+(\gamma-1)t][(2\gamma-3)t^2+8t-3],\label{eq444add}
\end{eqnarray}
and
\begin{eqnarray}\label{eq447add}
-F_2
&=&\frac{4\gamma\hat{p}}{(x^0)^2}\left(\hat{E}-\frac{1}{\gamma-1}(c^2-c_b^2)\right)
+\frac{4}{x^0}\p_0\hat{p}\left(\hat{E}-\frac{\gamma}{\gamma-1}(c^2-c_b^2)\right)\nonumber\\
&&-(c^2-c_b^2)\frac{1}{(x^0)^2}\Delta'\hat{p}+\left[2\hat{E}-\frac{\gamma+1}{\gamma-1}(c^2-c_b^2)\right]\p_0^2\hat{p}\nonumber\\
%\end{eqnarray}
%\begin{eqnarray}
&&-\frac{u_b^2}{p_b}(\p_0\hat{p})^2+\p_0(p+p_b)\p_0\hat{p}\left[\frac{u_b^2}{p_b}-
\frac{|u|^2}{p}\right]+(\p_0p_b)^2\frac{\hat{p}}{p_b}\left[\frac{|u|^2}{p}-\frac{u_b^2}{p_b}\right]\nonumber\\
%\end{eqnarray}
%\begin{eqnarray}
&&-\frac{(\p_0p_b)^2}{\gamma u_b^2}\left[\left(\frac{c^2+c_b^2}{p}-\frac{2c_b^2}{p_b}\right)(c^2-c_b^2)
-\frac{c_b^4}{p_b}\hat{p}\left(\frac{1}{p}-\frac{1}{p_b}\right)\right]
-\frac{c_b^4}{\gamma p_bu_b^2}(\p_0\hat{p})^2\nonumber\\
&&-\frac{(\p_0p_b)^2}{\gamma}\left(\frac{1}{|u|^2}-\frac{1}{u_b^2}\right)
\left[\left(\frac{c^4}{p}-\frac{c_b^4}{p_b}\right)+\frac{2c_b^4}{p_bu_b^2}
\left(\frac{c^2-c_b^2}{\gamma-1}-\hat{E}\right)\right]-\frac{\p_0(p+p_b)\p_0\hat{p}}{\gamma}\left(\frac{c^4}{p|u|^2}-
\frac{c_b^4}{p_bu_b^2}\right)\nonumber\\
%\end{eqnarray}
%\begin{eqnarray}
&&+\left[-\frac{4\gamma}{\gamma-1}\frac{p_b}{(x^0)^2}-\frac{4\gamma}{\gamma-1}\frac{\p_0p_b}{x^0}
-\frac{\gamma+1}{\gamma-1}\p_0^2p_b+\frac{2}{\gamma-1}\frac{(\p_0p_b)^2}{p_b}
-\frac{2c_b^2}{\gamma p_b u_b^2}(\p_0p_b)^2\left(1+\frac{1}{\gamma-1}\frac{c_b^2}{u_b^2}\right)\right]\nonumber\\
&&\times\left[O(1)\hat{p}^2+O(1)(A(s)-A(s_b))^2\right].
\end{eqnarray}
We easily see that \eqref{eq441} is an elliptic equation of (perturbed) pressure.

\subsubsection{Problem (S3)}
We now reformulate Problem (S) equivalently as the following Problem (S3).

\medskip
\fbox{
\parbox{0.90\textwidth}{
Problem (S3): Solve functions $U=(E, A(s), p, u'=u^\alpha\p_\alpha)$ that satisfying the following problems \eqref{eq444}--\eqref{eq447}.}}

\begin{eqnarray}
&&\begin{cases}\label{eq444}
D_uE=0&\text{in}\quad \mathcal{M},\\
E=E_1(x')&\text{on}\quad M_1;
\end{cases}\\
&&\begin{cases}\label{eq445}
D_uA(s)=0&\text{in}\quad \mathcal{M},\\
A(s)=A(s_1)(x')&\text{on}\quad M_1;
\end{cases}\\
&&\begin{cases}\label{eq446}
L(\hat{p})=F(U)& \text{in}\quad \mathcal{M},\\
\hat{p}=p_0(x')-p_b &\text{on}\quad M_0,\\
\p_0\hat{p}+\gamma_1\hat{p}=G(U) &\text{on}\quad M_1;
\end{cases}\\
&&\begin{cases}\label{eq447}
G(D_uu+\frac{\grad\,p}{\rho},\p_\alpha)=0& \text{in}\quad \mathcal{M},\ \ \alpha=1,2,\\
u'=u'_1(x') &\text{on}\quad M_1.
\end{cases}
\end{eqnarray}
We note that  \eqref{eq446} is a mixed boundary value problem of a second order elliptic equation, while \eqref{eq444}, \eqref{eq445} and \eqref{eq447} are Cauchy problems of transport equations. In fact, we could write the equation in \eqref{eq447} in a local spherical coordinates as
\begin{eqnarray}\label{eq452add}
u^j\p_ju^\alpha+u^ju^m\Gamma^\alpha_{jm}=-\frac{1}{\rho}\p_i p G^{i\alpha}.
\end{eqnarray}
For $\alpha=1$, it reads
\begin{eqnarray}\label{eq448}
u^j\p_ju^1+\frac{2u_b^0}{x^0}u^1=-\frac{1}{(x^0)^2\rho}\p_1p+h_1(U),
\end{eqnarray}
with
\begin{eqnarray}\label{eq454add}
h_1(U)=(\sin x^1\cos x^1)(u^2)^2+\frac{2}{x^0}(u_b^0-u^0)u^1.
\end{eqnarray}
For $\alpha=2$, it reads, in the local coordinates used above, that
$$u^j\p_ju^2+\frac{2}{x^0}u^0u^2=-2\cot x^1 u^1u^2-\frac{1}{\rho}\frac{1}{(x^0\sin x^1)^2}\p_2p.$$
This equation has an artificial  singularity when $\sin x^1=0$. It is obvious that we can avoid this by using another local (spherical) coordinates. So by symmetry of $u^1$ and $u^2$, it suffices to solve \eqref{eq448} and obtain an estimate.

\subsection{Proof of Theorem \ref{thm41}}\label{sec43}
For $k=2,3$, and $U=(p,s,E, u'=u^1\p_1+u^2\p_2)$, suppose that $p\in C^{k,\alpha}(\overline{\mathcal{M}})$ and $s,E, u^\beta\p_\beta\in C^{k-1,\alpha}(\overline{\mathcal{M}})$. We define the  norm
\begin{eqnarray}\label{eq452}
\norm{U}_k\triangleq \norm{p}_{C^{k,\alpha}(\overline{\mathcal{M}})}+
\norm{s}_{C^{k-1,\alpha}(\overline{\mathcal{M}})}+
\norm{E}_{C^{k-1,\alpha}(\overline{\mathcal{M}})}+
\sum_{\beta=1}^2\norm{u^\beta}_{C^{k-1,\alpha}(\overline{\mathcal{M}})}.
\end{eqnarray}
Let $K$ be a positive number to be chosen. We define $X_{K\varepsilon}$ to be the (non-empty) set of functions $U$ so that \eqref{eq31} and \eqref{eq32} hold, and
%\begin{eqnarray}
 $\norm{U-U_b}_3\le K\varepsilon.$
%\end{eqnarray}
To prove Theorem \ref{thm41}, we construct a mapping $\mathcal{T}$ on $X_{K\varepsilon}$ for suitably chosen $K$ and $\varepsilon$, and show that it contracts under the norm $\norm{\cdot}_2$. Then by a Banach fixed point theorem, $\mathcal{T}$ has uniquely one fixed point $U\in X_{K\varepsilon}$, which is exactly a solution to Problem (S3).

\subsubsection{Construction of mapping $\mathcal{T}$}
For any $U\in X_{K\varepsilon}$, by the following process we define a mapping $\mathcal{T}: U\mapsto \tilde{U}$.  Set $\hat{U}=\tilde{U}-U_b$. Then we only need to determine $\hat{U}$. We also use $C$ to denote generic positive constants which might be different in different lines.
\medskip

\paragraph{\bf Determination of $\tilde{E}$ and $\widetilde{A(s)}$}

Noting that $s_b$ and  $E_b$ are constants, we solve the unknowns $\hat{E}, \widehat{A(s)}$ from the following Cauchy problems of linear transport equations, where the velocity field $u$ is given as part of $U\in X_{K\varepsilon}$:
\begin{eqnarray}
\begin{cases}\label{eq454}
D_u\hat{E}=0&\text{in}\quad \mathcal{M},\\
\hat{E}=E_1-E_b&\text{on}\quad M_1;
\end{cases}\quad
\begin{cases}%\label{eq455}
D_u\widehat{A(s)}=0&\text{in}\quad \mathcal{M},\\
\widehat{A(s)}=A(s_1)-A(s_b)&\text{on}\quad M_1.
\end{cases}
\end{eqnarray}
Since $u\in C^{2,\alpha}(\overline{\mathcal{M}})$, and recall that ${E_1-E_b}, s_1-s_b\in C^{3,\alpha}(M_1)$,  by Theorem \ref{thm61}, we have

\begin{lemma}\label{lem41}
There are uniquely $\hat{E}, \widehat{A(s)}\in C^{2,\alpha}(\overline{\mathcal{M}})$ solve \eqref{eq454}. In addition,
\begin{eqnarray}\label{eq456}
\norm{\hat{E}}_{C^{2,\alpha}(\overline{\mathcal{M}})}+\norm{\widehat{A(s)}}_{C^{2,\alpha}
(\overline{\mathcal{M}})}\le C_1\left(\norm{\hat{E}}_{C^{2,\alpha}({M}_1)}+\norm{\widehat{A(s)}}_{C^{2,\alpha}
({M_1})}\right)\le C_1\varepsilon,
\end{eqnarray}
with $C_1$  a positive constant depending only on $U_b$. The second inequality holds provided that \eqref{eq42add} is valid.
\end{lemma}

Once we solved $\hat{E}$ and $\widehat{A(s)}$, we get $\tilde{E}=E_b+\hat{E}, \ \widetilde{A(s)}=A(s_b)+\widehat{A(s)}.$

\paragraph{\bf Determination of $\tilde{p}$}
Now consider the following mixed boundary value problem of $\hat{p}$ (comparing to problem \eqref{eq446}):
\begin{eqnarray}\label{eq460}
\begin{cases}
\mathcal{L}(\hat{p})\triangleq-\frac{1}{(x^0)^2}\Delta'\hat{p}+(t-1)\p_0^2\hat{p}+\frac{4}{x^0}
b(t)\p_0\hat{p}+\frac{1}{(x^0)^2}e(t)\hat{p}\\
\qquad\qquad=-\frac{\rho_b}{(x^0)^2}d_1(t)\hat{E}
-\frac{\rho_b^\gamma }{(x^0)^2}d_2(t)\widehat{A(s)}+F(U),\\
\hat{p}=p_0-p_b \qquad\text{on}\quad M_0,\\
\p_0\hat{p}+\gamma_1\hat{p}=G(U) \qquad\text{on}\quad M_1.
\end{cases}
\end{eqnarray}
Note here that for the two terms $\frac{\rho_b}{(x^0)^2}d_1(t)\hat{E}$ and  $\frac{\rho_b^\gamma }{(x^0)^2}d_2(t)\widehat{A(s)}$ on the right-hand side, we take $\hat{E}$ and $\widehat{A(s)}$ to be the functions solved from Lemma \ref{lem41}.

We now specify the nonhomogeneous  term $F(U)$. In $F_2$ (see \eqref{eq447add}), all $\hat{U}$ should be replaced by $U-U_b$ (recall that the $U\in X_{K\varepsilon}$ has been fixed). So by direct computations we get that
%\begin{eqnarray*}
$\norm{F_2}_{C^{1,\alpha}(\overline{\mathcal{M}})}\le CK^2\varepsilon^2.$
%\end{eqnarray*}
In the expression of $F_1$ (see \eqref{eq434}), we take all $U$ to be the given one.  So by the smallness of $u^\alpha$, we have
%\begin{eqnarray*}
$\norm{\gamma pH_3}_{C^{1,\alpha}(\overline{\mathcal{M}})}\le CK^2\varepsilon^2, \norm{F_1}_{C^{1,\alpha}(\overline{\mathcal{M}})}\le CK^2\varepsilon^2.$
%\end{eqnarray*}
Therefore we obtain, for a positive constant $C$ depending only on $U_b$, that
\begin{eqnarray}\label{eq461}
\norm{F(U)}_{C^{1,\alpha}(\overline{\mathcal{M}})}\le CK^2\varepsilon^2
\end{eqnarray}

Next we specify the term $G(U)$ in boundary conditions.
For $G_1$ (see \eqref{eq49}), the $u^\alpha$ should be the boundary conditions $(u'_1)^\alpha$ (hence belong to $C^{3,\alpha}(M_1)$), and the others are replaced by the given $U$. So by \eqref{eq42add} we have
%\begin{eqnarray*}
$\norm{G_1}_{C^{2,\alpha}(M_1)}\le C\varepsilon.$
%\end{eqnarray*}
For the term $G_2$ (see \eqref{eq410}), $u^\alpha$ and $A(s), E$ should be the given boundary data $(u'_1)^\alpha$ and $A(s_1), E_1$ respectively, while the others are replaced by the given $U\in X_{K\varepsilon}$. Hence it still lies in $C^{2,\alpha}(M_1)$ and we have
%\begin{eqnarray*}
$\norm{G_2}_{C^{2,\alpha}(M_1)}\le C\varepsilon^2+CK\varepsilon^2.$
%\end{eqnarray*}
For the expression of $G_3$ (see \eqref{eq418}), except the
underline terms are replaced by the given boundary data, all the other $U$ can be taken as the given $U$ in
$X_{K\varepsilon}$, and it easily follows that
%\begin{eqnarray*}
$\norm{G_3}_{C^{2,\alpha}(M_1)}\le CK^2\varepsilon^2+C\varepsilon.$
%\end{eqnarray*}
So in all, we obtain
\begin{eqnarray}\label{eq462}
\norm{G(U)}_{C^{2,\alpha}(M_1)}\le CK^2\varepsilon^2+C\varepsilon,
\end{eqnarray}
if we choose later, without loss of generality, that $K>1$.

Under the assumptions of the Theorem \ref{thm41}, {\it the coefficient $e(t)$ is nonnegative.} Recall also that $\gamma_1>0$ (see \eqref{eq417}). So by the standard theory of second order elliptic equations with Dirichlet and oblique derivative conditions (\cite[Chapter 6]{GT}), we have

\begin{lemma}\label{lem42}
There is uniquely one solution $\hat{p}\in C^{3,\alpha}(\overline{\mathcal{M}})$ to problem  \eqref{eq460}. In addition, there is a constant $C$ depending only on the background solution $U_b$ and $\mathcal{M}$, so that there holds
\begin{eqnarray}\label{eq463}
\norm{\hat{p}}_{C^{3,\alpha}(\overline{\mathcal{M}})}&\le& C\Big(\norm{G(U)}_{C^{2,\alpha}(M_1)}+\norm{p_0-p_b}_{C^{3,\alpha}(M_0)}+
\norm{F(U)}_{C^{1,\alpha}(\overline{\mathcal{M}})}\nonumber\\
&&
+\norm{\widehat{A(s)}}_{C^{1,\alpha}(\overline{\mathcal{M}})}+
\norm{\hat{E}}_{C^{1,\alpha}(\overline{\mathcal{M}})}\Big)
\nonumber\\
&\le& C(K^2\varepsilon^2+\varepsilon).
\end{eqnarray}
\end{lemma}
Hence we obtain that $\tilde{p}=\hat{p}+p_b$.

\begin{remark}\label{rm42}
We see that \eqref{eq41} is used to guarantee that problem \eqref{eq460} has uniquely one solution. A detailed study of the spectrum of the operator $\mathcal{L}$ with the homogeneous boundary conditions (like what we do in Section 5 by separation of variables) would definitely refine it. However, for our present purpose of clarifying the basic ideas, we are content ourselves with \eqref{eq41}.
\end{remark}

\paragraph{\bf Determination of $\tilde{u}'$}
Now we solve $\tilde{u}^\alpha$ ($\alpha=1,2$) from the transport equations $G(D_uu+\frac{\grad\,p}{\rho},\p_\alpha)=0$. For $\alpha=1$, we have the problem
\begin{eqnarray}\label{eq466}
\begin{cases}
u^j\p_j\tilde{u}^1+\frac{2u_b^0}{x^0}\tilde{u}^1=-\frac{1}{(x^0)^2\rho}\p_1\tilde{p}
+h_1(U)&\text{in}\quad\mathcal{M},\\
\tilde{u}^1=(u'_1)^1 &\text{on}\quad M_1.
\end{cases}
\end{eqnarray}
Here $(u'_1)^1$ is the given boundary data; $\tilde{p}$ on the right-hand side is given by Lemma \ref{lem42}, while  $u$ and $U$ in
\begin{eqnarray}\label{eq465}
h_1(U)=(\sin x^1\cos x^1)(u^2)^2+\frac{2}{x^0}(u_b^0-u^0)u^1
\end{eqnarray}
are the one we had fixed in $X_{K\varepsilon}$.

We have the following lemma due to Theorem \ref{thm61}.
\begin{lemma}\label{lem43}
There is uniquely one solution $\tilde{u}^1\in C^{2,\alpha}(\overline{\mathcal{M}})$ to problem \eqref{eq466}. In addition,
\begin{eqnarray}\label{eq467add}
\norm{\tilde{u}^1}_{C^{2,\alpha}(\overline{\mathcal{M}})}\le \norm{u'^1_1}_{C^{2,\alpha}({M}_1)}+C\Big(\norm{\tilde{p}}_{C^{3,\alpha}(\overline{\mathcal{M}})}+
\norm{h_1(U)}_{C^{2,\alpha}(\overline{\mathcal{M}})}\Big)\le C\Big(K^2\varepsilon^2+\varepsilon\Big)
\end{eqnarray}
for a positive constant $C$ depending only on the background solution and $r_0, r_1$.
\end{lemma}

As explained earlier, we could change the coordinates to solve $\tilde{u}^2$ and similar results hold. Then from $U\in X_{K\varepsilon}$ we obtain uniquely one  $\tilde{U}=(\tilde{p}, \tilde{s}, \tilde{E}, \tilde{u}^\alpha\p_\alpha)$, and the estimates \eqref{eq456}\eqref{eq463}\eqref{eq467add} yield that
%\begin{eqnarray*}
$\norm{\tilde{U}-U_b}_3\le C(K^2\varepsilon^2+\varepsilon).$
%\end{eqnarray*}
By choosing $K=\max\{2C, 1\}$, and $\varepsilon_0\le \min\{1/K^2,1\}$, we  have
%\begin{eqnarray*}
$\norm{\tilde{U}-U_b}_3\le K\varepsilon$
%\end{eqnarray*}
for all  $\varepsilon\le \varepsilon_0.$
Hence we proved that $\tilde{U}\in X_{K\varepsilon}$, and the mapping $\mathcal{T}: X_{K\varepsilon}\to X_{K\varepsilon}$ is well-defined.

\subsubsection{Contraction of the mapping $\mathcal{T}$}
For $U^{(1)}, U^{(2)}\in X_{K\varepsilon}$, set $\tilde{U}^{(k)}=\mathcal{T}(U^{(k)})$, $k=1,2$. We wish to show that if $\varepsilon_0$ is further small, then
\begin{eqnarray}\label{eq467}
\norm{\tilde{U}^{(1)}-\tilde{U}^{(2)}}_2\le\frac12\norm{U^{(1)}-U^{(2)}}_2.
\end{eqnarray}
The idea to establish \eqref{eq467} is to consider the problems satisfied by $\hat{U}=\tilde{U}^{(1)}-\tilde{U}^{(2)}$. Since the process is standard once we understand the definition of $\mathcal{T}$, and is quite similar (but simpler) to that described in Section \ref{sec532}, we omit the details here.

Then by Banach fixed point theorem, $\mathcal{T}$ has one and only one fixed point,  say, $U$, in $X_{K\varepsilon}$. By the construction of the mapping $\mathcal{T}$, the fixed point is a solution to Problem (S3). On the contrary, for a solution to Problem (S3) which lies in $X_{K\varepsilon}$, it must be a fixed point of $\mathcal{T}$.  The proof of  Theorem \ref{thm41} is completed.

\section{Stability of spherically symmetric transonic shocks}\label{sec5}

In this section we study the stability of spherically symmetric transonic shocks under multidimensional perturbations of boundary conditions, which was treated in \cite{BaeFeldman} by using the ``non-isentropic potential flow model", and in \cite{CY2013} by considering the full Euler system.  In \cite{CY2013} only uniqueness was proved; namely, if the perturbations of the upcoming supersonic flow and the back pressure are small, then there will be only one transonic shock solution in some  function space, provided that it exits. We will show below that in a suitable class of functions, a disturbed transonic shock solution does exist and is unique in a neighborhood of the background solution $U_b$, provided that the S-condition, which also occurred in \cite{CY2013}, is valid. The main difficulty is to decompose the Rankine-Hugoniot conditions (R-H conditions) on shock-front in a way compatible to the decomposition of the  Euler system stated in Theorem \ref{thm21}. We will discover many subtle intrinsic structures lying beneath the related free boundary problem.

In the following we firstly review the formulation of the transonic shock problem (T) and the existence of background solutions, and state the main theorem, {\it i.e.} Theorem \ref{thm501}. Then we start to decompose the R-H conditions, and formulate Problem (T) step by step into the more tractable problems (T1), (T2), (T3), (T4), each of which is equivalent to Problem (T). Finally, Theorem \ref{thm501} is proved by applying a nonlinear iteration method to the nonlinear free boundary problem (T4).

%----------2014-7-14 to check.
\subsection{Transonic shock problem (T) and main result}
Let $\mathcal{M}=(r_0, r_1)\times S^2$ be a spherical shell with entry $M_0=\{r_0\}\times S^2$ and exit $M_1=\{r_1\}\times S^2$. We use $U=(u,p,s)$ to represent the state of the gas flows in $\mathcal{M}$.
%Recall that in a local spherical coordinates $x=(x^0, x')=(x^0, x^1, x^2)$, we  write the velocity field as $u=u^0\p_0+u'$, with $u'$ the tangential velocity.
Suppose that
\begin{eqnarray}\label{eq501}
S^\psi=\{(x^0, x')\in\mathcal{M}: x^0=\psi(x'), \ x'\in S^2\}
\end{eqnarray}
is a surface, where $\psi: S^2\to\mathcal{M}$ is a $C^1$ function. The normal vector field on $S^\psi$ is given by
\begin{eqnarray}\label{eq502add}
n=(\p_\alpha \psi G^{\alpha\beta}\p_\beta-G^{00}\p_0)|_{S^\psi}.
\end{eqnarray}
We also set $M_\psi^-=\{x\in\mathcal{M}: x^0<\psi(x'), \ x'\in S^2\}$ to be the {\it supersonic region}, and
$M_\psi^+=\{x\in\mathcal{M}: x^0>\psi(x'), \ x'\in S^2\}$ to be the {\it subsonic region}.

\begin{definition}[Transonic shock]
Let $\psi\in C^1(S^2)$ and $U^\pm\in C^1(\mathcal{M}_\psi^\pm)\cap C(\overline{\mathcal{M}_\psi^\pm})$. We say that $U=(U^-, U^+; \psi)$ is a {\it transonic shock solution}, if
\begin{itemize}
\item[1)] $U^\pm$ solve the Euler system \eqref{eq104} in $\mathcal{M}_\psi^\pm$ in the pointwise sense;
\item[2)] $U^-$ is supersonic, and $U^+$ is subsonic;
\item[3)] the following R-H conditions hold across $S^\psi$:
\begin{eqnarray}\label{eq502}
[[G(u,n)\rho u+pn]]=0,\quad
[[G(u,n)\rho]]=0,\quad
[[G(u,n)\rho E]]=0,
\end{eqnarray}
where $[[f(U)]]\triangleq f(U^+|_{S^\psi})-f(U^-|_{S^\psi})$ denotes the jump of a quantity $f(U)$ across $S^\psi$;
\item[4)] there holds the following physical entropy condition
\begin{eqnarray}\label{eq503}
[[p]]=p^+|_{S^\psi}-p^-|_{S^\psi}>0.
\end{eqnarray}
\end{itemize}
\end{definition}

By the definition we infer that a transonic shock solution is a weak entropy solution of the steady Euler system ({\it cf.} Section 4.3 and Section 4.5 in \cite{Da}).

To formulate the transonic shock problem, we also need to specify boundary conditions.
Since the flow $U^-$ is supersonic near the entry $M_0$, the following Cauchy data should be given:
\begin{eqnarray}\label{eq504}
U=U_0^-(x') \qquad \text{on}\quad M_0.
\end{eqnarray}
Here we also require that $(u^0)_0^->c^-_0$ to make sure the steady Euler system is symmetric hyperbolic in the positive $x^0$-direction on $M_0$.
On the exit $M_1$, it is known now that one and only one boundary condition is necessary. From the physical considerations, as in the studies of subsonic flows, we propose the pressure
\begin{eqnarray}\label{eq505}
p=p_1(x')\qquad \text{on}\quad M_1.
\end{eqnarray}
It turns out that this is also a mathematically challenging boundary condition ({\it cf.} \cite{Yuan2007} and discussions at the end of \cite{yuan2006}).

\medskip
\fbox{
\parbox{0.90\textwidth}{
Problem (T): Find a transonic shock solution in $\mathcal{M}$ which satisfies the  boundary conditions \eqref{eq504} and \eqref{eq505} pointwisely.}}

\subsubsection{Background solution} \label{sec511}
The existence of  spherically symmetric transonic shock solutions $U_b=(U_b^-, U_b^+; r_b)$ (which are called as {\it background solutions} in the sequel) to Problem (T) has been established in \cite{Yu3} (see also \cite{CY2013}).  For given
$[r_0, r_1]$, we note that  $U_b^+$ is
determined by, and actually depends analytically on, the five parameters $(\gamma, r_b, p_b^+(r_b),
\rho_b^+(r_b), t(r_b))$ with $t(r_b)=(M_b^+)^2(r_b)\in(0,1)$, $\gamma>1,\
r_b\in(r^0,r^1), \ p_b^+(r_b)>0,$ and $\rho_b^+(r_b)>0$.  As a consequence, all the coefficients on the left-hand side of \eqref{eq441}, such as $e(t)$, $\rho_b^\gamma d_2(t)$, depend analytically on these
parameters. In addition, it is useful to note that the spherically symmetric subsonic flow $U_b^+$ could be extended to $[r_b-h^\sharp, r_1]\times S^2$, for a small constant $h^\sharp$ depending solely on these parameters ({\it cf.} \cite[p.323]{Yu3}).

\subsubsection{The main result} \label{sec512}
We now state the forth main result of this paper.
\begin{theorem}\label{thm501}
Let $U_b$ and $\gamma$ satisfy the S-Condition (see Definition \ref{def52} and Lemma \ref{lem54}), and  $\alpha\in(0,1)$. There exist $\varepsilon_0$ and $C_*$ depending
only on $U_b$ and $\gamma, \alpha, r_0, r_1$ such that, if the upcoming supersonic
flow $U_0^-$ on $M_0$ and the back pressure on $M_1$ satisfy
\begin{eqnarray}\label{th101}
\norm{U_0^--U_b^-}_{C^{4}(M_0)}\le\varepsilon\le\varepsilon_0,&\\
\norm{p_1-p_b^+}_{C^{3,\alpha}(M_1)}\le\varepsilon\le\varepsilon_0,&\label{eq58add}
\end{eqnarray}
then there exists a transonic shock
solution $U=(U^-,U^+; \psi)$ to Problem (T), so that  $\psi\in
C^{4,\alpha}({S}^2)$, $U^-\in C^{4}(\overline{\mathcal{M}_\psi^-})$, $p^+\in C^{3,\alpha}(\overline{\mathcal{M}_\psi^+})$, $u^+, \rho^+, s^+\in C^{2,\alpha}(\overline{\mathcal{M}_\psi^+})$, $u^+|_{S^\psi}, \rho^+|_{S^\psi}, s^+|_{S^\psi}\in C^{3,\alpha}(S^2)$, and
\begin{eqnarray}\label{th102}
\norm{\psi-r_b}_{C^{4,\alpha}({S}^2)}\le C_*\varepsilon,&\\
\label{th103}
\norm{U^--U^-_b}_{C^4(\overline{\mathcal{M}_\psi^-})}\le C_*\varepsilon,&\\
\label{th104}
\norm{\left.U^+\right|_{S^\psi}-\left.U_b^+\right|_{S^\psi}}_{C^{3,\alpha}(S^2)}+\norm{U^+-U^+_b}_{3}\le C_*\varepsilon.&
\end{eqnarray}
Furthermore, such solution is unique in the class of functions $\psi, U^-, U^+$ with
\begin{eqnarray}\label{th105}
\norm{\psi-r_b}_{C^{3,\alpha}({S}^2)}\le C_*\varepsilon,&\\
\label{th106}
\norm{U^--U^-_b}_{C^4(\overline{\mathcal{M}_\psi^-})}\le C_*\varepsilon,&\\
\label{th107}
\norm{\left.U^+\right|_{S^\psi}-\left.U_b^+\right|_{S^\psi}}_{C^{2,\alpha}(S^2)}+\norm{U^+-U^+_b}_{2}\le C_*\varepsilon.&
\end{eqnarray}
The norm $\norm{\cdot}_k$ is defined by \eqref{eq452}, with $\mathcal{M}$ there replaced by ${\mathcal{M}_\psi^+}$.
\end{theorem}

\subsubsection{Existence of supersonic flow}\label{sec513}
The existence and uniqueness of supersonic
flow $U^-$ in $\mathcal{M}=(r_0,r_1)\times S^2$ subjected to the initial data
$U_0^-$ satisfying \eqref{th101} follow from the
theory of  classical solutions of the Cauchy problem of
quasi-linear symmetric hyperbolic systems if $\varepsilon_0$ is
sufficiently small and $|r_1-r_0|$ is not large ({\it cf.} \cite{BS2007}). Furthermore, one can obtain that
\begin{eqnarray}\label{55015}
\norm{U^--U^-_b}_{C^{4}(\overline{\mathcal{M}})}\le C_1\varepsilon,
\end{eqnarray}
where $C_1>0$ and $\varepsilon_0>0$  depend solely on $U_b^-(r_0)$
and $r_1-r_0$. This implies \eqref{th103}. So Problem (T) is actually a one-phase free boundary problem, for which the free boundary ({\it i.e.} the shock-front) $S^\psi$ and the subsonic flow $U^+$ are to be solved. For simplicity, from now on we write $U^+$ as $U$.

\subsection{Reformulation of R-H conditions}
For transonic shock problem, the R-H conditions represent a nontrivial nonlinear structure of a discontinuous flow field. In this subsection we  decompose  them to find their roles in determination of  shock-front and subsonic flows.  For the sake of completeness, we  shall take  some computations  from \cite[pp.2520-2523]{CY2013}.

\subsubsection{Decomposition of R-H conditions}
Consulting \eqref{eq502add}, we see the normal 1-form on the shock-front \eqref{eq501} is given by $\bar{n}=\dd\psi-\dd x^0,$
and the R-H conditions \eqref{eq502} could be written in terms of differential forms as
\begin{eqnarray}
\left[\left[\bar{n}(u)\rho\bar{u}+p\bar{n}\right]\right]=0,&\label{eq51}\\
\left[\left[\bar{n}(u)\rho\right]\right]=0,&\label{eq52}\\
\left[\left[\bar{n}(u)\rho E\right]\right]=0.&\label{eq53}
\end{eqnarray}
Here $\bar{n}(u)$ is the action of the 1-form $\bar{n}$ on the vector field $u$.%and recall that $\bar{u}$ denotes the $1$-form associated with $u$.

By the definition of shock-front, the mass flux $m\triangleq-\bar{n}(u)\rho|_{S^\psi}=-\bar{n}(u^-)\rho^-|_{S^\psi}\ne0$. So from \eqref{eq52} and \eqref{eq53} we infer that $E^-|_{S^\psi}=E^+|_{S^\psi}$. Combining this with \eqref{eq21}, one concludes that the Bernoulli constant $E$  is invariant along flow trajectories, even across a shock-front; so it is actually determined by the upcoming supersonic flow $U_0^-$.

Set
%\begin{eqnarray*}
$u=u_0+u'\triangleq u^0\,\p_0+u^\alpha\,\p_\alpha,\ \
\bar{u}=\bar{u}^0+\bar{u}'\triangleq u^0G_{00}\, \dd x^0+u^\alpha
G_{\alpha\beta}\, \dd x^\beta.$
%\end{eqnarray*}
Then $\bar{n}(u)=\dd\psi(u')-u^0$, and
$m=\rho u^0-\dd\psi(\rho u').$
So \eqref{eq52} and \eqref{eq53} may be written as
\begin{eqnarray}
[[m]]=0,\quad [[E]]=0,\label{eq54}
\end{eqnarray}
while  \eqref{eq51} is decomposed to be
\begin{eqnarray}
\left.[[mu^0+p]]\right.&=&0,\label{eq55}\\
\left.[[p\,\dd\psi-m\bar{u}']]\right.&=&0.\label{eq56}
\end{eqnarray}

If $[[p]]>0$ (which is guaranteed by the physical entropy condition satisfied by the background solution, and the small perturbation estimate \eqref{th104} to be established),  from \eqref{eq56} we  solve that
\begin{eqnarray}\label{eq57}
\dd\psi=\omega\triangleq \left.\frac{m[[\bar{u}']]}{[[p]]}\right|_{S^\psi}
=\mu_0\psi^*(\bar{u}')+g_0,
\end{eqnarray}
with $\mu_0$ a positive constant, and $g_0$ a $1$-form on $S^2$:
\begin{eqnarray}
&&\mu_0=\left.\frac{(\rho u^0)_b}{p_b^+-p_b^-}\right|_{r_b}>0,\label{eq58}\\
&&g_0=g_0(U,U^-, D\psi)\triangleq\left.\frac{m[[\bar{u}']]}{[[p]]}\right|_{S^\psi}-
\mu_0\psi^*(\bar{u}').\label{eq59}
\end{eqnarray}
Here we use the pull back $\psi^*$, so $\psi^* (\bar{u}')=\bar{u}'|_{S^\psi}$ is a 1-form on $S^2$. We note that $g_0$ is a higher-order term (see Definition \ref{def501} below), which depends on $U|_{S^\psi}, U^-|_{S^\psi}$ and $D\psi$.

The R-H conditions \eqref{eq51}-\eqref{eq53} are equivalent to \eqref{eq54}--\eqref{eq55} and \eqref{eq57}, if $[[p]]\ne0$.

\begin{remark}\label{rem51}
By \eqref{eq57}, it is necessary that $\dd\omega=0$, which implies
\begin{eqnarray}\label{eq510}
\dd(\psi^*(\bar{u}'))=-\frac{1}{\mu_0}\dd g_0.
\end{eqnarray}
On the contrary, since the first Betti number of $S^2$ is zero, by Poincar\'{e} lemma, \eqref{eq510} is also sufficient for the existence of a function $\psi^p$ on $S^2$ so that \eqref{eq57} holds, and $\int_{S^2} \psi^p\mathrm{vol}=0$ (here $\mathrm{vol}$ is the  volume 2-form on $S^2$).
Such a function $\psi^p$ is called the {\it profile} of the shock-front $\psi$. We also call the number %\begin{eqnarray}\label{eq511}
$r^p\triangleq\psi-\psi^p$
%\end{eqnarray}
as the {\it position} of the shock-front.
\end{remark}

\subsubsection{Linearization of R-H conditions}\label{sec522}

We recall the following definition from \cite[p.2522]{CY2013}.
\begin{definition}\label{def501}
Set $\hat{U}=U^+-U_b^+$ and $\hat{r}=r^p-r_b$. A {\it higher order term} is an expression containing either

(i) $U^--U_b^-$ and its first or second order derivatives;

\noindent
or

(ii) the products of $\psi^p,\ \hat{r}, \hat{U}$, and their
derivatives $D\hat{U}, D^2\hat{U}, D\psi,D^2\psi$, and $D^3\psi$,
where $D^ku$ are the $k$-{th} order derivatives of $u$.
\end{definition}

As shown in \cite[p.2523]{CY2013}, we can write \eqref{eq54}\eqref{eq55} equivalently as
\begin{eqnarray}
\psi^*(\widehat{u^0})=\mu_1\,(\psi^p+r^p-r_b)+g_1(U,U^-,\psi,D\psi),&&
\label{eq512}\\
\psi^*(\hat{p})=\mu_2\,(\psi^p+r^p-r_b)+g_2(U,U^-,\psi,D\psi),&&
\label{eq513}\\
\psi^*(\hat{\rho})=\mu_3\,(\psi^p+r^p-r_b)+g_3(U,U^-,\psi,D\psi).&&
\label{eq514}
\end{eqnarray}
Using $A(s)=p\rho^{-\gamma}$, it follows that
\begin{eqnarray}
\psi^*(\widehat{A(s)})=\mu_4\,(\psi^p+r^p-r_b)+g_4(U,U^-,\psi,D\psi).\label{eq515}
\end{eqnarray}
Here $\mu_j$ ($j=1, 2, 3, 4$) are  constants %so that (see the line below (2.31) in \cite[p.2523]{CY2013})
%\begin{eqnarray*}
%&&\mu_1>0,\quad
%\mu_2\triangleq\left.-\frac{4\rho_b}{(\gamma+1)r_b}((\gamma-1)u_b^2+c_b^2)\right|_{r_b}<0,\nonumber\\
%&&\mu_3<0,\quad
%\mu_4\triangleq\left.\frac{4(\gamma-1)}{(\gamma+1)r_b}\frac{c_b^2-u_b^2}{\rho_b^{\gamma-1}}\right|_{r_b}>0,
%\end{eqnarray*}
and $g_k$ ($k=1,2,3,4$) are higher order terms.
From \eqref{eq513} and \eqref{eq515} we also have
\begin{eqnarray}\label{eq516add}
\psi^*(\widehat{A(s)})=\frac{\mu_4}{\mu_2}\psi^*(\hat{p})+g_4-\frac{\mu_4}{\mu_2}g_2.
\end{eqnarray}

We then obtain the following lemma.
\begin{lemma}\label{lem51}
Suppose that $\psi$ and $U=U^+$ are $C^1$, and the physical entropy condition holds. Then the R-H conditions \eqref{eq51}--\eqref{eq53} are equivalent to  $[[E]]|_{S_\psi}=0$, \eqref{eq513}, \eqref{eq515},  and  \eqref{eq57}.
\end{lemma}

\subsection{Problem (T1)}

Now we  consider the boundary condition \eqref{eq215} with $M_1$ there replaced by  the shock-front $S^\psi$ and then formulate Problem (T1), which is easily seen to be equivalent to Problem (T).

\subsubsection{Divergence of tangential velocity field on shock-front}
%We now compute a boundary condition equivalent to  \eqref{eq218} while $M_0$ there is replaced by the shock-front $S^\psi$.  The computation is a little bit tricky.
By the process of deriving \eqref{eq48} and \eqref{eq416}, we infer the following identity
\begin{eqnarray}\label{eq532add}
\frac{\rho u^0}{\left(\frac{u^0}{c}\right)^2-1}\left(\frac{D_up}{\gamma p}+\di u+L_3(D_uu+\frac{\grad\,p}{\rho})\right)
=\p_0\hat{p}+\gamma_1\hat{p}-G(U),
\end{eqnarray}
which holds actually at any point in $\overline{\mathcal{M}}$. Now we restrict it on $S^\psi$. So particularly $x^0$ %appeared in the expressions of  \eqref{eq532add}
should be replaced by $\psi$. Using the commutator relation
\begin{eqnarray*}
\p_\alpha(\psi^*f)=(\p_\alpha\psi)\psi^*(\p_0f)+\psi^*(\p_\alpha f), \quad \alpha=1,2,
\end{eqnarray*}
we have
\begin{eqnarray}
\psi^*G_1%&=&\psi^*\left(\frac{-\rho u^0}{\frac{(u^0)^2}{c^2}-1}\left(\frac{1}{\sqrt{g}}\p_\alpha(\sqrt{g}u^\alpha)\right)\right)\nonumber\\
&=&\psi^*\left(\frac{-\rho u^0}{\frac{(u^0)^2}{c^2}-1}\right)
\left(\frac{1}{\sqrt{g}}\p_\alpha(\sqrt{g}\psi^*u^\alpha)-(\p_\alpha\psi) \psi^*(\p_0u^\alpha)\right)\nonumber\\
&=&\psi^*\left(\frac{-\rho u^0}{\frac{(u^0)^2}{c^2}-1}\right)\left(
-\dd^*(\psi^*\bar{u}')-\p_\alpha\psi \psi^*(\p_0u^\alpha)\right).\label{eq534}
\end{eqnarray}
Here $\dd^*$ is the co-differential operator on $S^2$, and the last equality holds because $\dd^*\bar{u}=-\di\, u$ for a vector field $u$.

From \eqref{eq45add33}, we have
\begin{eqnarray*}
&&\psi^*\{G(D_uu+(\grad\, p)/\rho, G^{\alpha\beta}\p_\beta\psi\p_\alpha)\}=\left.\{G_{k\alpha}(u^j\p_ju^k+u^ju^m\Gamma^{k}_{jm}
+\frac{1}{\rho}\p_ip G^{ik})G^{\alpha\beta}\p_\beta\psi\}\right|_{S^\psi}\\
%&=&\left.\{u^0\p_0u^\beta\p_\beta\psi+u^\alpha\p_\alpha u^\beta\p_\beta\psi+u^ju^m\Gamma^\beta_{jm}\p_\beta\psi+\frac{1}{\rho}\sum_{\beta=1}^2 G^{\beta\beta}\p_\beta p\p_\beta\psi\}\right|_{S^\psi}\\
&=&\psi^*\left\{(u^0-u^\alpha\p_\alpha\psi)\p_\beta\psi \p_0u^\beta+u^ju^m\Gamma^\beta_{jm}\p_\beta\psi+\frac{1}{\rho}\sum_{\beta=1}^2 G^{\beta\beta}\p_\beta p\p_\beta\psi\right\}+\psi^*(u^\alpha) \p_\beta\psi\p_\alpha(\psi^*u^\beta).
\end{eqnarray*}
One solves $\psi^*(\p_\beta\psi \p_0u^\beta)$, and  \eqref{eq534} becomes
\begin{eqnarray*}
\psi^*G_1&=&\psi^*\left(\frac{\rho u^0}{\frac{(u^0)^2}{c^2}-1}\right)
\dd^*(\psi^*\bar{u}')+\psi^*\left(\frac{\rho u^0}{\frac{(u^0)^2}{c^2}-1}\times\frac{1}{u^0-u^\alpha\p_\alpha\psi}\right)\nonumber\\
&&\quad\times\psi^*\Big\{
G(D_uu+(\grad\, p)/\rho, G^{\alpha\beta}\p_\beta\psi\p_\alpha)-
u^ju^m\Gamma^{\beta}_{jm}\p_\beta\psi\nonumber\\
&&\qquad\qquad\qquad-\frac{1}{\rho}\sum_{\beta=1}^2G^{\beta\beta}\p_\beta p\p_\beta\psi-u^\alpha\p_\beta\psi\p_\alpha(\psi^*u^\beta)\Big\}.
\end{eqnarray*}

Similarly, one my replace the normal derivatives by tangential derivatives to compute $\psi^*G_2$ and $\psi^*G_3$.
Now set
%\begin{eqnarray}
%\gamma_2&\triangleq&\left.\left(\gamma_1-\frac{2}{r_b}\frac{2+(\gamma-1)t}{1+(\gamma-1)t}\frac{1}{1-t}
%\right)\right|_{x^0=r_b}=\left.\frac{2\gamma t}{r_b}\frac{(\gamma-1)t^2+t+1}{(1-t)^2(1+(\gamma-1)t)}\right|_{x^0=r_b}>0,\\
%\gamma_3&\triangleq&\left.\left(\frac{\rho u^0}{\left(\frac{u^0}{c}\right)^2-1}\right)\right|_{x^0=r_b}<0,
%\end{eqnarray}
%and
\begin{eqnarray*}
&&L_1(w)\triangleq-\frac{1}{(u^0)^2}\frac{u^\delta\p_\delta\psi}{u^0-u^\delta\p_\delta\psi}\psi^*w,\\
&&L_2(w)\triangleq\frac{\rho^{\gamma-1}}{(\gamma-1)(u^0)^2}\frac{u^\delta\p_\delta\psi}{u^0-u^\delta\p_\delta\psi}\psi^*w,\\
&&L'_3(w)\triangleq L_3(w)+\frac{1}{u^0-u^\delta\p_\delta\psi}G(w, G^{\alpha\beta}\p_\beta\psi\p_\alpha)+\frac{\psi^2}{(u^0)^2}\frac{u^\delta\p_\delta\psi}{u^0-u^\delta\p_\delta\psi}
G(w, G^{\delta\alpha}g_{\alpha\beta}u^\beta\p_\delta).
\end{eqnarray*}
Then \eqref{eq532add} becomes
\begin{eqnarray*}
&&\psi^*\left\{\frac{\rho u^0}{\left(\frac{u^0}{c}\right)^2-1}\left(\frac{D_up}{\gamma p}+\di\,u+L^1(D_uE)+L^2(D_uA(s))+L'_3(D_uu+\frac{\grad\, p}{\rho})\right)\right\}\nonumber\\
&=&\psi^*\left\{\p_0(\hat{p})+\gamma_2\hat{p}-\gamma_3\dd^*(\psi^*\bar{u}')\right\}+\tilde{G}.
\end{eqnarray*}
Here $\gamma_2>0$ and $\gamma_3<0$ are constants determined by the background solution, and
\begin{eqnarray}\label{eq543add3}
\tilde{G}&\triangleq&\psi^*\left(-\frac{\rho u^0}{\left(\frac{u^0}{c}\right)^2-1}+\gamma_3\right)\dd^*(\psi^*\bar{u}')
+\left(\frac{2}{\psi}\psi^*\left(\frac{\gamma t^2-t+2}{(1-t)^2}\right)-\frac{2}{r_b}\left.\frac{\gamma t^2-t+2}{(1-t)^2}\right|_{x^0=r_b}\right)\psi^*\hat{p}\nonumber\\
&&+\psi^*\left\{\frac{\rho u^0}{\left(\frac{u^0}{c}\right)^2-1}\frac{1}{u^0-u^\alpha\p_\alpha\psi}
\left(u^ju^m\Gamma^\beta_{jm}\p_\beta\psi+\frac{1}{\rho}\sum_{\beta=1}^2G^{\beta\beta}\p_\beta p\p_\beta\psi+u^\alpha\p_\alpha(\psi^*u^\beta)\p_\beta\psi\right)\right\}\nonumber\\
&&+\psi^*\left\{\frac{1}{\left(\frac{u^0}{c}\right)^2-1}\frac{1}{u^0-u^\alpha\p_\alpha\psi}
\left(\frac{\rho^\gamma}{\gamma-1}u^\alpha\p_\alpha(\psi^*A(s))-\rho u^\alpha\p_\alpha(\psi^*E)\right)\right\}\nonumber%\\
\end{eqnarray}
\begin{eqnarray}
&&+\psi^*\left\{\frac{1}{2u^0}\frac{\rho\psi^2 }{\left(\frac{u^0}{c}\right)^2-1}
\left(u^\delta\p_\delta g_{\alpha\beta}u^\alpha u^\beta+\frac{2u^0g_{\alpha\beta}u^\beta u^\delta\p_\delta(\psi^*u^\alpha)}{u^0-u^\alpha\p_\alpha\psi}\right)\right\}\nonumber\\ &&-\psi^*\left\{\frac{1}{\left(\frac{u^0}{c}\right)^2-1}
\left(\rho u^\alpha u^\beta\Gamma^0_{\alpha\beta}-u^0\left(\frac{1}{c^2}+\frac{1}{(u^0)^2}\right)u^\delta\p_\delta p\right)\right\}\nonumber\\
%\end{eqnarray}
%\begin{eqnarray}
&&+\psi^*\left\{\frac{1}{u^0}\frac{\rho \psi^2}{\left(\frac{u^0}{c}\right)^2-1}\frac{u^\delta\p_\delta\psi}{u^0-u^\alpha\p_\alpha\psi}
\left(g_{\alpha\beta}u^\beta u^j u^m\Gamma^{\alpha}_{jm}+\frac{1}{\rho\psi^2}u^\alpha\p_\alpha p\right)\right\}\nonumber\\
&&+O(1)\psi^*(|\hat{U}|^2+|\hat{\psi}|^2+\hat{E})+\frac{2}{\psi}
\psi^*\left\{\frac{t+\frac{2}{\gamma-1}}{(1-t)^2}\rho_b^\gamma \left(g_4-\frac{\mu_4}{\mu_2}g_2\right)\right\}.
\end{eqnarray}
We see \eqref{eq215} is equivalent to
\begin{eqnarray}\label{eq516add2}
\dd^*(\psi^*\bar{u}')=\mu_5\,
\psi^*(\p_0\hat{p})+\mu_6\,
\psi^p+\mu_6\,(r^p-r_b)%\nonumber\\&&
+g_5(U,U^-,\psi,DU,D\psi),
\end{eqnarray}
if we replace $\psi^*(\hat{p})$ by ${\psi}$ via \eqref{eq513}. Here
%\begin{eqnarray}
$\mu_5<0,  \mu_6>0$
%\end{eqnarray}
are constants determined by the background solution, and
\begin{eqnarray}\label{eqbarg5}
g_5\triangleq\frac{1}{\gamma_3}\tilde{G}+\frac{\gamma_2}{\gamma_3}g_2.
\end{eqnarray}
\begin{remark}\label{rem52}
In the expression of $g_5$, there appear first order derivatives of $\hat{p}$, and  only first order tangential derivatives of $A(s), u', E, \psi$ on $S^\psi$.
Note that \eqref{eq516add2} is a first-order boundary condition on the shock-front. Together with \eqref{eq510}, we have a div-curl system of  the tangential velocity $\psi^*\bar{u}'$ on $S^2$.
\end{remark}

\subsubsection{Problem (T1)}
For functions $U=(E, A(s), p, u'=u^1\p_1+u^2\p_2)$ and $\psi$ (note that $\hat{U}=U-U_b^+$ and $\psi=\psi^p+r^p$), we formulate the following problems:
\begin{eqnarray}
&&\begin{cases}\label{eq520}
D_uE=0&\text{in}\quad \mathcal{M}^+_\psi,\\
E=E^-&\text{on}\quad S^\psi;
\end{cases}\\
&&\begin{cases}\label{eq521}
L(\hat{p})=F(U, DU, D^2p)& \text{in}\quad \mathcal{M}^+_\psi,\\
\hat{p}=p_1-p_b^+ &\text{on}\quad M_1,\\
\psi^*(\hat{p})=\mu_2\,(\psi^p+r^p-r_b)+g_2(U,U^-,\psi,D\psi);
\end{cases}\\
&&\begin{cases}\label{eq522}
D_uA(s)=0&\text{in}\quad \mathcal{M}^+_\psi,\\
\psi^*(\widehat{A(S)})=\mu_4\,(\psi^p+r^p-r_b)+g_4(U,U^-,\psi,D\psi);\label{eq515add1}
\end{cases}\\
&&\begin{cases}\label{eq523}
\dd\psi=\mu_0\psi^*(\bar{u}')+g_0(U,U^-,D\psi),\\
\dd^*(\psi^*\bar{u}')=\mu_5\,
\psi^*(\p_0\hat{p})+\mu_6\,
\psi^p+\mu_6\,(r^p-r_b)+g_5(U,U^-,\psi,DU,D\psi)
\end{cases}\text{on}\ \ S^2;\\
&&\begin{cases}\label{eq524}
G(D_uu+\frac{\grad p}{\rho}, \p_\alpha)=0&\text{in}\quad \mathcal{M}^+_\psi,\quad \alpha=1,2,\\
u'|_{S^\psi}=u'_0.
\end{cases}
\end{eqnarray}
The initial data $u'_0$ in \eqref{eq524} is  the vector field on $S^2$ associated with the 1-form $\psi^*\bar{u}'$ in \eqref{eq523}.

By Theorem \ref{thm21}, Lemma \ref{lem51} and the analysis above,
it is obvious that for given supersonic flow $U^-$, the solution $(U, \psi)$ to these problems also solves the Euler system and the R-H conditions hold across $S^\psi.$ Hence  we could rewrite Problem (T) equivalently as the following Problem (T1).

\medskip
\fbox{
\parbox{0.90\textwidth}{
Problem (T1): Find $\psi$ and $U=U^+$ in $\mathcal{M}_\psi^+$ satisfying \eqref{th102}, \eqref{th104} and solving the problems \eqref{eq520}--\eqref{eq524}.}}
\medskip

\subsection{Problem (T2)}

Acting $\dd^*$ to the first equation in \eqref{eq523} and using the second equation, we derive that
\begin{eqnarray}\label{eq525}
-\Delta'\psi^p+\mu_7\psi^p=\mu_0\mu_6(r^p-r_b)+\mu_0\mu_5\,\psi^*\p_0\hat{p}
+g_6(U, U^-, \psi, DU^-, DU, D\psi, D^2\psi),
\end{eqnarray}
with $g_6=\mu_0g_5+\dd^*g_0$ and $\mu_7=-\mu_0\mu_6<0$.
Here $\Delta'$ is the standard Laplace operator on $S^2$.
Then using the third equation in \eqref{eq521}, we get
\begin{multline}\label{eq526}
-\Delta'(\psi^*\hat{p})+\mu_7(\psi^*{\hat{p}})+\mu_9(\psi^*\p_0\hat{p})
=g_8(U, U^-, \psi, DU, DU^-, D\psi, D^2U, D^2U^-, D^2\psi, D^3\psi),
\end{multline}
where $\mu_9=-\mu_0\mu_2\mu_5<0$ and
$g_8=-\Delta'g_2+\mu_7g_2+\mu_0\mu_2g_5+\mu_2\dd^*g_0.$
We remind here again that $DU$ and $D^2U$ represent first order and second order  derivatives of $\psi^*u, \psi^*p, \psi^*E, \psi^*A(s)$ on $S^2$.

By the second equation in \eqref{eq523}, using the divergence theorem, and recall that $\int_{S^2}\psi^p\vol=0,$ we have
\begin{eqnarray}\label{eq527}
r^p-r_b=-\frac{1}{4\pi\mu_6}\int_{{S}^2}\Big(\mu_5\,
\psi^*(\p_0\hat{p})+g_5(U,U^-,\psi,DU,D\psi)\Big)\, \vol.
\end{eqnarray}
Substituting this into the third equation in \eqref{eq521}, we then obtain
\begin{eqnarray}\label{eq528}
\psi^p=\frac{1}{\mu_2}\left(\psi^*\hat{p}-{\mu_8}
\int_{{S}^2}\psi^*(\p_0\hat{p})\,\vol+{g_7(U,U^-,\psi,D\psi)}\right),
\label{177}
\end{eqnarray}
with $\mu_8=-\frac{\mu_2\mu_5}{4\pi\mu_6}<0$ and
$g_7=\frac{\mu_2}{4\pi\mu_6}\int_{{S}^2}g_5\,\vol-g_2.$

We now formulate the following Problem (T2).

\medskip
\fbox{
\parbox{0.90\textwidth}{
Problem (T2): Find $\psi$ and $U=\hat{U}+U_b^+$ that solve  \eqref{eq520}, \eqref{eq530}, \eqref{eq531}, \eqref{eq515add1}, \eqref{eq533} and \eqref{eq524}.}}

\begin{eqnarray}
&&\begin{cases}\label{eq530}
L(\hat{p})=F(U, DU, D^2p)& \text{in}\quad \mathcal{M}_\psi,\\
\hat{p}=p_1-p_b^+ &\text{on}\quad M_1,\\
-\Delta'(\psi^*\hat{p})+\mu_7(\psi^*{\hat{p}})+\mu_9(\psi^*\p_0\hat{p})\\
\qquad=g_8(U, U^-, \psi, DU, DU^-, D\psi, D^2U, D^2U^-, D^2\psi, D^3\psi);
\end{cases}\\
&&\begin{cases}\label{eq531}
r^p-r_b=-\frac{1}{4\pi\mu_6}\int_{{S}^2}\Big(\mu_5\,
\psi^*(\p_0\hat{p})+g_5(U,U^-,\psi,DU,D\psi)\Big)\, \vol,\\
\psi^p=\frac{1}{\mu_2}\left(\psi^*\hat{p}-{\mu_8}
\int_{{S}^2}\psi^*(\p_0\hat{p})\,\vol+{g_7(U,U^-,\psi,D\psi)}\right),\\
\psi=\psi^p+r^p;
\end{cases}\\
&&\begin{cases}\label{eq533}
\dd(\psi^*\bar{u}')=-\frac{1}{\mu_0}\dd g_0(U, U^-, D\psi),\\
\dd^*(\psi^*\bar{u}')=\mu_5\,
\psi^*(\p_0\hat{p})+\mu_6\,
\psi^p+\mu_6\,(r^p-r_b)+g_5(U,U^-,\psi,DU,D\psi)
\end{cases} \text{on}\ \ S^2.
\end{eqnarray}

\begin{lemma}\label{lem52}
Problem (T2) is equivalent to Problem (T1).
\end{lemma}

\begin{proof}
We only need to prove $\int_{S^2}\psi^p\vol=0$, the third equation in \eqref{eq521}, and the first equation in \eqref{eq523}. Obviously the second equation in \eqref{eq533} and the definition of $r^p$ in \eqref{eq531} imply that $\int_{S^2}\psi^p\vol=0$. And the third equation in \eqref{eq521}, namely, \eqref{eq513}, follows directly from \eqref{eq531}.

Now from the third equation in \eqref{eq530}, namely \eqref{eq526}, and  \eqref{eq513}, we get \eqref{eq525}, which means that, with the help of the second equation in \eqref{eq533},
\begin{eqnarray}\label{eq535}
\dd^*\Big(\dd\psi^p-\mu_0\psi^*(\bar{u}')-g_0\Big)=0.
\end{eqnarray}
The first equation in \eqref{eq533} says that the 1-form $\mu_0\psi^*(\bar{u}')+g_0$ is closed. Since the first Betti number of $S^2$ vanishes, there is a function $\phi$ with $\int_{S^2}\phi\vol=0$ so that $\dd\phi=\mu_0\psi^*(\bar{u}')+g_0$. Furthermore, if $\phi'$ also satisfies $\dd\phi'=\mu_0\psi^*(\bar{u}')+g_0$, then $\dd(\phi-\phi')=0$, and  we infer that $\phi'=\phi+k$, with $k\in\mathbb{R}$. So from \eqref{eq535} we have
$-\Delta'(\psi^p-\phi)=\dd^*\dd (\psi^p-\phi)=0.$
By maximum principle, $\psi^p-\phi$ is a constant. Since both of them have mean zero, we conclude that $\psi^p=\phi$. This shows that $\psi^p$, hence $\psi$, also solves the first equation in \eqref{eq523}. %The proof is completed.
\end{proof}

\subsection{Problem (T3)}

The above equations and boundary conditions are formulated in
$\mathcal{M}_\psi^+.$
Supposing $\psi\in C^{4,\alpha}(S^2)$, we
introduce a $C^{4,\alpha}$--homeomorphism $ \Psi:
(x^0,x')\in\mathcal{M}_\psi^+\mapsto ({y}^0,y')\in
\Omega\triangleq(r_b,r_1)\times {S}^2$ by
\begin{eqnarray}\label{259}
{y}^0=\frac{x^0-\psi(x')}{r_1-\psi(x')}(r_1-r_b)+r_b,\qquad y'=(y^1, y^2)=x'=(x^1, x^2)
\end{eqnarray}
to normalize $\mathcal{M}_\psi^+$ to $\Omega$. %Then
%\begin{eqnarray}
%x^0=\frac{y^0-r_b}{r_1-r_b}(r_1-\psi(y'))+\psi(y'), \quad x'=y',
%\end{eqnarray}
%and
%$$\frac{\p y^0}{\p x^0}=\frac{r_1-r_b}{r_1-\psi},\quad \frac{\p y^0}{\p x^\alpha}=\frac{y^0-r_1}{r_1-\psi}\p_\alpha\psi.$$
We set
$\Omega_{0}=\{r_b\}\times{S}^2$ and $\Omega_{1}=\{r_1\}\times{S}^2$. They are respectively the image of $S^\psi$ and $M_1$. Then
$\p\Omega=\Omega_0\cup\Omega_1$. We use $i$ to denote the
embedding of $\Omega_0$ in $\Omega$.

To avoid complication of notations, in the following we still write the unknowns in $y$-coordinates as $(u, p, \rho, s, E)$ etc.; namely, we write  $(\Psi^{-1})^*p$ as $p$, $(\Psi^{-1})^*E$ as $E$, etc. and write
\begin{eqnarray}\label{eq562add}
\Psi_*u=\frac{r_1-r_b}{r_1-\psi(y')}\left(u^0(\Psi^{-1}(y))+\frac{y^0-r_1}{r_1-r_b}u^\alpha\p_\alpha \psi(y')\right)\p_{y^0}+u^\alpha(\Psi^{-1}(y))\p_{y^\alpha}
\end{eqnarray}
still as $u$. (Note that $(\Psi_*u)^\alpha=u^\alpha$ for $\alpha=1,2$. So there is no confusion.) By this convention $\psi^*E$ becomes $i^*E$ in $y$-coordinates.
We have
%\begin{eqnarray*}
%&&\frac{\p \hat{p}}{\p{x^0}}=\frac{\p y^0}{\p x^0}\frac{\p \hat{p}}{\p{y^0}}=\frac{\p \hat{p}}{\p{y^0}}+O(\psi-r_b)\frac{\p \hat{p}}{\p y^0},\\
%&&\frac{\p^2 \hat{p}}{(\p{x^0})^2}=\left(\frac{r_1-r_b}{r_1-\psi}\right)^2
%\frac{\p^2 \hat{p}}{(\p{y^0})^2}=\frac{\p^2 \hat{p}}{(\p{y^0})^2}+O(\psi-r_b)D^2\hat{p},\\
%&&\frac{\p \hat{p}}{\p{x^\alpha}}=\frac{\p \hat{p}}{\p y^\alpha}+\frac{y^0-r_1}{r_1-\psi(y')}\p_\alpha\psi\frac{\p \hat{p}}{\p y^0}=\frac{\p \hat{p}}{\p{y^\alpha}}+O(D\psi)D\hat{p},
%\end{eqnarray*}
%and
\begin{eqnarray*}
\Delta'_x\hat{p}%&=&\frac{1}{\sqrt{g}}\p_{x^\alpha}(\sqrt{g}g^{\alpha\beta}
%\p_{x^\beta} \hat{p})\\
%&=&\frac{1}{\sqrt{g}}\p_{y^\alpha}(\sqrt{g}g^{\alpha\beta}
%\p_{y^\beta} \hat{p})+\frac{1}{\sqrt{g}}\p_{y^\alpha}\left(\sqrt{g}g^{\alpha\beta}
%\p_{y^\beta}\psi\p_{y^0}\hat{p}\frac{y^0-r_1}{r_1-\psi(y')}\right)\\
%&&+\frac{y^0-r_1}{r_1-\psi(y')}g^{\alpha\beta}\p_{y^\alpha}\psi\frac{\p^2\hat{p}}{\p y^0\p y^\beta}+\frac{y^0-r_1}{r_1-\psi(y')}g^{\alpha\beta}\p_{y^\alpha}\psi\p_{y^\beta}\psi
%\frac{\p}{\p y^0}\left(\frac{y^0-r_1}{r_1-\psi(y')}\frac{\p \hat{p}}{\p y^0}\right)\\
&=&\Delta'_y\hat{p}+O(1)(D^2\hat{p}D\psi)+O(1)(D\hat{p}D^2\psi)+O(1)(D\hat{p}D\psi),\\
%\end{eqnarray*}
%and
\psi^*\left({\frac{\p p}{\p{x^0}}}\right)&=&i^*\left(\frac{\p p}{\p{y^0}}\right)+O(1)(\psi-r_b)D\hat{p}.
\end{eqnarray*}
Hence Problem (T2) could be rewritten as the following Problem (T3), where we use $\bar{F}$ or $\bar{g}$ to denote the higher order terms in the $y$-coordinates, and $\p_j$ means $\frac{\p}{\p y_j}$.

\fbox{
\parbox{0.90\textwidth}{
Problem (T3): Find $\psi\in C^{4,\alpha}(S^2)$ and $U=U_b^++\hat{U}$ that solve the following problems \eqref{eq537}--\eqref{eq542}. The initial data $u_0'$ in \eqref{eq542}  is the vector field corresponding to the 1-form  $\bar{u}_0'$ on $S^2$ obtained from \eqref{eq541}.}}

\begin{eqnarray}
&&\begin{cases}\label{eq537}
D_u\hat{E}=0&\text{in}\quad \Omega,\\
\hat{E}=E^--E_b^-&\text{on}\quad \Omega_0;
\end{cases}%\\
\end{eqnarray}
\begin{eqnarray}
&&\begin{cases}\label{eq538}
L(\hat{p})=\bar{F}(U, DU, D^2p)& \text{in}\quad \Omega,\\
\hat{p}=p_1-p_b^+ &\text{on}\quad \Omega_1,\\
-\Delta'(i^*\hat{p})+\mu_7(i^*{\hat{p}})+\mu_9(i^*\p_0\hat{p})\\
\qquad=\bar{g}_8(U,U^-,\psi,DU,D\psi,D^2U,D^2\psi,D^3\psi);
\end{cases}\\
%\end{eqnarray}
%\begin{eqnarray}
&&\begin{cases}\label{eq539}
r^p-r_b=-\frac{1}{4\pi\mu_6}\int_{{S}^2}\Big(\mu_5\,
i^*(\p_0\hat{p})+\bar{g}_5(U,U^-,\psi,DU,D\psi)\Big)\, \vol,\\
\psi^p=\frac{1}{\mu_2}\left(i^*\hat{p}-{\mu_8}
\int_{{S}^2}i^*(\p_0\hat{p})\,\vol+\bar{g}_7(U,U^-,\psi,D\psi)\right),\\
\psi=\psi^p+r^p;
\end{cases}\\
%\end{eqnarray}
%\begin{eqnarray}
&&\begin{cases}\label{eq540}
D_u\widehat{A(s)}=0&\text{in}\quad \Omega,\\
i^*(\widehat{A(s)})=\mu_4\,(\psi^p+r^p-r_b)+\bar{g}_4(U,U^-,\psi,D\psi);
\end{cases}\\
&&\begin{cases}\label{eq541}
\dd(\bar{u}_0')=-\frac{1}{\mu_0}\dd \bar{g}_0(U,U^-,D\psi),\\
\dd^*(\bar{u}_0')=\mu_5\,
i^*(\p_0\hat{p})+\mu_6\,
\psi^p+\mu_6\,(r^p-r_b)+\bar{g}_5(U,U^-,\psi,DU,D\psi)
\end{cases} \text{on}\  S^2;\\
&&\begin{cases}\label{eq542}
G(D_uu+\frac{\grad p}{\rho}, \p_\alpha)=0&\text{in}\quad \Omega,\quad \alpha=1,2,\\
u'|_{\Omega_0}=u'_0.
\end{cases}
\end{eqnarray}
%The initial data $u_0'$ in \eqref{eq542}  is the vector field corresponding to the 1-form  $\bar{u}_0'$ on $S^2$ obtained from \eqref{eq541}.

\begin{remark}
Recalling \eqref{eq441}, in \eqref{eq538} we should have
\begin{eqnarray}\label{eq543}
L(\hat{p})&\triangleq&-\frac{1}{(y^0)^2}\Delta'\hat{p}+(t(y^0)-1)\p_0^2\hat{p}+\frac{4}{y^0}
b(t(y^0))\p_0\hat{p}+\frac{1}{(y^0)^2}e(t(y^0))\hat{p}\nonumber\\
&&\qquad+\frac{\rho_b(y^0)}{ (y^0)^2}d_1(t(y^0))\hat{E}
+\frac{\rho_b(y^0)^\gamma }{(y^0)^2}d_2(t(y^0))\widehat{A(s)}
=\bar{F}\triangleq F+\tilde{F}_1,
\end{eqnarray}
where the coefficients are known functions of $y^0$, and $F=F(U,DU,D^2p, D\psi, D^2\psi)$ is the higher order term appeared below \eqref{eq441}, which is now writing in the $y$-coordinates; and
\begin{eqnarray}\label{eq554add}
\tilde{F}_1=%&(t(y^0)-t(x^0))\p_0^2\hat{p}+\left(\frac{4}{y^0}
%b(t(y^0))-\frac{4}{x^0}
%b(t(x^0))\right)\p_0\hat{p}+\left(\frac{1}{(y^0)^2}-\frac{1}{(x^0)^2}\right)e(t(y^0))\hat{p}\nonumber\\
%&&+
%\left(\frac{\rho_b(y^0)}{ (y^0)^2}d_1(t(y^0))-\frac{\rho_b(x^0)}{ (x^0)^2}d_1(t(x^0))\right)\hat{E}
%+\left(\frac{\rho_b(y^0)^\gamma }{(y^0)^2}d_2(t(y^0))-\frac{\rho_b(x^0)^\gamma }{(x^0)^2}d_2(t(x^0))\right)\widehat{A(s)},
%\nonumber\\&&+
O(1)({\psi-r_b})(D^2\hat{p}+D\hat{p}+\hat{U})+O(1)(D^2\hat{p}D\psi)
+O(1)(D\hat{p}D^2\psi)+O(1)(D\hat{p}D\psi).
\end{eqnarray}
%with $x^0=\frac{r_1-\psi}{r_1-r_b}(y^0-r_b)+\psi$. Note that
%$$y^0-x^0=\frac{r_1-y^0}{r_1-r_b}(r_b-\psi)$$
%is a first-order small term.  So by analyticity of polynomials and the background solution $U_b(\cdot)$, $\tilde{F}_1$ is actually a higher-order term.
We also note that
\begin{eqnarray}
&&\bar{g}_8=g_8+O(\psi-r_b)D\hat{p},\quad \bar{g}_7=g_7+O(\psi-r_b)D\hat{p},\label{eqbarg1}\\
&&\bar{g}_5=g_5+O(\psi-r_b)D\hat{p}, \quad \bar{g}_4=g_4, \quad \bar{g}_0=g_0, \quad \bar{g}_6=\mu_0\bar{g}_5+\dd^*\bar{g}_0.\label{eqbarg2}
\end{eqnarray}
\end{remark}
%------------------------------------2014-6-18
\begin{remark}
Using \eqref{eq448},  we could  obtain, for $\alpha=1$,  the specific expression of \eqref{eq542} in a local coordinates in $\Omega$:%the $y$-coordinate:
\begin{eqnarray}
\begin{cases}\label{eq543add}
u^j\p_ju^1+\frac{2u_b(y^0)}{y^0}u^1=-\frac{1}{\rho (y^0)^2}\p_1 p +\bar{f}_1(U, \psi, D\psi, Dp) &\text{in}\quad \Omega,\\
u^1=u^1_0&\text{on}\quad \Omega_0.
\end{cases}
\end{eqnarray}
Here $u^j\p_j$ is $\Psi_*u$ given by \eqref{eq562add},   $u'_0=u_0^1\p_1+u_0^2\p_2$, and
\begin{eqnarray}\label{eq572add2}
\bar{f}_1(U, \psi, D\psi, Dp)&=&\left(\frac{2u_b(y^0)}{y^0}-\frac{2u_b(x^0)}{x^0}\right)u^1
+\frac{\p_1\hat{p}}{\rho}\left(\frac{1}{(y^0)^2}-\frac{1}{(x^0)^2}\right)\nonumber\\
&&\qquad\quad-\frac{1}{\rho (x^0)^2}\left(\frac{y^0-r_1}{r_1-\psi(y')}\p_1\psi\right)\p_0\hat{p}+h_1(U).
\end{eqnarray}
By changing the local coordinates $(y^1, y^2)$, we know that  $u^2$ solves a problem similar to \eqref{eq543add}.
\end{remark}

\subsection{Problem (T4)}
Since the elliptic problem \eqref{eq538} is coupled with the hyperbolic problem \eqref{eq540}, we need to further reformulate Problem (T3) equivalently as the following Problem (T4).

We firstly note that by Lemma \ref{lem52}, Problem (T3) is equivalent to Problem (T1), provided that
\begin{eqnarray}\label{eq556add}
\psi\in C^{4,\alpha}(S^2) \quad\text{and}\quad  |\psi-r_b|<\min\left\{\frac{r_b-r_0}{4}, \frac{r_1-r_b}{4}, h^\sharp\right\}.
\end{eqnarray}
So we may replace the boundary condition in \eqref{eq540} by \eqref{eq516add}, reads now
\begin{eqnarray}\label{eq545}
i^*(\widehat{A(s)})=\frac{\mu_4}{\mu_2}i^*(\hat{p})+\bar{g}_4-\frac{\mu_4}{\mu_2}\bar{g}_2,
\end{eqnarray}
where $\bar{g}_2=g_2, \bar{g}_4=g_4$, and the resultant problem is still equivalent to Problem (T1).

%-------------------------------------------------------------------20140529
We now consider the Cauchy problem \eqref{eq540}:
\begin{eqnarray}
\begin{cases}\label{eq546}
D_uA(s)=0&\text{in}\quad \Omega,\\
i^*(\widehat{A(s)})=\frac{\mu_4}{\mu_2}i^*(\hat{p})+\bar{g}_4-\frac{\mu_4}{\mu_2}\bar{g}_2.
\end{cases}
\end{eqnarray}
For the vector field $u$ defined in $\Omega$, as in the proof of Theorem \ref{thm61}, we consider the non-autonomous vector field $\frac{u'}{u^0}(y^0, y')$ defined for $y'=(y^1, y^2)\in S^2$ and $y^0\in[r_b, r_1]$. For $\bar{y}\in S^2$, we write the integral curve passing $(r_b, \bar{y})$
as $y'=\varphi(y^0, \bar{y})$, which is a $C^{k,\alpha}$ function in $\Omega$ if $u\in C^{k,\alpha}(\bar{\Omega})$ and $u^0>\delta$ for a positive  constant $\delta$, and $k\in\mathbb{N}$. For fixed $y^0$, the map $\varphi_{y^0}: S^2\to S^2, \ \bar{y}\mapsto y'=\varphi(y^0, \bar{y})$  is a $C^{k,\alpha}$ homeomorphism. Note that $\varphi_{r_b}$ is the identity map on $S^2$.

\begin{lemma}\label{lem53}
Suppose that $u=u^0\p_0+u'\in C^{0,1}$ and $u^0>\delta$. There is a positive constant $C=C(\delta,r_1-r_b)$ so that for any $y'\in S^2$ and $y_0\in[r_b, r_1]$, it holds
\begin{eqnarray}\label{eq576}
\left|(\varphi_{y^0})^{-1}y'-y'\right|\le C\norm{u'}_{C^0(\bar{\Omega})}.
\end{eqnarray}
\end{lemma}

\begin{proof}
Let $\bar{y}=(\varphi_{y^0})^{-1}y'$. Then
%\begin{eqnarray*}
$\left|\varphi_{y^0}(\bar{y})-\bar{y}\right|\le\int_{r_b}^{y^0}\left|\frac{u'}{u^0}(s, \varphi(s,\bar{y}))\right|\,\dd s\le C\norm{u'}_{C^0(\bar{\Omega})}$
%\end{eqnarray*}
as desired.
\end{proof}

We write the unique solution to the linear transport equation \eqref{eq546} as follows:
\begin{eqnarray}\label{eq548}
A(s)(y)=A(s)(y^0, y')&=&A(s)(y^0,  \varphi_{y^0}(\bar{y}))=A(s)(r_b, \varphi_{r_b}(\bar{y}))=A(s)(r_b, \bar{y})\nonumber\\
&=&(i^*A(s))(\bar{y})=(i^*A(s))((\varphi_{y^0})^{-1}(y')).
\end{eqnarray}
Hence, recall that the entropy is a constant behind the shock-front for the background solution, we have
\begin{eqnarray}\label{eq549}
\widehat{A(s)}(y)%&=&A(s)(y^0,y')-A(s_b^+)(y^0,y')=(i^*A(s))( (\varphi_{y^0})^{-1}(y'))-A(s_b^+)\nonumber\\&=&i^*(A(s))( y')-A(s_b^+)+(i^*A(s))((\varphi_{y^0})^{-1}(y'))-(i^*A(s))(y')\nonumber\\
&=&i^*(\widehat{A(s)})(y')+\Big(A(s)(r_b, (\varphi_{y^0})^{-1}(y'))-A(s)(r_b, y')\Big)\nonumber\\
&=&\frac{\mu_4}{\mu_2}i^*(\hat{p})+\bar{g}_4-\frac{\mu_4}{\mu_2}\bar{g}_2+
\Big(A(s)(r_b, (\varphi_{y^0})^{-1}(y'))-A(s)(r_b, y')\Big).
\end{eqnarray}
Set
\begin{eqnarray}\label{eq550}
\tilde{F}_2=-\left\{\bar{g}_4-\frac{\mu_4}{\mu_2}\bar{g}_2+\Big(A(s)(r_b, (\varphi_{y^0})^{-1}(y'))-A(s)(r_b, y')\Big)\right\},
\end{eqnarray}
which is a higher order term (note that $\p_\alpha A(s)$ itself is small, and $\varphi_{y^0}$ is close to the identity map since $u'$ is nearly zero, so $|(\varphi_{y^0})^{-1}(y')-y'|$ is small by \eqref{eq576}). Then
we could write the elliptic equation  \eqref{eq543} as
\begin{eqnarray}\label{eq551}
&&-\frac{1}{(y^0)^2}\Delta'\hat{p}+(t(y^0)-1)\p_0^2\hat{p}+\frac{4}{y^0}
b(t(y^0))\p_0\hat{p}+\frac{1}{(y^0)^2}e(t(y^0))\hat{p}+\frac{\mu_4}{\mu_2}\frac{\rho_b(y^0)^\gamma}{(y^0)^2}d_2(t(y^0))i^*(\hat{p})\nonumber\\
&=&-\frac{\rho_b(y^0)}{ (y^0)^2}d_1(t(y^0))\hat{E}+F+\tilde{F}_1+\tilde{F}_2.
\end{eqnarray}
Now set
\begin{eqnarray}\label{eq552}
e_1(y^0)& = &(y^0)^2\Big(t(y^0)-1\Big)<0,\quad
e_2(y^0)={4y^0}b(t(y^0)),\\
e_3(y^0)& = &e(t(y^0)),\quad
e_4(y^0)=\frac{\mu_4}{\mu_2}(\rho_b(y^0))^\gamma d_2(t(y^0)),\label{eq554}\\
e_5(y^0)& = &-\rho_b(y^0)d_1(t(y^0)),\label{eq555}\\
f& = &f(U,\psi, DU,D^2p, D\psi, D^2\psi)\triangleq (y^0)^2\Big(F+\tilde{F}_1+\tilde{F}_2\Big).\label{eq556}
\end{eqnarray}
Then equation \eqref{eq551} simply reads
\begin{eqnarray}\label{eq557}
\mathfrak{L}(\hat{p})\triangleq-\Delta'\hat{p}+e_1(y^0)\p_0^2\hat{p}+e_2(y^0)\p_0\hat{p}+e_3(y^0)\hat{p}
+e_4(y^0)i^*(\hat{p}) =e_5(y^0)\hat{E}+f.
\end{eqnarray}
Note that there is a nonlocal term $e_4(y^0)i^*(\hat{p})$. So problem \eqref{eq538} can be reformulated as follows:
\begin{eqnarray}
\begin{cases}\label{eq559}
\mathfrak{L}(\hat{p})=e_5(y^0)\hat{E}+f(U,\psi, DU,D^2p, D\psi, D^2\psi)& \text{in}\quad \Omega,\\
\hat{p}=p_1-p_b^+ &\text{on}\quad \Omega_1,\\
-\Delta'(i^*\hat{p})+\mu_7(i^*{\hat{p}})+\mu_9(i^*\p_0\hat{p})\\
\qquad=\bar{g}_8(U,U^-,\psi,DU,D\psi,D^2U,D^2\psi,D^3\psi).
\end{cases}
\end{eqnarray}

We then state Problem (T4), which is equivalent to Problem (T3), as can be seen from the above derivations.

\smallskip
\fbox{
\parbox{0.90\textwidth}{
Problem (T4): Find $\psi\in C^{4,\alpha}(S^2)$ and $U=U_b^++\hat{U}$ defined in $\Omega$ that solve problems \eqref{eq537}, \eqref{eq559} and \eqref{eq539}--\eqref{eq542}.}}

\subsection{A linear second order nonlocal elliptic equation with Venttsel boundary condition}

In this subsection we  study  the linear nonlocal elliptic equation \eqref{eq557} subjected to a Venttsel boundary condition on $\Omega_0$:
\begin{eqnarray}\label{eq564}
\begin{cases}
\mathfrak{L}(\hat{p})=f& \text{in}\quad \Omega,\\
\hat{p}=h_1(y') &\text{on}\quad \Omega_1,\\
-\Delta'(i^*\hat{p})+\mu_7(i^*{\hat{p}})+\mu_9(i^*\p_0\hat{p})=h_0(y') &\text{on}\quad \Omega_0.
\end{cases}
\end{eqnarray}
Here $f\in C^{k-2,\alpha}(\bar{\Omega}), h_1\in C^{k,\alpha}(\Omega_1)$ and $h_0\in C^{k-2,\alpha}(\Omega_0)$ are given nonhomogeneous terms, and $k=2,3,\cdots$. This problem is formulated according to problem \eqref{eq559}.

We remark that Venttsel condition of second order elliptic equations was proposed by A. D. Venttsel in 1959 in the  studies of  probability theory \cite{Venttsel}. It is quite interesting to see that it appears naturally in the studies of transonic shocks. Later  Y.  Luo and N. S. Trudinger established the classical solvability for the linear and quasilinear Venttsel problems in a series of papers (see \cite{LN} for the linear case). However, our problem has two different characters comparing to these classical results. The first is that the elliptic operator $\mathfrak{L}$ in the domain $\Omega$ contains a nonclassical nonlocal term; the second is that the coefficients of the zero-th order terms may change signs. So we could not use directly the classical maximum principles or energy estimates, and we need the S-Condition to avoid some possible spectrums.

\subsubsection{Uniqueness of solutions in Sobolev space $H^2(\Omega)\cap H^2(\Omega_0)$}
We firstly study under what conditions a strong solution $\hat{p}$ in Sobolev space $H^2(\Omega)$ with $i^*\hat{p}\in H^2(\Omega_0)$ to problem \eqref{eq564} is unique. To this end, we consider the homogeneous problem
\begin{eqnarray}\label{eq565}
\begin{cases}
\mathfrak{L}(\hat{p})=0& \text{in}\quad \Omega,\\
\hat{p}=0 &\text{on}\quad \Omega_1,\\
-\Delta'(i^*\hat{p})+\mu_7(i^*{\hat{p}})+\mu_9(i^*\p_0\hat{p})=0 &\text{on}\quad \Omega_0.
\end{cases}
\end{eqnarray}

\begin{remark}[Regularity]\label{rm55}
By Sobolev Embedding Theorem (Theorem 7.26 in \cite[p.171]{GT}), $H^2(\Omega)$ is embedded into $C^{0,\frac12}(\bar{\Omega})$, so the solution is bounded; $H^2(\Omega_0)$ is embedded into $C^{0,1}(\Omega_0)$, so $i^*\hat{p}\in C^{0,1}(\bar{\Omega})$. By Trace Theorem, $i^*(\p_0\hat{p})\in H^{\frac12}(\Omega_0)$.  Now considering $e_3(y^0)\hat{p}+e_4(y^0)i^*(\hat{p})$  as  a nonhomogeneous term,   the Schauder estimate (Corollary 6.7 in \cite[p.100]{GT}) and a standard existence theorem of Dirichlet problems (like Theorem 6.13 in \cite{GT}) imply that for any $\alpha\in (0,\frac12)$, $\hat{p}$ belongs to  $C^{2,\alpha}(\bar{\Omega}\setminus\Omega_0)$. Then by Theorem 6.17 in \cite{GT}, we infer that $\hat{p}\in C^\infty(\bar{\Omega}\setminus\Omega_0)\cap C^{0,\frac12}(\bar{\Omega})$.

We write the boundary equation in \eqref{eq565} as
\begin{eqnarray}\label{589}
\Delta'(i^*\hat{p})=\mu_7i^*(\hat{p})+\mu_9i^*(\p_0\hat{p}).
\end{eqnarray}
Note that the right-hand side belongs to $H^{\frac12}(\Omega_0)$. Applying the $H^s$ regularity theory ({\it cf.} Theorem 11.1 in \cite[Chapter 5, p.442]{T1}) and Sobolev Embedding Theorem, we have $i^*\hat{p}\in H^{\frac52}(\Omega_0)\subset\subset C^{1,\frac12}(\Omega_0)$. Thus by intermediate Schauder estimate ({\it cf.} Notes of Chapter 6 in \cite[p.138]{GT}), we have $\hat{p}\in C^\infty(\Omega\cup\Omega_1)\cap C^{1,\frac12}(\bar{\Omega})$. Hence the right-hand side of \eqref{589} lies in $C^{0,\frac12}(\Omega_0)$. By  Theorem 6.14 in \cite[p.107]{GT}, as well as uniqueness of $H^1$ solutions of $\Delta'$ on $S^2$ (modulo a constant), we get that $i^*\hat{p}\in C^{2,\frac12}(\Omega_0)$. Then Theorem 6.14 in \cite[p.107]{GT} implies that $\hat{p}\in C^{2,\frac12}(\bar{\Omega})$, hence $i^*\hat{p}\in C^{3,\frac12}({\Omega_0})$. Thanks to Theorem 6.19 in \cite[p.111]{GT}, we get that $\hat{p}\in C^{3,\frac12}(\bar{\Omega})$, then $i^*\hat{p}\in C^{4,\frac12}(\Omega_0)$, and hence $\hat{p}\in C^{4,\frac12}(\bar{\Omega})$, etc. Using this boot-strap argument, we deduce that if $\hat{p}\in H^2(\Omega)$ with $i^*\hat{p}\in H^2(\Omega_0)$ is a strong solution to \eqref{eq565}, then $\hat{p}$ is a classical solution and  belongs to $C^\infty(\bar{\Omega})$.
\end{remark}

Now we wish to show that $\hat{p}\equiv0$. The idea is to use the method of separation of variables via the spherical harmonics.

Let $u_{n,m}(y')$ be the eigenfunctions
of $-\Delta'$ on ${S}^2$ with respect to the eigenvalues
$\lambda_n=n(n+1)\ge 0, n=0,1,2,\cdots, m=-n, \cdots, -1, 0,1, \cdots, n$. (See Lemma A.4 in \cite[p.2539]{CY2013} or \cite[Section 8.4, Theorem 1$'$]{BC}.) Since $\{u_{n,m}(y')\}_{n,m}$ form a complete unit orthogonal
basis of $L^2(S^2)$, we could  write
\begin{eqnarray}\label{eq572add}
\hat{p}(y)=\sum_{n=0}^\infty\sum_{m=-n}^{n} v_{n,m}(y^0)u_{n,m}(y'),
\end{eqnarray}
with (recall that $y=(y^0, y')$)
\begin{eqnarray*}
v_{n,m}(y^0)=\int_{S^2}\hat{p}(y^0, y')u_{n,m}(y')\,\dd S(y').
\end{eqnarray*}
Here $\dd S$ is the standard Lebesgue measure on $S^2$. For $\hat{p}\in C^{k,\alpha}(\bar{\Omega})$ and $k\ge 2$, we easily deduce that  $v_{n,m}(y^0)$ belongs to $C^{k,\alpha}([r_b, r_1])$ for all $n, m$, and it is also true that the series \eqref{eq572add} converges in $C^{k-2}(\bar{\Omega})$ ({\it cf.} \cite[Section 8.6, Theorem 3]{BC}).

Substituting \eqref{eq572add} into \eqref{eq565},  each $v_{n,m}$ solves the following nonlocal ordinary differential equation:
\begin{eqnarray}\label{eq566}
&&e_1(y^0)v_{n,m}''+e_2(y^0)v_{n,m}'+(e_3(y^0)
+\lambda_n)v_{n,m}+e_4(y^0)v_{n,m}(r_b)=0,\nonumber\\ &&\qquad\qquad n=0, 1, 2, \cdots;\ \  m=-n, \cdots, n, \quad  y^0\in[r_b, r_1],
\end{eqnarray}
and the two-point boundary conditions:
\begin{eqnarray}\label{eq567}
v_{n,m}'(r_b)+\frac{\mu_7+\lambda_n}{\mu_9}v_{n,m}(r_b)=0,\qquad v_{n,m}(r_1)=0.
\end{eqnarray}
We need to find conditions to guarantee that all $v_{n,m}$ are zero.

Suppose that $v_{n,m}'(r_b)=0,$ then by uniqueness of solutions of Cauchy problem of ordinary differential equations, obviously one has that $v_{n,m}\equiv 0$ in $[r_b, r_1]$.

If $v_{n,m}(r_b)\ne0$ for some $n,m$, we set $w_n=w_{n,m}(y^0)=v_{n,m}(y^0)/v_{n,m}(r_b)$. Then it solves
\begin{eqnarray}\label{eq568}
\begin{cases}
e_1(y^0)w_{n}''+e_2(y^0)w_{n}'+(e_3(y^0)+\lambda_n)w_{n}=-e_4(y^0),\\
w_{n}(r_b)=1,\qquad w_{n}'(r_b)=-\frac{\mu_7+\lambda_n}{\mu_9},\qquad w_{n}(r_1)=0.
\end{cases}
\end{eqnarray}

\begin{definition}\label{def52}
We say a background solution $U_b$ satisfies the {\it S-Condition}, if for
each $n=0,1,2,\cdots$,  problem \eqref{eq568} does {\it not} have
a classical solution.
\end{definition}

If the background solution $U_b$ satisfies the S-Condition, then all $v_{n,m}$ are zero, hence problem \eqref{eq565} has only the trivial solution. Recall that a background  solution is determined by the five parameters: $\gamma>1, r_b\in(r_0, r_1), p_s=p_b^+(r_b)>0, \rho_s=\rho_b^+(r_b)>0, M_s=M_b^+(r_b)\in(0,1)$.
We have shown in \cite{CY2013} that almost all background solutions satisfy the S-Condition. For example, we have the following lemma ({\it cf.} Lemma 2.6 in \cite{CY2013}).

\begin{lemma}\label{lem54}
For any given $\gamma>1$, $M_s\in(0, 1)$, $\rho_s>0$ and  $p_s>0$ , there is a set $\mathcal{S}\subset(r_0, r_1)$ of at most countable infinite points such that the background solution $U_b$ determined by $\gamma, r_b\in(r_0, r_1)\setminus \mathcal{S}, \rho_s, p_s, M_s$ satisfies the S-Condition.
\end{lemma}

\begin{proof}
Using \eqref{eq35} and \eqref{eqmachode}, we  note that all the coefficients $e_1, e_2, e_3$ and $e_4$ could be solved, for the given constant $\gamma>1$ and initial data $M_s$, $\rho_s$. They are real analytic functions independent of $r_b$. Now we change the independent variables $y^0$ to $z$ given by  $z=\frac{y^0-r_b}{r_1-r_b}$. Then \eqref{eq568} becomes
\begin{eqnarray}\label{eq594add}
&&\begin{cases}
e_1((r_1-r_b)z+r_b)w_n''(z)+(r_1-r_b)e_2((r_1-r_b)z+r_b)w_n'(z)\\
+(r_1-r_b)^2
\Big(e_3((r_1-r_b)z+r_b)+\lambda_n\Big)w_n(z)\\
\qquad\qquad=-(r_1-r_b)^2
e_4((r_1-r_b)z+r_b)&z\in[0,1],\\
w_n(0)=1, w_n'(0)=-(r_1-r_b)\frac{\mu_7+\lambda_n}{\mu_9},
\end{cases}\\
&&w_n(1)=0.\label{eq595add}
\end{eqnarray}
Checking the expressions of $\mu_7$ and $\mu_9$, they are analytic with respect to $r_b$. So we see that the problem \eqref{eq594add} depends continuously on $r_b\in[r_0, r_1]$ and analytically on $r_b\in[r_0, r_1)$. Hence the unique solution $w_n$ to this Cauchy problem is also real analytic with respect to the parameter $r_b$. We write it as $w_n=w_n(z; r_b)$. Particulary, $\vartheta_n(r_b)\triangleq w_n(1; r_b)$ is continuous for $r_b\in[r_0, r_1]$ and real analytic for $r_b\in [r_0, r_1)$. For given $n=0, 1, 2, \cdots$, suppose now there are infinite numbers of  $r_b$ so that $\theta_n(r_b)=0$. Then by compactness of $[r_0, r_1]$, the function $\vartheta_n$ has a non-isolated zero point. So it must be identically zero and we have $\theta_n(r_1)=0.$ However, for $r_b=r_1$, problem \eqref{eq594add} is reduced to
$$w_n''(z)=0,\quad z\in[0,1]; \qquad w_n(0)=1, \ \ w_n'(0)=0.$$
Hence $w_n(1)=1$, namely $\vartheta_n(r_1)=1$, contradicts to our conclusion that $\vartheta_n(r_1)=0$. So for each fixed $n$, there are at most finite numbers of zeros of $\vartheta_n$. Hence there are at most countable infinite numbers of $r_b$ so that the problem \eqref{eq594add} and \eqref{eq595add} may have a solution. The conclusion for \eqref{eq568} then follows.
\end{proof}

%\begin{remark}
%By the R-H condition for the Mach number, namely,
%$$\Big(\frac{1}{{M_b^-(r_b)}}+\frac{\gamma-1}{2}\Big)
%\Big(\frac{1}{{M_b^+(r_b)}}+\frac{\gamma-1}{2}\Big)=\frac{(\gamma+1)^2}{4},$$
%we actually should prescribe $M_s=M_b^+(r_b)$ to lie in $(\frac{\gamma-1}{2\gamma}, 1)$, rather than $(0,1)$.
%\end{remark}

\subsubsection{Uniform  a priori estimate in H\"{o}lder spaces}
By considering the nonlocal term $e_4(y^0)i^*(\hat{p})$ in $\mathfrak{L}(\hat{p})$  as part of the non-homogenous term,  and applying  Theorem 1.5 in \cite[p.198]{LN} for the Venttsel problem (note that $\mu_9<0$) and Theorem 6.6 in \cite{GT} for the
Dirichlet problem, with the aid of a standard higher regularity
argument as in Theorem 6.19 of \cite{GT}, and interpolation inequalities (Lemma 6.35 in \cite[p.135]{GT}),  we infer that  any $\hat{p}\in C^{k,\alpha}(\bar{\Omega})$ ($k=2,3$) solves problem \eqref{eq564} should satisfy the estimate
\begin{eqnarray}\label{eq569}
\norm{\hat{p}}_{C^{k,\alpha}(\bar{\Omega})}\le
C\Big(\norm{\hat{p}}_{C^0(\bar{\Omega})}+\norm{f}_{C^{k-2,\alpha}(\bar{\Omega})}
+\norm{h_1}_{C^{k,\alpha}(\Omega_1)}+\norm{h_0}_{C^{k-2,\alpha}(\Omega_0)}\Big),
\end{eqnarray}
with  $C$ a constant depending only on the background solution $U_b$.

Then, by \eqref{eq569}
and a compactness argument, we have the {\it a priori} estimate:
\begin{eqnarray}\label{eq570}
\norm{\hat{p}}_{C^{k,\alpha}(\bar{\Omega})}\le
C\Big(\norm{f}_{C^{k-2,\alpha}(\bar{\Omega})}
+\norm{h_1}_{C^{k,\alpha}(\Omega_1)}+\norm{h_0}_{C^{k-2,\alpha}(\Omega_0)}\Big)
\end{eqnarray}
for any $C^{k,\alpha}$ solution of problem \eqref{eq564}, provided that the only solution to problem \eqref{eq565} is zero.

In fact, to prove \eqref{eq570}, by setting $K\triangleq\norm{f}_{C^{k-2,\alpha}(\bar{\Omega})}
+\norm{h_1}_{C^{k,\alpha}(\Omega_1)}+\norm{h_0}_{C^{k-2,\alpha}(\Omega_0)},$ we just need to show that there is a constant $C$ so that
\begin{eqnarray}\label{eq571add}
\norm{\hat{p}}_{C^{0}(\bar{\Omega})}\le
CK,
\end{eqnarray}
thanks to the inequality \eqref{eq569}. Suppose that \eqref{eq571add} is false. Then there are $\hat{p}^{(n)}\in C^{k,\alpha}(\bar{\Omega})$ so that $\norm{\hat{p}^{(n)}}_{C^{0}(\bar{\Omega})}>
nK_{n}.$ Here $K_n$ is obtained from replacing $f$ by $f^{(n)}\triangleq\mathfrak{L}(\hat{p}^{(n)})$, $h_1$ by $h_1^{(n)}\triangleq\hat{p}^{(n)}|_{\Omega_1}$, and $h_0$ by $h_0^{(n)}\triangleq-\Delta'(i^*\hat{p}^{(n)})+\mu_7i^*\hat{p}^{(n)}
+\mu_9i^*(\p_0\hat{p}^{(n)})$. By linearity of the problem, without loss of generality, we assume that for any $n$, $\norm{\hat{p}^{(n)}}_{C^{0}(\bar{\Omega})}= 1$. Then $K_n\to 0$ as $n\to\infty$ and \eqref{eq569} implies that
$\norm{\hat{p}^{(n)}}_{C^{k,\alpha}(\bar{\Omega})}\le 2C.$
From Ascoli--Arzela Lemma, there is a subsequence $\{\hat{p}^{(n_j)}\}_j$ of $\{\hat{p}^{(n)}\}_n$ so that $\hat{p}^{(n_j)}$ converges to a $\hat{p}^\sharp$ in $C^{k}(\bar{\Omega})$ ($k=2,3$). Thus by uniqueness of solution of problem \eqref{eq564}, taking $j\to\infty$,  we must have $\hat{p}^\sharp\equiv0$, contradicting to the fact that $\norm{\hat{p}^{\sharp}}_{C^0(\bar{\Omega})}\
=\lim_{j\to\infty}\norm{\hat{p}^{(n_j)}}_{C^0(\bar{\Omega})}=1$.
Thus we proved \eqref{eq571add}.

\subsubsection{Uniform a priori estimate in Sobolev spaces}
Suppose now that $\hat{p}\in H^2(\Omega)\cap H^2(\Omega_0)$, which means that $\hat{p}\in H^2(\Omega)$ and the trace of $\hat{p}$ on $\Omega_0$, namely $i^*\hat{p}$, belongs to $H^2(\Omega_0)$. Obviously our assumptions on problem \eqref{eq564} guarantee that $f\in L^2(\Omega)$, $h_1\in H^2(S^2)$, and $h_0\in L^2(S^2)$.  Then by Trace Theorem and Interpolation Inequalities of Sobolev functions ({\it cf.} Theorem 7.28 in \cite[p.173]{GT}), we have
\begin{eqnarray}
\norm{i^*\hat{p}}_{L^2(S^2)}&\le& C\norm{\hat{p}}_{H^1(\Omega)}\le
\varepsilon\norm{\hat{p}}_{H^2(\Omega)}+C(\varepsilon)\norm{\hat{p}}_{L^2(\Omega)},\\
\norm{i^*(\p_0\hat{p})}_{L^2(S^2)}&\le& \norm{i^*(\p_0\hat{p})}_{H^{\frac14}(S^2)}\le C\norm{\p_0\hat{p}}_{H^{\frac34}(\Omega)}\le C\norm{\hat{p}}_{H^{\frac74}(\Omega)}\nonumber\\
&\le&\varepsilon\norm{\hat{p}}_{H^2(\Omega)}+C(\varepsilon)\norm{\hat{p}}_{L^2(\Omega)},\quad \forall \ \varepsilon\in(0,1).
\end{eqnarray}
Now applying Theorem 8.12 in \cite[p.186]{GT}  to the boundary equation in \eqref{eq564}, we have
\begin{eqnarray}
\norm{i^*\hat{p}}_{H^2(S^2)}&\le& C\Big(\norm{i^*\hat{p}}_{L^2(S^2)}+\norm{h_0}_{L^2(S^2)}+\norm{i^*(\p_0\hat{p})}_{L^2(S^2)}\Big)\nonumber\\
&\le&C\varepsilon\norm{\hat{p}}_{H^2(\Omega)}+C'(\varepsilon)\norm{\hat{p}}_{L^2(\Omega)}+C\norm{h_0}_{L^2(S^2)}.
\end{eqnarray}
Using the same theorem to problem \eqref{eq564}, with given Dirichlet data $i^*\hat{p}$,  it follows that
\begin{eqnarray*}
&&\norm{\hat{p}}_{H^2(\Omega)}+\norm{i^*\hat{p}}_{H^2(S^2)}\le C\Big(\norm{i^*\hat{p}}_{H^2(S^2)}+\norm{h_1}_{H^2(S^2)}+\norm{\hat{p}}_{L^2(\Omega)}
+\norm{f}_{L^2(\Omega)}\Big)\nonumber\\
&\le&C\varepsilon\norm{\hat{p}}_{H^2(\Omega)}+C'(\varepsilon)\norm{\hat{p}}_{L^2(\Omega)}+
C\Big(\norm{h_0}_{L^2(S^2)}+\norm{h_1}_{H^2(S^2)}+\norm{f}_{L^2(\Omega)}\Big).
\end{eqnarray*}
Taking $\varepsilon=1/(2C)$, we get
\begin{eqnarray}\label{eq5103add}
\norm{\hat{p}}_{H^2(\Omega)}+\norm{i^*\hat{p}}_{H^2(S^2)}\le
C\Big(\norm{\hat{p}}_{L^2(\Omega)}+\norm{h_0}_{L^2(S^2)}+\norm{h_1}_{H^2(S^2)}+\norm{f}_{L^2(\Omega)}\Big).
\end{eqnarray}
With an argument similar to the proof of \eqref{eq570} ({\it cf.} Lemma 9.17 in \cite[p.242]{GT}), we deduce the {\it a priori} estimate
\begin{eqnarray}\label{eq5102add}
\norm{\hat{p}}_{H^2(\Omega)}+\norm{i^*\hat{p}}_{H^2(S^2)}\le
C\Big(\norm{h_0}_{L^2(S^2)}+\norm{h_1}_{H^2(S^2)}+\norm{f}_{L^2(\Omega)}\Big),
\end{eqnarray}
provided that the S-Condition holds. Here the constant $C$ depends only on the background solution $U_b$ and $r_b, r_1$.

%
%As a matter of fact, to prove \eqref{eq5102add}, by \eqref{eq5103add}, we only need to show that there is a constant $C$ so that $\norm{\hat{p}}\le CK$ for any $\hat{p}\in H^2(\Omega)\cap H^2(\Omega_0)$, where $K\triangleq\norm{h_0}_{L^2(S^2)}+\norm{h_1}_{H^2(S^2)}+\norm{f}_{L^2(\Omega)}.$
%Suppose that this is not true. Then for any $m\in\mathbb{N}$, there is a $\hat{p}_m\in  H^2(\Omega)\cap H^2(\Omega_0)$ with $\norm{\hat{p}_m}_{L^2(\Omega)}=1$ so that the corresponding $K_m$ is less than $1/m.$ Then \eqref{eq5103add} implies that
%$$\norm{\hat{p}_m}_{H^2(\Omega)}+\norm{i^*(\hat{p}_m)}_{H^2(S^2)}\le
%C(1+\frac{1}{m})\le 2C.$$
%So we get a subsequence, which, without loss of generality,  still denoted to be $\{\hat{p}_m\}$, converges weakly to a $\hat{p}$ in $H^2(\Omega)\cap H^2(\Omega_0)$. Note that the coefficients in the problem \eqref{eq564} are analytic, so we may write it in divergence form. Since we may assume that $\{\hat{p}_m\}$ itself converges strongly in $H^1(\Omega)\cap H^1(\Omega_0)$ by compact embedding $H^2(\Omega)\subset\subset H^1(\Omega)$, we see that $\hat{p}$ is a $H^1$ weak solution of problem \eqref{eq565}. Then by $H^2$ regularity (Theorem 8.8 and Theorem 8.12 in \cite[p.183, p.187]{GT}), we infer that $\hat{p}\in H^2(\Omega)\cap H^2(\Omega_0)$, hence $\hat{p}=0$ by S-Condition. However, by $\norm{\hat{p}_m}_{L^2(\Omega)}=1$, one should have $\norm{\hat{p}}_{L^2(\Omega)}=1$, which is a contradiction.

\subsubsection{Approximate solutions}
We now use spherical harmonic expansion to establish a family of approximate solutions to problem \eqref{eq564}.

For simplicity, we take $h_1=0$  in the sequel. There is no loss of generality, since  this accounts we replace $\hat{p}$ by $\hat{p}-h_1$, and $f$ by $f-\mathfrak{L}h_1$, $h_0$ by $h_0+\Delta'h_1-\mu_7h_1$ in problem \eqref{eq564}.

We also set $\{f^{(k)}\}_k$ to be a sequence of $C^\infty(\bar{\Omega})$ functions that converges to $f$ in $C^{k-2,\alpha}(\bar{\Omega})$, and $\{h^{(k)}_0\}_k\subset C^\infty(S^2)$ converges to $h_0$ in $C^{k-2,\alpha}(S^2)$. Now for fixed $k$, we consider problem \eqref{eq564}, with $f$ there replaced by $f^{(k)}$, and $h_0$ replaced by $h^{(k)}_0$.

Suppose that
\begin{gather}
f^{(k)}(y)=\sum_{n=0}^\infty\sum_{m=-n}^n f^{(k)}_{n,m}(y^0)u_{n,m}(y'),\quad
h_0^{(k)}(y')=\sum_{n=0}^\infty\sum_{m=-n}^n (h_0^{(k)})_{n,m}u_{n,m}(y').\label{eq599}
\end{gather}
Then for $\hat{p}$ given by \eqref{eq572add}, each $v_{n,m}(y^0)$ should solve the following two-point boundary value problem of an ordinary differential equation containing a nonlocal term:
\begin{eqnarray}\label{eq5100add}
\begin{cases}
L_{n,m}(v_{n,m})\triangleq v_{n,m}''(y^0)+p(y^0)v_{n,m}'(y^0)+q_n(y^0)v_{n,m}(y^0)\\
\qquad\quad\quad\quad=r(y^0)v_{n,m}(r_b)+\tilde{f}_{n,m}(y^0)&\quad y^0\in[r_b, r_1],\\
v_{n,m}'(r_b)+a_nv_{n,m}(r_b)=h_{n,m},\\ v_{n,m}(r_1)=0.
\end{cases}
\end{eqnarray}
Here we define
\begin{gather*}
p(y^0)=\frac{e_2(y^0)}{e_1(y^0)},\ \ q_n(y^0)=\frac{e_3(y^0)+\lambda_n}{e_1(y^0)},\quad r(y^0)=-\frac{e_4(y^0)}{e_1(y^0)},\ \ \tilde{f}_{n,m}(y^0)=\frac{f^{(k)}_{n,m}(y^0)}{e_1(y^0)},\\
a_n=\frac{\mu_7+\lambda_n}{\mu_9},\quad h_{n,m}=\frac{(h_0^{(k)})_{n,m}}{\mu_9}.
\end{gather*}
We will show that this problem is uniquely solvable.

Let $\varphi^1_{n,m}$ be the unique solution of the Cauchy problem
$$L_{n,m}(v_{n,m})=0,\quad v_{n,m}(r_b)=1,\quad v_{n,m}'(r_b)=0,$$
and $\varphi^2_{n,m}$ the unique solution to
$$L_{n,m}(v_{n,m})=0,\quad v_{n,m}(r_b)=0,\quad v_{n,m}'(r_b)=1.$$
Then by standard theory of linear ordinary differential equations,  a general solution  to problem
\begin{eqnarray}\label{eq5100add2}
\begin{cases}
L_{n,m}(v_{n,m})=r(y^0)v_{n,m}(r_b)+\tilde{f}_{n,m}(y^0)&\quad y^0\in[r_b, r_1],\\
v_{n,m}(r_b)=c_1, \qquad v_{n,m}'(r_b)=c_2
\end{cases}
\end{eqnarray}
is given by
\begin{eqnarray}\label{eq5101add}
v_{n,m}(y^0)=c_1\varphi^1_{n,m}(y^0)+c_2\varphi^2_{n,m}(y^0)+\int_{r_b}^{y^0}K(y^0, s) \Big(c_1r(s)+\tilde{f}_{n,m}(s)\Big)\,\dd s,
\end{eqnarray}
where
$$K(t,s)\triangleq \frac{\varphi^1_{n,m}(s)\varphi^2_{n,m}(t)-\varphi^1_{n,m}(t)\varphi^2_{n,m}(s)}
{\varphi^1_{n,m}(s)(\varphi^2_{n,m})'(s)-(\varphi^1_{n,m})'(s)\varphi^2_{n,m}(s)}.$$
We claim that there are constants $c_1$ and $c_2$ so that
\begin{gather}\label{eq5103add}\begin{cases}
c_2+a_nc_1=h_{n,m},\\
\left(\varphi^1_{n,m}(r_1)+\int_{r_b}^{r_1}K(r_1,s)r(s)\,\dd s\right)c_1+\varphi^2_{n,m}(r_1)c_2=-\int_{r_b}^{r_1}K(r_1,s)\tilde{f}_{n,m}(s)\,\dd s.\end{cases}
\end{gather}
Actually, this is a linear algebraic system and we know that, under the S-Condition, the homogeneous system has only the trivial solution. So there is one and only one pair $(c_1, c_2)$ solves \eqref{eq5103add}. Hence \eqref{eq5101add} gives the unique solution to problem \eqref{eq5100add}. Note that $\tilde{f}_{n,m}\in C^{\infty}([r_b, r_1])$ as  $f^{(k)}\in C^{\infty}(\bar{\Omega})$, and the coefficients $p, q_n, r$ are all real analytic, so the solution $v_{n,m}(y^0)$ belongs to  $C^{\infty}([r_b, r_1]).$

Now for $N\in\mathbb{N}$, we define
\begin{gather*}
\hat{p}_{N}(y)=\sum_{n=0}^N\sum_{m=-n}^n v_{n,m}(y^0)u_{n,m}(y'), \\
f^{(k)}_{N}(y)=\sum_{n=0}^N\sum_{m=-n}^n f^{(k)}_{n,m}(y^0)u_{n,m}(y'),\quad (h_0^{(k)})_{N}(y')=\sum_{n=0}^N\sum_{m=-n}^n (h_0^{(k)})_{n,m}u_{n,m}(y').
\end{gather*}
Apparently  $\hat{p}_{N}, \ f^{(k)}_N\in C^{\infty}(\bar{\Omega})$, and $(h_0^{(k)})_{N}\in C^\infty(S^2)$. It is also easy to check that $\hat{p}_{N}$ solves the following problem:
\begin{eqnarray}\label{eq5106add}
\begin{cases}
\mathfrak{L}(\hat{p}_N)=f^{(k)}_N& \text{in}\quad \Omega,\\
\hat{p}_N=0 &\text{on}\quad \Omega_1,\\
-\Delta'(i^*\hat{p}_N)+\mu_7(i^*{\hat{p}_N})+\mu_9(i^*\p_0\hat{p}_N)=(h_0^{(k)})_N &\text{on}\quad \Omega_0.
\end{cases}
\end{eqnarray}

\subsubsection{Existence}
By the estimate \eqref{eq5102add}, for any $N,M\in\mathbb{N}$ with $N<M$, there holds
\begin{multline*}
\norm{\hat{p}_M-\hat{p}_N}_{H^2(\Omega)}+\norm{i^*(\hat{p}_M-\hat{p}_N)}_{H^2(\Omega_0)}\le C\Big(\norm{(h_0^{(k)})_M-(h_0^{(k)})_N}_{L^2(S^2)}+\norm{f^{(k)}_M-f^{(k)}_N}_{L^2(\Omega)}\Big).
\end{multline*}
Recall that $(h_0^{(k)})_N\to h_0^{(k)}$ in $L^2(S^2)$ and $f^{(k)}_N\to f^{(k)}$ in $L^2(\Omega)$, we infer that  $\{\hat{p}_N\}$  (respectively $i^*{\hat{p}_N}$) is a Cauchy sequence in $H^2(\Omega)$ (respectively $H^2(\Omega_0)$). So there is a $\hat{p}^{(k)}\in H^2(\Omega)$ (respectively $q^{(k)}\in H^2(\Omega_0)$) and $\hat{p}_N\to \hat{p}^{(k)}$ in $H^2(\Omega)$ (respectively $i^*\hat{p}_N\to q^{(k)}$ in $H^2(\Omega_0)$)  as $N\to \infty$. By continuity of trace operator, we conclude that $q^{(k)}=i^*\hat{p}^{(k)}$. Taking the limit $N\to\infty$ in problem  \eqref{eq5106add}, one sees that $\hat{p}^{(k)}$ is a $H^2(\Omega)\cap H^2(\Omega_0)$ solution to problem \eqref{eq564}, with $f$ there replaced by $f^{(k)}$, $h_0$ replaced by $h_0^{(k)}$. Then by the same arguments as in Remark \ref{rm55}, $\hat{p}^{(k)}\in C^{\infty}(\bar{\Omega})$ and of course it satisfies the estimate \eqref{eq570}.

Now for the approximate solutions $\{\hat{p}^{(k)}\}_k$, we use the estimate \eqref{eq570} to infer that
\begin{eqnarray*}
\norm{\hat{p}^{(k)}}_{C^{k,\alpha}(\bar{\Omega})}&\le& C\Big(\norm{f^{(k)}}_{C^{k-2,\alpha}(\bar{\Omega})}+\norm{h_0^{(k)}}_{C^{k-2,\alpha}(\Omega_0)}\Big)\nonumber\\
&\le&C\Big(\norm{f}_{C^{k-2,\alpha}(\bar{\Omega})}+\norm{h_0}_{C^{k-2,\alpha}(\Omega_0)}\Big).
\end{eqnarray*}
Hence by Ascoli--Arzela Lemma, there is a subsequence of $\{\hat{p}^{(k)}\}$ that converges to some $\hat{p}\in C^{k,\alpha}(\bar{\Omega})$ in the norm of $C^{k}(\bar{\Omega})$. Taking limit with respect to this subsequence in the boundary value problems of $\hat{p}^k$, we easily see that $\hat{p}$ is a classical solution to problem \eqref{eq564}. Therefore, we proved the following lemma.

\begin{lemma}\label{lem55}
Suppose that the S-Condition holds. Then problem \eqref{eq564} has one and only one solution in $C^{k,\alpha}(\bar{\Omega})$, and it satisfies the estimate \eqref{eq570}.
\end{lemma}

\subsection{Solvability of Problem (T4)}\label{sec58}
We now use Banach fixed point theorem to solve the transonic shock problem (T4), provided that the background solution $U_b$ satisfies the S-Condition.

\subsubsection{The iteration sets}\label{sec581}
Let $\sigma_0$  be a positive constant to be specified later, and
\begin{eqnarray*} \label{258}
\mathcal{K}_\sigma\triangleq\left\{\psi\in C^{4,\alpha}({S}^2)\,:\,
\norm{\psi-r_b}_{C^{4,\alpha}({S}^2)}\le
\sigma\le\sigma_0\right\}
\end{eqnarray*}
be the set of possible shock-front.  For any given $\psi\in \mathcal{K}_\sigma$, its position $r^p$ and
profile $\psi^p$ also satisfy
\begin{eqnarray*}
|{r^p-r_b}|\le\sigma,\qquad \norm{\psi^p}_{C^{4,\alpha}(\mathbf{S}^2)}\le 2\sigma.
\end{eqnarray*}
%The first follows from
%\begin{eqnarray*}
%r^p-r_b&=&\frac{1}{4\pi}\int_{S^2}(\psi-r_b)\vol\le\frac{1}{4\pi}\int_{S^2}|\psi-r_b|\vol \le\sigma,\\
%r_b-r^p&=&\frac{1}{4\pi}\int_{S^2}(r_b-\psi)\vol\le\frac{1}{4\pi}\int_{S^2}|\psi-r_b|\vol \le\sigma,
%\end{eqnarray*}
%and then triangle inequality implies the second estimate.
We write the set of possible variations of the subsonic flows as
\begin{eqnarray*}
O_\delta\triangleq\Big\{\check{U}=(\check{E},\check{p}, \check{s}, \check{u}')\,:\,
\norm{\check{U}}_{3}+\norm{i^*\check{U}}_{C^{3,\alpha}(S^2)}\le \delta\le\delta_0\Big\},
\end{eqnarray*}
with  $\delta_0$ a constant to be chosen. The norm $\norm{\cdot}_k$ appeared here is defined by \eqref{eq452}, with $\mathcal{M}$ there replaced by $\Omega$.

Given $U^-$ satisfying
\eqref{55015}, for any $\psi\in \mathcal{K}_\sigma$ and $\check{U}\in
O_\delta$, we construct a mapping
$$\mathcal{T}: \mathcal{K}_\sigma\times
O_\delta\rightarrow \mathcal{K}_\sigma\times O_\delta, \quad (\psi,\check{U})\mapsto (\hat{\psi},\hat{U})
$$
as follows. One should note that a fixed point of this mapping is a solution to Problem (T4).

\subsubsection{Construction of iteration mapping}\label{sec531}
For any $\psi\in \mathcal{K}_\sigma$ and $\check{U}\in
O_\delta$, we set $$U=\check{U}+U_b^+(\frac{r_1-\psi}{r_1-r_b}(y^0-r_b)+\psi(y'), y').$$
Then with the known supersonic flow $U^-$, we could specify all the higher-order terms $f$ and $\bar{g}_k$ appeared in Problem (T4).
\medskip

\paragraph{\bf Bernoulli constant $E$.\ }
We firstly solve the linear problem ({\it cf.} \eqref{eq537})
\begin{eqnarray}\label{eq587}
\begin{cases}
D_u\hat{E}=0&\text{in}\quad \Omega,\\
\hat{E}=E^--E_b&\text{on}\quad \Omega_0.
\end{cases}
\end{eqnarray}
Note here that $E^-$ is determined by $U^-$, and  $E_b$ is a constant. So by \eqref{55015}, it holds
\begin{eqnarray}\label{eq5116}
\norm{i^*(E^--E_b)}_{C^{3,\alpha}(S^2)}
=\norm{(E^--E_b)|_{S^\psi}}_{C^{3,\alpha}(S^2)}\le C_0\varepsilon.
\end{eqnarray}
Hence we could easily get the unique existence of $\hat{E}\in C^{2,\alpha}(\bar{\Omega})$ (note that $u\in C^{2,\alpha}(\bar{\Omega})$) with
\begin{eqnarray}\label{eq588}
\norm{\hat{E}}_{C^{2,\alpha}(\bar{\Omega})}\le \norm{i^*(E^--E_b)}_{C^{2,\alpha}(S^2)}\le C_0\varepsilon.
\end{eqnarray}
The constant $C_0$ appeared here and below depends only on the background solution and $r_b, r_1$.

\medskip
\paragraph{\bf Pressure $p$.\ }
We then consider the following problem on $\hat{p}$ ({\it cf.} \eqref{eq559}):
\begin{eqnarray}
\begin{cases}\label{eq589}
\mathfrak{L}(\hat{p})=e_5(y^0)\hat{E}+f(U,\psi, DU,D^2p, D\psi, D^2\psi)& \text{in}\quad \Omega,\\
\hat{p}=p_1-p_b^+ &\text{on}\quad \Omega_1,\\
-\Delta'(i^*\hat{p})+\mu_7(i^*{\hat{p}})+\mu_9(i^*\p_0\hat{p})\\
\qquad=\bar{g}_8(U,U^-,\psi,DU,D\psi,D^2U,D^2\psi,D^3\psi).
\end{cases}
\end{eqnarray}
Here the non-homogeneous terms $f$ and $\bar{g}_8$  are determined by $\psi\in \mathcal{K}_\sigma$ and $U=\check{U}+U_b$, with $\check{U}\in O_\delta$, and $\hat{E}$ is solved from \eqref{eq587}.
Then, since we assumed that the S-Condition holds, by Lemma \ref{lem55}, we could solve uniquely one $\hat{p}\in C^{3,\alpha}(\bar{\Omega})$ and it satisfies the following estimate:
\begin{eqnarray}\label{eq590}
\norm{\hat{p}}_{C^{3,\alpha}(\bar{\Omega})}\le C_0\Big(\norm{f}_{C^{1,\alpha}(\bar{\Omega})}+\norm{\bar{g}_8}_{C^{1,\alpha}(S^2)}
+\norm{\hat{E}}_{C^{1,\alpha}(\bar{\Omega})}+\norm{p_1-p_b^+}_{C^{3,\alpha}(\Omega_1)}\Big).
\end{eqnarray}
Checking the definitions of $f$ and $\bar{g}_8$, the right-hand side is finite; actually we have (see Lemma \ref{lem61} in the Appendix)
\begin{eqnarray}\label{eq591}
\norm{f}_{C^{1,\alpha}(\bar{\Omega})}\le C(\delta^2+\sigma^2+\varepsilon^2+\varepsilon), \qquad \norm{\bar{g}_8}_{C^{1,\alpha}(S^2)}\le C(\delta^2+\varepsilon^2+\sigma^2+\varepsilon).
\end{eqnarray}
So combining \eqref{eq588}, \eqref{eq58add} and \eqref{eq590}, \eqref{eq591}, one infers that
\begin{eqnarray}\label{eq592}
\norm{\hat{p}}_{C^{3,\alpha}(\bar{\Omega})}\le C_0\Big(\delta^2+\sigma^2+\varepsilon^2+\varepsilon\Big).
\end{eqnarray}

\medskip
\paragraph{\bf Update shock-front $\hat{\psi}$.\ }
With the specified higher-order terms $\bar{g}_5$ and $\bar{g}_7$, and $\hat{p}$ solved from \eqref{eq589}, we now set ({\it cf.} \eqref{eq539})
\begin{eqnarray}
&&\begin{cases}\label{eq593}\displaystyle
\hat{r}^p-r_b=-\frac{1}{4\pi\mu_6}\int_{{S}^2}\Big(\mu_5\,
i^*(\p_0\hat{p})+\bar{g}_5(U,U^-,\psi,DU,D\psi)\Big)\, \vol,\\
\displaystyle\hat{\psi}^p=\frac{1}{\mu_2}\left(i^*\hat{p}-{\mu_8}
\int_{{S}^2}i^*(\p_0\hat{p})\,\vol+\bar{g}_7(U,U^-,\psi,D\psi)\right),\\
\hat{\psi}=\hat{\psi}^p+\hat{r}^p.
\end{cases}
\end{eqnarray}
It follows easily that (using \eqref{eq592})
\begin{eqnarray}\label{eq594}
\norm{\hat{\psi}^p}_{C^0(S^2)}+|\hat{r}^p-r_b|&\le& C_0\Big(\norm{\bar{g}_5}_{C^0(S^2)}+\norm{\bar{g}_7}_{C^0(S^2)}
+\norm{\hat{p}}_{C^{1}(\bar{\Omega})}\Big)\nonumber\\
&\le&
C_0\Big(\delta^2+\sigma^2+\varepsilon^2+\varepsilon\Big).
\end{eqnarray}

For the $C^{4,\alpha}$ estimate of $\hat{\psi}^p$, we note that $i^*\hat{p}$ solves the third  equation in \eqref{eq589}, hence $\hat{\psi}^p$ solves the following elliptic equation on $S^2$ ({\it cf.} \eqref{eq525}):
\begin{eqnarray}\label{eq595}
-\Delta'\hat{\psi}^p+\mu_7\hat{\psi}^p=\mu_0\mu_6(\hat{r}^p-r_b)+\mu_0\mu_5\,i^*\p_0\hat{p}
+\bar{g}_6(U,U^-,\psi,DU^-,DU,D\psi,D^2\psi).
\end{eqnarray}
Standard Schauder estimates \cite[Chapter 6]{GT} yield that
\begin{eqnarray}\label{eq5125}
\norm{\hat{\psi}^p}_{C^{4,\alpha}(S^2)}&\le& C_0\Big(\norm{\bar{\psi}^p}_{C^0(S^2)}
+\norm{\bar{g}_6}_{C^{2,\alpha}(S^2)}
+\norm{\hat{p}}_{C^{3,\alpha}(S^2)}+|\hat{r}^p-r_b|\Big)\nonumber\\
&\le&C_0\Big(\delta^2+\sigma^2+\varepsilon^2+\varepsilon\Big).
\end{eqnarray}
One then has
\begin{eqnarray}\label{eq596}
\norm{\hat{\psi}-r_b}_{C^{4,\alpha}(S^2)}\le C_0\Big(\delta^2+\sigma^2+\varepsilon^2+\varepsilon\Big).
\end{eqnarray}

We also need to show that
\begin{eqnarray}\label{eq596add}
\int_{S^2}\hat{\psi}^p\,\vol=0.
\end{eqnarray}
In fact, integrating  \eqref{eq595} on $S^2$, and recall $\bar{g}_6=\mu_0\bar{g}_5+\dd^*\bar{g_0}$, using Divergence Theorem and definition of $\hat{r}^p-r_b$ in \eqref{eq593}, we  have directly \eqref{eq596add}.

\medskip
\paragraph{\bf Entropy $A(s)$.\ }
Now we  solve ({\it cf.} \eqref{eq540}, and note that ${A(s_b^+)}$ is constant)
\begin{eqnarray}\label{eq597}
\begin{cases}
D_u\widehat{A(s)}=0&\text{in}\quad \Omega,\\
i^*(\widehat{A(s)})=\mu_4\,(\hat{\psi}-r_b)+\bar{g}_4(U,U^-,\psi,D\psi)
\end{cases}
\end{eqnarray}
to obtain the unique solution $\widehat{A(s)}$. It also holds
\begin{eqnarray}\label{eq598}
&&\norm{i^*\widehat{A(s)}}_{C^{3,\alpha}(S^2)}\le C_0\Big(\norm{\hat{\psi}-r_b}_{C^{3,\alpha}(S^2)}+\norm{\bar{g}_4}_{C^{3,\alpha}(S^2)}\Big)%\nonumber\\&&\qquad\qquad\qquad\qquad\qquad\qquad
\le C_0\Big(\delta^2+\sigma^2+\varepsilon^2+\varepsilon\Big),\\
&&\norm{\widehat{A(s)}}_{C^{2,\alpha}(\bar{\Omega})}\le C_0\norm{i^*\widehat{A(s)}}_{C^{2,\alpha}(S^2)}\le
C_0\Big(\delta^2+\sigma^2+\varepsilon^2+\varepsilon\Big).\label{eq5130}
\end{eqnarray}

\medskip
\paragraph{\bf Tangential velocity field ${u}_0'$ on $\Omega_0$.}
Next we study the tangential velocity $u'_0$ on $\Omega_0$ by ({\it cf.} \eqref{eq541})
\begin{eqnarray}
\begin{cases}\label{eq5100}
\dd(\bar{u}_0')=-\frac{1}{\mu_0}\dd \bar{g}_0(U,U^-, \psi, D\psi),\\
\dd^*(\bar{u}_0')=\mu_5\,
i^*(\p_0\hat{p})+\mu_6\,
\hat{\psi}^p+\mu_6\,(\hat{r}^p-r_b)+\bar{g}_5(U,U^-,\psi,DU,D\psi)
\end{cases}\qquad \text{on}\  S^2.
\end{eqnarray}
By \eqref{eq596add} and the first equation in \eqref{eq593},  we see the requirements in Theorem \ref{them62} are fulfilled. So one can solve a unique  $\bar{u}_0'$ on $S^2$, and the  following estimate is valid:
\begin{eqnarray}\label{eq5101}
\norm{\bar{u}_0'}_{C^{3,\alpha}(S^2)}&\le& C_0\Big(\norm{\hat{p}}_{C^{3,\alpha}(\bar{\Omega})}
+\norm{\hat{\psi}-r_b}_{C^{2,\alpha}(S^2)}
+\norm{\bar{g}_0}_{C^{3,\alpha}(S^2)}
+\norm{\bar{g}_5}_{C^{2,\alpha}(S^2)}\Big)\nonumber\\
&\le& C_0 \Big(\delta^2+\sigma^2+\varepsilon^2+\varepsilon\Big).
\end{eqnarray}
Note that $\bar{u}_0'$ is a 1-form on $S^2$. Let
${u}_0'=u_0^\beta\p_\beta$ be the associated vector field on $S^2$.
Then according to \eqref{eq5101}, $\hat{u}_0^\beta=u_0^\beta$
satisfies the following inequality
\begin{eqnarray}\label{eq5102}
\sum_{\beta=1}^2 \norm{\hat{u}_0^\beta}_{C^{3,\alpha}(S^2)}\le C_0
\Big(\delta^2+\sigma^2+\varepsilon^2+\varepsilon\Big).
\end{eqnarray}
\medskip

\paragraph{\bf Tangential velocity $\hat{u}'$ in $\Omega$.}
Finally, we solve the tangential velocity ${u}'$ in $\Omega$ through ({\it cf.} \eqref{eq543add})
\begin{eqnarray}
\begin{cases}\label{eq5103}
u^j\p_j\hat{u}^1+2\frac{u_b^0(y^0)}{y^0}\hat{u}^1=-\frac{1}{\rho (y^0)^2}\p_1 \hat{p} +\bar{f}_1(U, \psi, D\psi) &\text{in}\quad \Omega,\\
\hat{u}^1=\hat{u}_0^1&\text{on}\quad \Omega_0.
\end{cases}
\end{eqnarray}
Here the Cauchy data $\hat{u}_0^\beta$ on $\Omega_0$ is solved from
\eqref{eq5100}.

We obtain a unique $\hat{u}^1$ in $\Omega$ by Theorem \ref{thm61}. From \eqref{eq590} and \eqref{eq5102},  it holds that
\begin{eqnarray}\label{eq5104}
\norm{\hat{u}^1}_{C^{2,\alpha}(\bar{\Omega})}&\le& C_0\Big(\norm{\hat{p}}_{C^{3,\alpha}(\bar{\Omega})}+\norm{\bar{f}_1}_{C^{2,\alpha}(\bar{\Omega})}
+\norm{\hat{u}_0^1}_{C^{2,\alpha}(\Omega_0)}\Big)\nonumber\\
&\le&C_0\Big(\delta^2+\sigma^2+\varepsilon^2+\varepsilon\Big).
\end{eqnarray}
By changing the coordinates system, we get $\hat{u}^2$ and it also satisfies an estimate like  \eqref{eq5104}.

\medskip
\paragraph{\bf Conclusion.\ } From the above six steps, we get uniquely one pair $(\hat{U}, \hat{\psi})$ and it follows from \eqref{eq5116}, \eqref{eq588}, \eqref{eq592}, \eqref{eq596}, \eqref{eq598},
\eqref{eq5130}, \eqref{eq5102} and \eqref{eq5104} that
\begin{eqnarray}\label{eq5105}
\norm{\hat{\psi}-r_b}_{C^{4,\alpha}(S^2)}+\norm{\hat{U}}_3+\norm{i^*\hat{U}}_{C^{3,\alpha}(S^2)}\le \tilde{C}\Big(\delta^2+\sigma^2+\varepsilon^2+\varepsilon\Big).
\end{eqnarray}
Here $\tilde{C}$ is a constant depending only on the background solution and $r_b, r_1$.
Now we choose $C_*=4\tilde{C}$ and $\varepsilon_0\le \min\Big\{ {1}/{(16\tilde{C}^2)}, 1, h^\sharp/(8\tilde{C})\Big\}$.
Then, for $\delta=\sigma=C_*\varepsilon$, we have
$\tilde{C}\Big(\delta^2+\sigma^2+\varepsilon^2+\varepsilon\Big)\le\delta,\,
\forall \varepsilon\in(0,\varepsilon_0),$
and the  estimate \eqref{eq5105} shows
that $\hat{U}\in O_{C_*\varepsilon}$ and $\hat{\psi}\in \mathcal{K}_{C_*\varepsilon}.$
Hence we construct the desired mapping $\mathcal{T}$ on $\mathcal{K}_{C_*\varepsilon}\times O_{C_*\varepsilon}$.

\subsubsection{Contraction of iteration mapping}\label{sec532}
What left is to show that the mapping $$\mathcal{T}: \mathcal{K}_{C_*\varepsilon}\times
O_{C_*\varepsilon}\rightarrow \mathcal{K}_{C_*\varepsilon}\times O_{C_*\varepsilon}, \quad (\psi,\check{U})\mapsto (\hat{\psi},\hat{U})
$$
is a contraction in the sense that
\begin{multline}\label{eq5106}
\norm{\hat{\psi}^{(1)}-\hat{\psi}^{(2)}}_{C^{3,\alpha}(S^2)}+\norm{\hat{U}^{(1)}
-\hat{U}^{(2)}}_2+\norm{i^*(\hat{U}^{(1)}
-\hat{U}^{(2)})}_{C^{2,\alpha}(S^2)}\\
\le\frac12\Big(\norm{{\psi}^{(1)}-{\psi}^{(2)}}_{C^{3,\alpha}(S^2)}+\norm{\check{U}^{(1)}
-\check{U}^{(2)}}_2+\norm{i^*(\check{U}^{(1)}
-\check{U}^{(2)})}_{C^{2,\alpha}(S^2)}\Big)\triangleq\frac12Q,
\end{multline}
provided that $\varepsilon_0$ is further small (depending only on the background solution). Here for $j=1,2$, and any $\psi^{(j)}\in \mathcal{K}_{{C_*\varepsilon}}, \check{U}^{(j)}\in O_{{C_*\varepsilon}}$, we have defined $(\hat{\psi}^{(j)}, \hat{U}^{(j)})=\mathcal{T}(\psi^{(j)}, \check{U}^{(j)})$.

To prove \eqref{eq5106}, we set $\hat{\psi}=\hat{\psi}^{(1)}-\hat{\psi}^{(2)}$, and $\hat{U}=\hat{U}^{(1)}-\hat{U}^{(2)}$.   For $k=1,2$, we also use the notations
\begin{gather*}
(U^-)^{(k)}=\left.U^-\right|_{S^{\psi^{(k)}}}, \quad i^*(U_b^+)^{(k)}=\left.U_b^+\right|_{S^{\psi^{(k)}}},\\ (U_b^+)^{(k)}=(U_b^+)(\frac{r_1-\psi^{(k)}(y')}{r_1-r_b}(y^0-r_b)+\psi^{(k)}(y'), y'),\quad U^{(k)}=\check{U}^{(k)}+(U_b^+)^{(k)}.\end{gather*}
By \eqref{55015} and analyticity of $U_b^\pm$, the mean value theorem implies that there is a constant $C$ depending only on the background solution so that
\begin{eqnarray}
&&\norm{i^*(U_b^+)^{(1)}-i^*(U_b^+)^{(2)}}_{C^{k,\alpha}(S^2)}\le C\norm{\psi^{(1)}-\psi^{(2)}}_{C^{k,\alpha}(S^2)},\label{eq5139}\\
&&\norm{(U_b^+)^{(1)}-(U_b^+)^{(2)}}_{C^{k,\alpha}{(\bar{\Omega}})}\le C\norm{\psi^{(1)}-\psi^{(2)}}_{C^{k,\alpha}(S^2)},\quad k=1,2,3,4,\label{eq5140}\\
&&\norm{(U^--U_b^-)^{(1)}-(U^--U_b^-)^{(2)}}_{C^{2,\alpha}(S^2)}\le C\varepsilon\norm{\psi^{(1)}-\psi^{(2)}}_{C^{2,\alpha}(S^2)}.\label{eq5141}
\end{eqnarray}

\noindent{\it Step 1.} We note that $\hat{E}$ solves the following problem ({\it cf.} \eqref{eq587})
\begin{eqnarray*}
\begin{cases}
D_{u^{(1)}}\hat{E}+D_{u^{(1)}-u^{(2)}}\hat{E}^{(2)}=0&\text{in}\ \ \Omega,\\
\hat{E}=E^-|_{S^{\psi^{(1)}}}-E^-|_{S^{\psi^{(2)}}} &\text{on}\ \ \Omega_0.
\end{cases}
\end{eqnarray*}
By mean value theorem, \eqref{eq21} and \eqref{55015}, there holds
\begin{eqnarray}\label{eq5143}
\norm{i^*\hat{E}}_{C^{2,\alpha}(S^2)}&\le&\norm{\p_0E^-}_{C^{2,\alpha}(\overline{\mathcal{M}})}
\norm{\psi^{(1)}-\psi^{(2)}}_{C^{2,\alpha}(S^2)}\nonumber\\
&\le&C\norm{{(u^\beta)}^- \p_\beta E^-}_{C^{2,\alpha}(\overline{\mathcal{M}})}\norm{\psi^{(1)}-\psi^{(2)}}_{C^{2,\alpha}(S^2)}\nonumber\\
&\le& C\varepsilon\norm{\psi^{(1)}-\psi^{(2)}}_{C^{2,\alpha}(S^2)}.
\end{eqnarray}
Then using Theorem \ref{thm61} and \eqref{eq5140}, one has
\begin{eqnarray}\label{eq5144}
\norm{\hat{E}}_{C^{1,\alpha}(\bar{\Omega})}&\le&\norm{i^*\hat{E}}_{C^{1,\alpha}(S^2)}
+\norm{u^{(1)}-u^{(2)}}_{C^{1,\alpha}(\bar{\Omega})}\norm{\hat{E}^{(2)}}_{C^{2,\alpha}(\bar{\Omega})}\nonumber\\
&\le& C\varepsilon\norm{\psi^{(1)}-\psi^{(2)}}_{C^{2,\alpha}(S^2)}+
CC_*\varepsilon\left(\norm{\check{U}^{(1)}-\check{U}^{(2)}}_{C^{1,\alpha}(\bar{\Omega})}
+\norm{{U_b^+}^{(1)}-{U_b^+}^{(2)}}_{C^{1,\alpha}(\bar{\Omega})}\right)\nonumber\\
&\le&C\varepsilon\left(\norm{\psi^{(1)}-\psi^{(2)}}_{C^{2,\alpha}(S^2)}+
\norm{\check{U}^{(1)}-\check{U}^{(2)}}_{2}\right)\le C\varepsilon Q.
\end{eqnarray}

\noindent {\it Step 2.} Next we seek an estimate of $\hat{p}$, which solves  ({\it cf.} \eqref{eq589})
\begin{eqnarray}
\begin{cases}\label{eq5145}
\mathfrak{L}(\hat{p})=e_5(y^0)\hat{E}+f^{(1)}-f^{(2)}& \text{in}\quad \Omega,\\
\hat{p}=0 &\text{on}\quad \Omega_1,\\
-\Delta'(i^*\hat{p})+\mu_7(i^*{\hat{p}})+\mu_9(i^*\p_0\hat{p})=\bar{g}_8^{(1)}-\bar{g}_8^{(2)}.
\end{cases}
\end{eqnarray}
Here for $k=1,2$,
\begin{eqnarray*}
&&f^{(k)}=f(U^{(k)},\psi^{(k)}, DU^{(k)},D^2p^{(k)}, D\psi^{(k)}, D^2\psi^{(k)}),\\
&&\bar{g}_8^{(k)}=\bar{g}_8(U^{(k)},(U^-)^{(k)},\psi^{(k)},DU^{(k)},D\psi^{(k)},D^2U^{(k)},D^2\psi^{(k)},D^3\psi^{(k)}).
\end{eqnarray*}
By Lemma \ref{lem55}, Lemma \ref{lem62} and \eqref{eq5144}, direct computation yields
\begin{eqnarray}\label{eq5148}
\norm{\hat{p}}_{C^{2,\alpha}(\bar{\Omega})}\le C\left(\norm{\hat{E}}_{C^{\alpha}(\bar{\Omega})}
+\norm{f^{(1)}-f^{(2)}}_{C^{\alpha}(\bar{\Omega})}
+\norm{\bar{g}_8^{(1)}-\bar{g}_8^{(2)}}_{C^\alpha(S^2)}\right)
\le C\varepsilon Q.%\left(\norm{\psi^{(1)}-\psi^{(2)}}_{C^{3,\alpha}(S^2)}+
%\norm{\check{U}^{(1)}-\check{U}^{(2)}}_{2}+\norm{i^*(\check{U}^{(1)}
%-\check{U}^{(2)})}_{C^{2,\alpha}(S^2)}\right).\nonumber
\end{eqnarray}

\noindent {\it Step 3.} From \eqref{eq593}, we see that
\begin{eqnarray*}
&&\begin{cases}\displaystyle
\hat{r}^p=-\frac{1}{4\pi\mu_6}\int_{{S}^2}\Big(\mu_5\,
i^*(\p_0\hat{p})+(\bar{g}_5^{(1)}-\bar{g}_5^{(2)})\Big)\, \vol,\\
\hat{\psi}^p=\frac{1}{\mu_2}\left(i^*\hat{p}-{\mu_8}
\int_{{S}^2}i^*(\p_0\hat{p})\,\vol+(\bar{g}_7^{(1)}-\bar{g}_7^{(2)})\right),\\
\hat{\psi}=\hat{\psi}^p+\hat{r}^p.
\end{cases}
\end{eqnarray*}
Here, for $k=1,2$,
\begin{eqnarray*}
&&\bar{g}_5^{(k)}=\bar{g}_5(U^{(k)},(U^-)^{(k)},\psi^{(k)},DU^{(k)},D\psi^{(k)}),\\
&&\bar{g}_7^{(k)}=\bar{g}_7(U^{(k)},(U^-)^{(k)},\psi^{(k)},D\psi^{(k)}).
\end{eqnarray*}
Then we have the following estimate via \eqref{eq5148}, Lemma \ref{lem62} and some straightforward computations:
\begin{eqnarray}\label{eq5152}
\norm{\hat{\psi}^p}_{C(S^2)}+|\hat{r}^p|&\le& C\left(\norm{\hat{p}}_{C^1(\bar{\Omega})}+\norm{\bar{g}_5^{(1)}-\bar{g}_5^{(2)}}_{C(\bar{\Omega})}+
\norm{\bar{g}_7^{(1)}-\bar{g}_7^{(2)}}_{C(\bar{\Omega})}\right)\nonumber\\
&\le&C\varepsilon Q.
\end{eqnarray}
By \eqref{eq595}, note that $\hat{\psi}^p$ also solves
\begin{eqnarray*}
-\Delta'\hat{\psi}^p+\mu_7\hat{\psi}^p=\mu_0\mu_6\hat{r}^p+\mu_0\mu_5\,i^*\p_0\hat{p}
+\bar{g}_6^{(1)}-\bar{g}_6^{(2)},
\end{eqnarray*}
with
\begin{eqnarray*}
\bar{g}_6^{(k)}=\bar{g}_6(U^{(k)},(U^-)^{(k)},\psi^{(k)},(DU^-)^{(k)},DU^{(k)},D\psi^{(k)},D^2\psi^{(k)}),\quad k=1,2,
\end{eqnarray*}
it follows that, from \eqref{eq5152} and \eqref{eq5148},
\begin{eqnarray*}
\norm{\hat{\psi}^p}_{C^{3,\alpha}(S^2)}&\le& C\left(\norm{\hat{\psi}^p}_{C(S^2)}+|\hat{r}^p|+
\norm{\hat{p}}_{C^{2,\alpha}(\bar{\Omega})}
+\norm{\bar{g}_6^{(1)}-\bar{g}_6^{(2)}}_{C^{1,\alpha}(S^2)}\right)\nonumber\\
&\le&C\varepsilon Q.
\end{eqnarray*}
This and \eqref{eq5152} imply that
\begin{eqnarray}\label{eq5156}
\norm{\hat{\psi}}_{C^{3,\alpha}(S^2)}\le C\varepsilon Q.
\end{eqnarray}

\noindent {\it Step 4.} From \eqref{eq597}, one has
\begin{eqnarray*}
\begin{cases}
D_{u^{(1)}}\widehat{A(s)}+D_{u^{(1)}-u^{(2)}}\widehat{A(s)}^{(2)}=0&\text{in}\quad \Omega,\\
i^*(\widehat{A(s)})=\mu_4\,\hat{\psi}+\bar{g}_4^{(1)}-\bar{g}_4^{(2)},
\end{cases}
\end{eqnarray*}
where
\begin{eqnarray*}
\bar{g}^{(k)}_4=\bar{g}_4(U^{(k)},(U^-)^{(k)},\psi^{(k)},D\psi^{(k)}).\quad k=1,2.
\end{eqnarray*}
By \eqref{eq5156} and Lemma \ref{lem62}, we have
\begin{eqnarray}\label{5159}
\norm{i^*(\widehat{A(s)})}_{C^{2,\alpha}(S^2)}&\le& C\left(\norm{\hat{\psi}}_{C^{2,\alpha}(S^2)}+\norm{\bar{g}_4^{(1)}-\bar{g}_4^{(2)}}_{C^{2,\alpha}(S^2)}\right)\nonumber\\ &\le&C\varepsilon Q.
\end{eqnarray}
It is also easy to show that
\begin{eqnarray}\label{5160}
\norm{\widehat{A(s)}}_{C^{1,\alpha}(\bar{\Omega})}\le \norm{i^*(\widehat{A(s)})}_{C^{1,\alpha}(S^2)}+\norm{u^{(1)}-u^{(2)}}_{C^{1,\alpha}(\bar{\Omega})}
\norm{\widehat{A(s)}^{(2)}}_{C^{2,\alpha}(\bar{\Omega})}
\le C\varepsilon Q.
\end{eqnarray}

\noindent {\it Step 5.} Next we turn to \eqref{eq5100} to find that the difference of tangential velocity field on $\Omega_0$ solves
\begin{eqnarray*}
\begin{cases}
\dd(\bar{\hat{u}}_0')=-\frac{1}{\mu_0}\dd (\bar{g}_0^{(1)}-\bar{g}_0^{(2)}),\\
\dd^*(\bar{\hat{u}}_0')=\mu_5\,
i^*(\p_0\hat{p})+\mu_6\,
\hat{\psi}+(\bar{g}_5^{(1)}-\bar{g}_5^{(2)})
\end{cases}\qquad \text{on}\  S^2,
\end{eqnarray*}
where, for $k=1,2$,
\begin{eqnarray*}
\bar{g}_0^{(k)}&=&\bar{g}_0(U^{(k)},(U^-)^{(k)}, \psi^{(k)}, D\psi^{(k)}),\\
\bar{g}_5^{(k)}&=&\bar{g}_5(U^{(k)},(U^-)^{(k)},\psi^{(k)},DU^{(k)},D\psi^{(k)}).
\end{eqnarray*}
By Theorem \ref{them62}, we have the estimate
\begin{eqnarray}\label{5164}
\norm{{\hat{u}}_0'}_{C^{2,\alpha}(S^2)}&\le& C\left(\norm{\bar{g}_0^{(1)}-\bar{g}_0^{(2)}}_{C^{2,\alpha}(S^2)}+\norm{\hat{p}}_{C^{2,\alpha}(\bar{\Omega})}
+\norm{\hat{\psi}}_{C^{1,\alpha}(S^2)}+\norm{\bar{g}_5^{(1)}-\bar{g}_5^{(2)}}_{C^{1,\alpha}(S^2)}\right)\nonumber\\
&\le&C\varepsilon Q.
\end{eqnarray}

\noindent {\it Step 6.} From \eqref{eq5103}, $\hat{u}^1$, the first component of the difference of the tangential velocity in $\Omega$, solves
\begin{eqnarray*}
\begin{cases}
(u^j)^{(1)}\p_j\hat{u}^1+2\frac{u_b^0(y^0)}{y^0}\hat{u}^1+\Big((u^j)^{(1)}
-(u^j)^{(2)}\Big)\p_j(\hat{u}^{(2)})^1\\
\qquad=-\frac{1}{\rho^{(1)} (y^0)^2}\p_1 \hat{p}+\frac{1}{(y^0)^2}\frac{\rho^{(1)}-\rho^{(2)}}{\rho^{(1)}\rho^{(2)}}\p_1 \hat{p}^{(2)} +(\bar{f}_1^{(1)}-\bar{f}_1^{(2)}) &\text{in}\quad \Omega,\\
\hat{u}^1=\hat{u}_0^1&\text{on}\quad \Omega_0.
\end{cases}
\end{eqnarray*}
where
\begin{eqnarray*}
\bar{f}_1^{(k)}=\bar{f}_1(U^{(k)}, \psi^{(k)}, D\psi^{(k)}), \quad k=1, 2.
\end{eqnarray*}
So there holds
\begin{eqnarray}\label{5167}
\norm{\hat{u}^1}_{C^{1,\alpha}(\bar{\Omega})}&\le& C\left(\norm{\hat{p}}_{C^{2,\alpha}(\bar{\Omega})}+
\norm{\bar{f}_1^{(1)}-\bar{f}_1^{(2)}}_{C^{1,\alpha}(\bar{\Omega})}
+C_*\varepsilon\norm{U^{(1)}-U^{(2)}}_{C^{1,\alpha}(\bar{\Omega})}\right)\nonumber\\
&\le&C\varepsilon Q.
\end{eqnarray}
There is a similar estimate  for the second component, namely $\hat{u}^2$.

\medskip \noindent {\it Conclusion.} Now summing up the inequalities
\eqref{eq5143}\eqref{eq5144}\eqref{eq5148}\eqref{eq5156}
\eqref{5159}\eqref{5160}\eqref{5164} and \eqref{5167}, we get
\begin{eqnarray*}
\norm{\hat{\psi}}_{C^{3,\alpha}(S^2)}+\norm{\hat{U}}_2
+\norm{i^*\hat{U}}_{C^{2,\alpha}(S^2)}
\le C'\varepsilon Q,
\end{eqnarray*}
which implies \eqref{eq5106} if $\varepsilon\in(0, \varepsilon_0)$ and  $C'\varepsilon_0<1/2$.
Finally, by Banach fixed point theorem, we infer Problem (T4), hence Problem (T), has one and only one solution in $\mathcal{K}_{C_*\varepsilon}\times O_{C_*\varepsilon}.$ This finishes the proof of Theorem  \ref{thm501}.

\section{Appendix}\label{sec6}
We provide here some results used in this paper, together with some details on the estimates of higher-order terms.

\subsection{Solvability and estimate of transport equations}
We consider the following Cauchy problem of a transport equation for the unknown $E$ in $\mathcal{M}=(r_0, r_1)\times M$, where $M$ is a smooth closed surface:
\begin{eqnarray}\label{eq61}
D_uE+a E=f\quad \text{in}\ \mathcal{M},\qquad E=E_0\quad \text{on} \ M_0=\{r_0\}\times M.
\end{eqnarray}
The main result is:
\begin{theorem}\label{thm61}
For a fixed number $\alpha\in(0,1)$ and $k\in\mathbb{N}$, suppose that the vector field $u=u^0\p_0+u'$ and the functions $a, f$ belong to the H\"{o}lder space  $C^{k,\alpha}(\overline{\mathcal{M}})$, and $E_0\in C^{k,\alpha}(M)$, and furthermore, $u^0$ has a positive lower bound $\delta$ in $\mathcal{M}$. Then there is uniquely one solution $E$ to problem \eqref{eq61}, and there is a positive constant $C=C({M}, |r_1-r_0|, \delta, \norm{u}_{C^{k,\alpha}}, \norm{a}_{C^{k,\alpha}})$ so that
\begin{eqnarray}\label{eq62}
\norm{E}_{C^{k,\alpha}(\overline{\mathcal{M}})}\le C\Big(\norm{E_0}_{C^{k,\alpha}(M)}+\norm{f}_{C^{k,\alpha}(\overline{\mathcal{M}})}\Big).
\end{eqnarray}
\end{theorem}

\begin{proof}
1. Since $u^0\ge\delta>0$ in $\mathcal{M}$, we may rewrite problem \eqref{eq61} as
\begin{eqnarray*}
\p_0E+\frac{u'}{u^0}E+\frac{a}{u^0} E=\frac{f}{u^0}\quad \text{in}\ \mathcal{M},\qquad E=E_0\quad \text{on} \ M_0.
\end{eqnarray*}
Hence thee is no loss of generality by assuming that $u^0\equiv1$ in the sequel.

2. For $u'=u^\alpha\p_\alpha$, consider the Cauchy problem of ordinary differential equations:
$$\frac{\dd {x'}}{\dd t}={u'}(t, {x'}),\quad {x'}(r_0)=\bar{x}\in M.$$
So ${x}'(t)$ is the integral curve of the (time-dependent) vector
field $u'$ on $M$, which passes through the point $\bar{x}$ on $M$ when $t=r_0$.
Since $k\ge 1,$ and $u'$ is bounded, by theorem of ordinary
differential equations, there is one and only one such solution for
$t\in[r_0, r_1]$, and the solution depends on $\bar{x}$ with the same
regularity as $u'$ depending on $\bar{x}$. (See, for example,
\cite[Section 13]{walter}.) We may also write the solution as
${x'}=\varphi(t,\bar{x})=\varphi_t(\bar{x}).$
The transform $\varphi_t: M\to M, \bar{x}\mapsto {x'}$ and its inverse both belong to $C^{k,\alpha}(M)$.

3. For the transport equation $\p_0E+{u'} E+a(x) E=f(x)$ of unknown $E(x)$ (we write here that $x=(t, {x'})$; that is, $x^0=t$),
%\begin{eqnarray*}
%\p_0E+{u'} E+a(x) E=f(x),
%\end{eqnarray*}
since $E(x)=E(t,{x'})=E(t, \varphi_t(\bar{x}))$, we have
$\frac{\dd E}{\dd t}+a(t,\varphi_t(\bar{x}))E=f(t, \varphi_t(\bar{x})).$
The solution is given by
\begin{eqnarray}\label{eq63}
E(x)&=&E(t, \varphi_t(\bar{x}))
=\exp\Big(-\int_{r_0}^ta(s, \varphi_s(\bar{x}))\,\dd s\Big)\nonumber\\
&&\times\left(E_0(\bar{x})+\int_{r_0}^t\exp\Big(\int_{r_0}^sa(\tau, \varphi_\tau(\bar{x}))\,\dd\,\tau\Big)f(s, \varphi_s(\bar{x}))\,\dd s\right).
\end{eqnarray}
This proves the existence.

4. Since products and compositions of $C^{k,\alpha}$ functions are still in $C^{k,\alpha}$, and note that $\bar{x}=(\varphi_t)^{-1}(x')\in C^{k,\alpha}(M)$, $(\varphi_t)^{-1}$ is $C^{k+1,\alpha}$ with respect to $t$,  we deduce that the right-hand side of \eqref{eq63} belongs to $C^{k,\alpha}(\overline{\mathcal{M}})$, and the estimate \eqref{eq62} follows.
\end{proof}

\subsection{Solvability and estimate of div-curl system on sphere $S^2$}
The following result is a refined version of that appeared in \cite{CY2013}, so we omit the proof here.
\begin{theorem}\label{them62}
For $\alpha\in (0,1)$ and $k\in\mathbb{N}$,  there exists a unique 1-form $\omega$ in $C^{k+1}(S^2)$ that solves
\begin{eqnarray}\label{eq571}
\dd\omega=\chi,\qquad \dd^*\omega=\psi,
\end{eqnarray}
if $\chi$ is a 2-form  in $C^{k,\alpha}({S}^2)$ and $\psi$ is a function in $C^{k,\alpha}({S}^2)$,  and there hold
\begin{eqnarray}\label{eq572}
\int_{{S}^2}\chi=0,\qquad \int_{{S}^2}\psi\,\vol=0.
\end{eqnarray}

Furthermore, there is constant $C>0$ so that
\begin{eqnarray}\label{eq573}
\norm{\omega}_{C^{k+1,\alpha}(S^2)}\le C\Big(\norm{\chi}_{C^{k,\alpha}({S}^2)}+
\norm{\psi}_{C^{k,\alpha}({S}^2)}\Big).
\end{eqnarray}
\end{theorem}

\subsection{Estimates of higher-order terms}

We list below the estimates of some higher-order terms, which were used in Section \ref{sec531} and Section \ref{sec532} in studying the iteration mapping $\mathcal{T}$.

\begin{lemma}\label{lem61}
There is a constant $C$ depending only on the background solution $U_b$ so that if $\psi\in\mathcal{K}_\sigma, \check{U}\in O_\delta$, and $U^-$ satisfying \eqref{55015}, then one has the following inequalities:
\begin{eqnarray}
\norm{\bar{g}_4}_{C^{3,\alpha}(S^2)}&\le& C\Big(\delta^2+\sigma^2+\varepsilon^2+\varepsilon\Big),\label{616}\\
\norm{f}_{C^{1,\alpha}(\bar{\Omega})}&\le& C(\delta^2+\sigma^2+\varepsilon^2+\varepsilon), \label{610}\\
\norm{\bar{f}_1}_{C^{2,\alpha}(\bar{\Omega})}&\le& C\Big(\delta^2+\sigma^2+\varepsilon^2+\varepsilon\Big),\label{611}\\
\norm{\bar{g}_0}_{C^{3,\alpha}(S^2)}&\le& C\Big(\delta^2+\sigma^2+\varepsilon^2+\varepsilon\Big),\label{617}\\
\norm{\bar{g}_5}_{C^{2,\alpha}(S^2)}&\le& C\Big(\delta^2+\sigma^2+\varepsilon^2+\varepsilon\Big),\label{618}\\
\norm{\bar{g}_6}_{C^{2,\alpha}(S^2)}&\le& C\Big(\delta^2+\sigma^2+\varepsilon^2+\varepsilon\Big),\label{614}\\
\norm{\bar{g}_7}_{C^0(S^2)}
&\le&C\Big(\delta^2+\sigma^2+\varepsilon^2+\varepsilon\Big),\label{613}\\
\norm{\bar{g}_8}_{C^{1,\alpha}(S^2)}&\le& C\Big(\delta^2+\varepsilon^2+\sigma^2+\varepsilon\Big).\label{615}
\end{eqnarray}
In the expressions of these higher order terms,  we have set $U=\check{U}+U_b^+(\frac{r_1-\psi}{r_1-r_b}(y^0-r_b)+\psi(y'), y').$
\end{lemma}

\begin{proof}
1. We firstly consider the terms $\bar{g}_k$ ({\it i.e.}, $g_k$)  defined through \eqref{eq512}--\eqref{eq515} for $k=1,2,3,4$. Checking (2.27) in \cite[p.2522]{CY2013} shows that $\bar{g}_k$ are of the form
\begin{eqnarray}\label{619}
\bar{g}_k&=&O(1)i^*(\check{U})(\psi-r_b)+O(1)i^*(\check{U})^2+O(1)i^*(\check{U})\psi^*(U^--U_b^-)\nonumber\\
&&+O(1)\psi^*(U^--U_b^-)+O(1)(D\psi) i^*(\check{U})+O(1)(\psi-r_b)^2.
\end{eqnarray}
It follows easily that
\begin{eqnarray}\label{620}
\norm{\bar{g}_k}_{C^{3,\alpha}(S^2)}\le C(\sigma\delta+\delta^2+\varepsilon\delta+\varepsilon+\delta\sigma+\sigma^2)\le
C\Big(\delta^2+\sigma^2+\varepsilon^2+\varepsilon\Big).
\end{eqnarray}
Thus we proved \eqref{616}.

2. The term $f$ is defined by \eqref{eq556}.  For $\tilde{F}_2$ (see \eqref{eq550}), we note that
\begin{eqnarray}\label{621}
A(s)(r_b, (\varphi_{y^0})^{-1}(y'))-A(s)(r_b, y')=O(1)(D(i^*\check{U}))\check{U},
\end{eqnarray}
hence according to the proof of Lemma \ref{lem53}, its $C^{1,\alpha}(\bar{\Omega})$ norm is controlled by $C\delta^2$.  Then using \eqref{619}, we find that
\begin{eqnarray}\label{622}
\norm{\tilde{F}_2}_{C^{1,\alpha}(\bar{\Omega})}\le C(\delta^2+\sigma^2+\varepsilon^2+\varepsilon).
\end{eqnarray}
For the term $\tilde{F}_1$ (see \eqref{eq554add}), it can be written as
\begin{eqnarray}\label{623}
\tilde{F}_1&=&O(1)(\psi-r_b)(D^2\check{p}+D\check{p}+\check{U})
+O(1)(D^2\psi)(D\check{p})\nonumber\\
&&+O(1)(D\psi)(D^2\check{p})+O(1)(D\psi)(D\check{p}).
\end{eqnarray}
Hence
\begin{eqnarray}\label{624}
\norm{\tilde{F}_1}_{C^{1,\alpha}(\bar{\Omega})}\le C\delta\sigma\le C(\delta^2+\sigma^2).
\end{eqnarray}

Next we consider the term $F$ (see \eqref{eq441}). For $F_2$ (see \eqref{eq447add}), taking into account of the transform $\Psi$ (see \eqref{259}), it is of the form
\begin{eqnarray}\label{625}
F_2=O(1)(\check{U})^2+O(1)(D\check{U})\check{U}
+O(1)(D^2\check{p})\check{U}+O(1)(D\check{U})^2,
\end{eqnarray}
where the quantities $O(1)$ may contain $\psi-r_b$, $D\psi$ or $D^2\psi$. Therefore one has
\begin{eqnarray}\label{626}
\norm{{F}_2}_{C^{1,\alpha}(\bar{\Omega})}\le C\delta^2.
\end{eqnarray}
For $H_3$ (see \eqref{eq433}), it can be written as
%\begin{eqnarray}\label{627}
$H_3=O_1\check{U}^2,$
%\end{eqnarray}
with $O_1$ a bounded quantity containing $D^2p, (\psi-r_b), D\psi$ and $Dp$.
Then
%\begin{eqnarray}\label{628}
$F_1=O_1(\check{U}^2+\check{U}(D\check{U})+(D\check{U})^2)$
%\end{eqnarray}
and
\begin{eqnarray}\label{629}
\norm{{F}_1}_{C^{1,\alpha}(\bar{\Omega})}\le C\delta^2.
\end{eqnarray}
Summing up the above estimates \eqref{622}\eqref{624}\eqref{626}\eqref{629}, we get \eqref{610}.

3. The term $\bar{f}_1$ is given by \eqref{eq572add2}, where $h_1$ is defined by \eqref{eq465}. We may write
\begin{eqnarray*}
\bar{f}_1=O(1)\check{U}^2+O(1)\check{U}(\psi-r_b)+O(1)(D\check{p})(\psi-r_b)
+O(1)(D\check{p})(D\psi),
\end{eqnarray*}
where $O(1)$ may contain $D\psi$ and $\psi-r_b$. Therefore it is straightforward to check that
\begin{eqnarray*}
\norm{\bar{f}_1}_{C^{2,\alpha}(\bar{\Omega})}\le C(\delta^2+\sigma^2),
\end{eqnarray*}
which implies \eqref{611}.

4. The terms $\bar{g}_k$ for $k=0, 5,6,7,8$ were defined in \eqref{eqbarg1} and  \eqref{eqbarg2}. We firstly consider $g_0$ (see \eqref{eq59}). It is of the form
\begin{eqnarray*}
\bar{g}_0=O(1)(i^*\check{U})^2+O(1)(i^*\check{U})\psi^*(U^--U_b^-)+O(1)\psi^*(U^--U_b^-),
\end{eqnarray*}
where $O(1)$ may contain $D\psi$. Therefore
\begin{eqnarray*}
\norm{\bar{g}_0}_{C^{3,\alpha}(S^2)}\le C(\delta^2+\delta\varepsilon+\varepsilon).
\end{eqnarray*}
Thus we proved \eqref{617}.

5. Next we consider $\bar{g}_5$, where $g_5$ was given by \eqref{eqbarg5}, and $\tilde{G}$ was defined by \eqref{eq543add3}. We see
\begin{eqnarray*}
\bar{g}_5&=&O(1)(i^*\check{U})(D(i^*\check{U}))+O(1)i^*\check{U}(\psi-r_b)+O(1)D\psi(i^*\check{U})^2\nonumber\\
&&+O(1)D\psi(i^*\check{U})+O(1)D\psi(i^*\check{U})(D(i^*\check{U}))+O(1)(i^*\check{U})^3\nonumber\\
&&+O(1)(i^*\check{U})^2(D(i^*\check{U}))
+O(1)(i^*\check{U})^2
+O(1)D\psi(i^*\check{U})^2(D(i^*\check{U}))\nonumber\\
&&+O(1)(\psi-r_b)^2+O(1)\psi^*(U^--U_b^-)+O(1)g_2+O(1)g_4+O(1)(\psi-r_b)D\check{p}.
\end{eqnarray*}
Hence we have
\begin{eqnarray*}
\norm{\bar{g}_5}_{C^{2,\alpha}(S^2)}\le C(\delta^2+\sigma^2+\varepsilon^2+\varepsilon)
\end{eqnarray*}
as desired.

6. Since $\bar{g}_6=\mu_0\bar{g}_5+\dd^*\bar{g}_0$, the estimate \eqref{614} follows directly from \eqref{617}\eqref{618}.

7. We note that $g_7=\frac{\mu_2}{4\pi\mu_6}\int_{S^2}g_5\,\vol-g_2$, so
\begin{eqnarray}
\bar{g}_7=O(1)\bar{g}_5+O(1)\bar{g}_2+O(1)(\psi-r_b)(Di^*\check{p})
\end{eqnarray}
and from \eqref{618}, \eqref{620}, we get  \eqref{613}.

8. Finally we consider $\bar{g}_8$, where $g_8=-\Delta'g_2+\mu_7g_2+\mu_2g_6$. So \eqref{615} follows from \eqref{614} and \eqref{620}.
\end{proof}

\begin{lemma}\label{lem62}
Under the assumptions stated in Section \ref{sec532}, there is a constant $C$ depending only on the background solution so that
\begin{eqnarray}
\norm{\bar{g}_4^{(1)}-\bar{g}_4^{(2)}}_{C^{2,\alpha}(S^2)}&\le& C\varepsilon Q,\label{636}\\
\norm{f^{(1)}-f^{(2)}}_{C^{\alpha}(\bar{\Omega})}&\le& C\varepsilon Q,\label{637}\\
\norm{\bar{f}_1^{(1)}-\bar{f}_1^{(2)}}_{C^{1,\alpha}(\bar{\Omega})}&\le& C\varepsilon Q,\label{638}\\
\norm{\bar{g}_0^{(1)}-\bar{g}_0^{(2)}}_{C^{2,\alpha}(S^2)}&\le& C\varepsilon Q,\label{639}\\
\norm{\bar{g}_5^{(1)}-\bar{g}_5^{(2)}}_{C^{1,\alpha}(S^2)}&\le& C\varepsilon Q,\label{640}\\
\norm{\bar{g}_6^{(1)}-\bar{g}_6^{(2)}}_{C^{1,\alpha}(S^2)}&\le& C\varepsilon Q,\label{641}\\
\norm{\bar{g}_7^{(1)}-\bar{g}_7^{(2)}}_{C(S^2)}&\le& C\varepsilon Q,\label{642}\\
\norm{\bar{g}_8^{(1)}-\bar{g}_8^{(2)}}_{C^\alpha(S^2)}&\le& C\varepsilon Q.\label{643}
\end{eqnarray}
Recall that $Q$ is defined by \eqref{eq5106}.
\end{lemma}

\begin{proof}
To prove \eqref{636}, we use \eqref{619}. Returning to (2.27) in \cite{CY2013}, some tedious computations show that  for each coefficient $O(1)$, it holds
%\begin{eqnarray}
$\norm{O(1)^{(1)}-O(1)^{(2)}}_{C^{2,\alpha}(S^2)}\le CQ.$
%\end{eqnarray}
Hence, as an example,  since
\begin{eqnarray*}
&&O(1)^{(1)}(\psi^{(1)})^*(U^--U_b^-)-O(1)^{(2)}(\psi^{(2)})^*(U^--U_b^-)\nonumber\\
&=&O(1)^{(1)}\Big((\psi^{(1)})^*(U^--U_b^-)-(\psi^{(2)})^*(U^--U_b^-)\Big)\nonumber\\
&&+(O(1)^{(1)}-O(1)^{(2)})(\psi^{(2)})^*(U^--U_b^-),
\end{eqnarray*}
by \eqref{eq5141} we have
\begin{eqnarray*}
\norm{O(1)^{(1)}(\psi^{(1)})^*(U^--U_b^-)-O(1)^{(2)}(\psi^{(2)})^*(U^--U_b^-)}_{C^{2,\alpha}(S^2)}\le C\varepsilon Q.
\end{eqnarray*}
The other example is to consider
\begin{eqnarray*}
&&O(1)^{(1)}i^*(\check{U}^{(1)})(\psi^{(1)}-r_b)-O(1)^{(2)}i^*(\check{U}^{(2)})(\psi^{(2)}-r_b)\nonumber\\
&=&O(1)^{(1)}i^*(\check{U}^{(1)})(\psi^{(1)}-\psi^{(2)})
+(\psi^{(2)}-r_b)O(1)^{(1)}i^*(\check{U}^{(1)}-\check{U}^{(2)})\nonumber\\
&&+(\psi^{(2)}-r_b)(O(1)^{(1)}-O(1)^{(2)})i^*\check{U}^{(2)},
\end{eqnarray*}
and we have straightforwardly that
\begin{eqnarray*}
\norm{O(1)^{(1)}i^*(\check{U}^{(1)})(\psi^{(1)}-r_b)-O(1)^{(2)}i^*(\check{U}^{(2)})(\psi^{(2)}-r_b)}_{C^{2,\alpha}(S^2)}\le C\varepsilon Q.
\end{eqnarray*}
Since all other terms are of similar form, they could be estimated similarly and hence we omit the details.
\end{proof}

%\bigskip
{\bf Acknowledgments.}  Li Liu was supported by National
Natural Science Foundation of China (NNSFC) under Grants No.11101264, No.11371250. % and a Special Research Fund for Selecting Excellent Young Teachers of the Universities in Shanghai sponsored by the Shanghai Municipal Education Commission.
Gang Xu was supported by NNSFC under Grants No.11101190, No.11371189 and No.11271164. Hairong Yuan (the corresponding author)  was supported by  NNSFC under Grants No.11371141. Yuan would also like to thank the Wuhan Institute of Physics and Mathematics, Chinese Academy of Science for the hospitality, where the manuscript was finished.

%%%----------------------------------------------------------------------------

\end{document}